\documentclass[11pt,reqno]{amsart} %
\usepackage{graphicx}
\usepackage{latexsym, amsmath,amssymb}
\usepackage{graphics}
\usepackage[T1]{fontenc}
\usepackage[latin1]{inputenc}
\usepackage{graphics}
\usepackage{graphicx}
\usepackage{sidecap}

\usepackage{epsfig}
\usepackage{latexsym}
\usepackage{eurosym}
\usepackage{amscd}
\usepackage{shadow}
\usepackage{a4wide}
\usepackage{caption}

\usepackage{subfigure}
\usepackage{wrapfig}

\usepackage{amsmath,amsfonts,amsthm,amssymb}

\usepackage{mathtools}
\mathtoolsset{showonlyrefs}

\newcommand{\Ai}{\mathrm{Ai}}

\usepackage{hyperref}
\hypersetup{colorlinks=true}

\usepackage[margin=2cm,bottom=3cm,top=2cm,footskip=1.5cm]{geometry}

\newcommand\N{{\mathbb N}} \newcommand\R{{\mathbb R}}
\newcommand\Z{{\mathbb Z}} \newtheorem{thm}{Theorem}

\newtheorem{rmq}{Remark} 
\newtheorem{lemma}{Lemma} %
\newtheorem{cor}{Corollary} \newtheorem{prop}{Proposition}

\begin{document}
 \title[Schr\"odinger equation in a strictly convex domain]{Dispersive estimates for the Schr\"odinger equation in a strictly convex domain and applications}

\author{Oana Ivanovici}
\address{Sorbonne Université, CNRS, Laboratoire Jacques-Louis Lions, LJLL, F-75005 Paris, France} \email{oana.ivanovici@math.cnrs.fr}
    \thanks{The author was
     supported by ERC grant ANADEL 757 996.
    } 
\begin{abstract}
We consider an anisotropic model case for a strictly convex domain
$\Omega\subset\mathbb{R}^d$ of dimension $d\geq 2$ with smooth
boundary $\partial\Omega\neq\emptyset$ and we describe dispersion for
the semi-classical Schrödinger equation with Dirichlet boundary condition. More specifically, we obtain the following fixed time decay rate for the linear semi-classical flow : a loss of $(\frac ht)^{1/4}$ occurs with respect to the boundary less case due to repeated swallowtail type singularities, and is proven optimal.  Corresponding Strichartz estimates allow to solve the cubic nonlinear Sch\"odinger equation on such a 3D model convex domain, hence matching known results on generic compact boundaryless manifolds.
\end{abstract}
\maketitle
\section{Introduction}
Let us consider the Schr\"odinger equation on a manifold $(\Omega,g)$, with a strictly convex boundary $\partial
\Omega$ (a precise definition of strict convexity will be provided later): 
\begin{equation}\label{schrod}
-i\partial_t v+\Delta_g v=\kappa |v|^{2} v, \quad v|_{t=0}=v_0,\quad  v|_{\mathbb{R}\times\partial\Omega}=0,
\end{equation}
where $\Delta_g$ denotes the Laplace operator with Dirichlet boundary condition, and $\kappa=0$ (linear equation) or $\kappa=\pm 1$ (defocusing or focusing nonlinear cubic equation).

For nonlinear partial differential equations on manifolds, understanding the linear flow is a pre-requisite to studying nonlinear problems: addressing the Cauchy problem for nonlinear wave equations starts with perturbative techniques and faces the difficulty of controlling solutions to the linear equation in term of the size of the initial data. Especially at low regularities, mixed norms of Strichartz type ($L^q _t L^r _x$) are particularly useful. For the linear Schr\"odinger flow $e^{-it\Delta_{g}}v_0$ (\eqref{schrod} with $\kappa=1$), local Strichartz estimates (in their most general form) read
\begin{equation}\label{schSE}
\|e^{-it\Delta_g}v_0\|_{L^{q}(0,T)L^r(\Omega)}\leq C_T\|v_0\|_{H^{\sigma}(\Omega)},
\end{equation}
where $2\leq q,r\leq \infty$ satisfy the Schr\"odinger admissibility condition, $\frac{2}{q}+\frac{d}{r}\leq \frac{d}{2}$, $(q,r,d)\neq(2,\infty,2)$ and
$\frac{2}{q}+\frac{d}{r}\geq \frac{d}{2}-\sigma$ (scale-invariant when equality; otherwise, {\it loss of derivatives} in the estimate \eqref{schSE} as it deviates from the optimal regularity predicted by scale invariance.) In Euclidean space $\mathbb{R}^d$ with $g=(\delta_{ij})$, \eqref{schSE} holds with $\sigma=0$ and extends globally in time, $T=+\infty$.

The canonical path leading to such Strichartz estimates is to obtain a
stronger, fixed time, dispersion estimate, which is then combined with
energy conservation, interpolation and a duality argument to obtain
\eqref{schSE}. Dispersion for the linear Schr\"odinger flow in $\mathbb{R}^{d}$ reads as:
\begin{equation}\label{dispschrodrd}
\|e^{\pm it\Delta_{\mathbb{R}^d}}\|_{L^1(\mathbb{R}^d)\rightarrow L^{\infty}(\mathbb{R}^d)}\leq C(d)t^{-d/2}, \text{ for all } t\neq 0.
\end{equation}
Indeed, \eqref{dispschrodrd} and the unitary property of the flow on $L^2(\mathbb{R}^d)$ are all that is required to obtain all known Strichartz estimates ; the endpoint cases are more delicate (see \cite{KeTa98}, \cite{Fo05}, \cite{Vi07}.)

On any boundaryless Riemann manifold $(\Omega,g)$ one may follow the same path, replacing the exact formula by a parametrix, constructed locally within a small ball, thanks to finite speed of propagation for waves or in semi-classical time for Schr\"odinger - short time, wavelength sized intervals (e.g. their size is the inverse of the frequency), allowing for  almost finite speed of propagation. by time rescaling, dispersion for the semi-classical Schr\"odinger equation in the Euclidean space reads, with $\psi \in C_{0}^{\infty}$ being a smooth cut-off to localize frequencies,
\begin{equation}\label{sclSchrod}
\sup\Big| \psi(hD_{t})e^{\pm i th \Delta_{\mathbb{R}^d}}\Big|
\leq \frac{C(d)}{h^{d}} \min (1, ({h\over t})^{\frac{d}{2}}) \text{ for all } 0<|t|\lesssim 1\,,
\end{equation}
While for $\Omega =\mathbb{R}^d$, dispersive properties of \eqref{schrod} are well understood, studying dispersive equations of Schr\"odinger type on manifolds (curved geometry, variable metric) started with Bourgain's work on KdV and Schr\"odinger on the torus, and then expanded in different directions, all of them with low regularity requirements (e.g. Staffilani-Tataru \cite{stta02}, Burq-G\'erard-Tzvetkov \cite{bgt04}, \cite{BuGeTz04} for Schr\"odinger, Smith \cite{sm98}, \cite{smit06}, Tataru \cite{tat02}, Bahouri-Chemin \cite{bach99}, \cite{bach99bis}, Klainerman-Rodnianski \cite{klrod03} and Smith-Tataru \cite{smtat05}, \cite{smtat02} for wave equations). In \cite{bgt04}, these linear estimates were used, together with a classical argument due to Yudovitch, to obtain global well-posedness for the defocusing cubic NLS on a generic 3D compact manifold without boundary. We aim at matching this result in our context, with a model convex boundary.

For compact manifolds (even without boundary) one cannot expect linear estimates to behave like in the Euclidean case: eventually a loss will occur, due to the volume being finite. No long time dispersion of wave packets may occur as they have nowhere to disperse. Long time estimates for the wave equation are unknown, while in the case of the Schr\"odinger equation, the infinite speed of propagation immediately produces unavoidable losses of derivatives in dispersive estimates. Informally, this may be related to the existence of eigenfunctions, but the complete understanding of the loss mechanism is still a delicate issue, even on the torus. On domains with boundaries, there are additional difficulties related to reflected waves. Partial progress was made in \cite{ant} and then in \cite{bss08}, \cite{bss12}, following the general strategy of the low regularity, boundary less case: reflect the metric across the boundary and deal with a boundaryless domain whose metric is only Lipschitz at the interface. Such results hold for any (smooth) boundary, regardless of its shape: however, they apply to 3D NLS only for nonlinearities that are weaker than cubic: \cite{bss12} obtains global well-posedness for smooth nonlinearities $F(v)$ with growth at most $|v|^{2/5} v$.

During the last decade, additional progress was made for the wave equation on domains with convex boundary. Our first result \cite{ilp12}, which deals with the model case of a strictly convex domain, highlights a loss in dispersion for the solution to the linear wave equation that we informally relate to caustics, generated in arbitrarily small time near the boundary. Such caustics appear when optical rays are no longer diverging from each other in the normal direction, where less dispersion occurs as compared to the $\mathbb{R}^d$ case. Our so-called Friedlander model domain is the half-space, for $d\geq 2$,
$\Omega_d=\{(x,y)|x>0, y\in\mathbb{R}^{d-1}\}$ with the metric $g_F$ inherited from the following Laplace operator
\begin{equation}
  \label{eq:LapM}
\Delta_F=\partial^2_x+\sum_{j}\partial_{y_j}^2+ x\sum_{j,k}q_{j,k}\partial_{y_j}\partial_{y_k}\,,
\end{equation}
where $q_{j,k} $ are constants and $q(\theta)=\sum_{j,k}q_{j,k}\theta_j \theta_k$ is a positive definite
quadratic form. Note that $q$ is not, in general, invariant by rotations and we cannot reduce to the radial case in $y$, unlike \cite{ilp12}, where $q(\theta)=|\theta|^{2}$. One may see $\Delta_{F}$ as the Laplace operator in geodesic normal coordinates near the boundary, but where one would freeze all coefficients $q_{j,k}(x,y)$ to their value on the boundary. Strict convexity of $\Omega_d$ with the metric inherited from $\Delta_F$ is equivalent to ellipticity of $\sum_{j,k}q_{j,k}\partial_{y_j}\partial_{y_k}$. When $q_{j,j}=1$ and $q_{j\neq k}=0$ (i.e. when $q(\theta)=|\theta|^2$)
the domain $(\Omega_d,g_F)$ is, indeed, a first order approximation of the unit disk in polar coordinates $(r,\theta)$: set $r=1-\frac x2$, $\theta=y$. Let $h,a\in (0,1)$: if $u_a(t,x,y)=\cos
(t\sqrt{|\Delta_F|})(\delta_{x=a,y=0})$ denotes the linear wave flow on $(\Omega,g)=(\Omega_d, g_F)$ with data $\delta_{x=a,y=0}$ and Dirichlet boundary condition, then, for $|t|\geq h$, \cite{ilp12} proves
\begin{equation}\label{dispcoc} 
\|\psi(hD_t)u_a(t,.)\|_{L^{\infty}}\leq C(d)
h^{-d}\min\Big\{1,(h/t)^{\frac{d-2}{2}}\Big( (\frac ht)^{1/2}+(\frac ht)^{1/3}+a^{1/4}(\frac ht)^{1/4}\Big)\Big\}.
\end{equation} 
Moreover, \eqref{dispcoc} is sharp, as there exists a sequence $(t_n)_n$ such that equality holds. This optimal $\frac 14$ loss in the $\frac ht$ exponent is unavoidable for small $a$ and is due to swallowtail type singularities in the wave front set of $u_a$. This first result opened several directions, from the generic convex case \cite{ILLP} to understanding more complicated boundary shapes \cite{LenMeas}.

In the present work, we address the same set of issues for the Schr\"odinger equation, where
parallel developments were expected, at least in the so called semiclassical setting (recall that ``semiclassical'' means, in our setting, dealing with time intervals whose size is comparable to the wavelength $h$, which reduces to almost finite speed of propagation.) In the non-trapping case, results for the classical Schr\"odinger equation may follow when combined with smoothing effects, but we will not address this situation (we model the interior of a convex.) In the case of a convex boundary, even the wavelength sized time behavior is complicated due to the existence of gliding rays. Let $h\in (0,1)$ and consider the semiclassical Schrödinger equation inside the Friedlander domain $(\Omega_d, g_F)$, with $\Delta_F$ given in \eqref{eq:LapM} and Dirichlet boundary condition
\begin{equation}\label{scl}
ih\partial_t v_{h}-h^2\Delta_F v_{h}=0, \quad v_{h|t=0}=v_0,\quad  v_{h|\partial\Omega_d}=0\,.
\end{equation}
With this rescaling, we are dealing with $O(1)$-bounded rather than $h-$sized intervals.
\begin{thm}\label{thmdispSchrodinger}
Let $\psi\in C^{\infty}_0([\frac 12, \frac 32])$, $0\leq \psi \leq 1$. There exists $C(d)>0$, $T_0>0$ and $a_0\leq 1$ such that, for all $a\in(0, a_0]$, $h\in (0,1)$, $|t|\in (h,T_0]$, $v_{h}(t,\cdot)$ solution to \eqref{scl} with data $v_{h,0}(x,y)=\psi(hD_y)\delta_{x=a,y=0}$,
  \begin{equation}
    \label{eq:1}
 \|\psi(hD_{t})v_{h}(t,x,y)\|_{L^{\infty}(\Omega_d)}
    \leq \frac{C(d)}{h^{d}}\Big(\frac{h}{|t|}\Big)^{\frac{(d-1)}{2}+\frac 1 4}\,.
  \end{equation}
Moreover, for every $|t|\in (\sqrt{a},T_0]$ and every $|t|h^{1/3}\ll a\leq a_0$, the bound is saturated:
\begin{equation}\label{optloss}
\|\psi(hD_{t})v_{h}(t,x,y)\|_{L^{\infty}(\Omega_d)}
 \sim  \frac{a^{\frac 1 4}}{h^{d}}(\frac{h}{|t|})^{\frac{(d-1)}{2}+\frac 1 4}\,.
\end{equation}
\end{thm}
Important additional difficulties appear as compared to the wave equation: for not too small $a$, the Green function for the wave flow can be explicitly expressed as a sum of "time-almost-orthogonal" waves, which are essentially supported between a finite number of consecutive reflections; in \cite{ilp12}, we were therefore reduced to obtaining good dispersion bounds for a {\it finite} sum of waves well localized in both time and tangential variables. We will establish a suitable subordination formula that yields a similar representation of the Schr\"odinger flow as a sum of wave packets; nonetheless, 
at a given time $t$, {\it all} waves in this sum provide important contributions, because they travel with different speeds. To sum up all these contributions we need sharp bounds for each of them, similar to those obtained in \cite{ILP3} for waves. For very small $a$, writing a parametrix as a sum over reflections no longer helps. Using the spectral decomposition of the data in terms of eigenfunctions of the Laplace operator allows to obtain a parametrix as a sum over the zeros of the Airy function. With the wave equation, the usual dispersion estimate holds for each term, hence we can sum sufficiently many of them and still get good bounds. However, for the semi-classical Schrödinger flow, even the very first modes - localized at distance $h^{2/3}$ from $\partial\Omega$ (known as gallery modes) yield a {\it sharp} loss of $1/6$ in both dispersion and Strichartz estimates (see \cite{doi08}.)
\begin{thm}\label{thmStrichartz}
Let $d\geq 2$, $(q,r)$ such that $\frac  1q\leq \Big(\frac d 2 -\frac 14\Big)\Big(\frac 1 2 -\frac 1 r\Big)$ and $s=\frac d2-\frac 2q-\frac dr$. There exist $C(d)>0$, $T_{0}>0$ such that, for $v$ solution to \eqref{scl} with data $v_{h,0}\in L^{2}(\Omega_{d})$,
  \begin{equation}
    \label{eq:2}
    \|\psi(hD_{t})v_{h}\|_{L^q([-T_{0},T_0],L^r(\Omega_d))}\leq C(d) h^{-s}\|v_{h,0}\|_{L^2(\Omega_d)}\,.
  \end{equation}
\end{thm}
The proof of Theorem \ref{thmStrichartz} follows from Theorem \ref{thmdispSchrodinger} using the classical $TT^*$ argument and the endpoint argument of Keel-Tao \cite{KeTa98} for $q=2$ when $d\geq 3$. The (scale-invariant) loss at the semi-classical level corresponds to $1/4$ derivative in space, as illustrated with $d=2$, for which the (forbidden) endpoint $(2,\infty)$ with $s=0$ is replaced by $(8/3,\infty)$ with $s=1/4$. This improves \cite{bss08} where for $d=2$, one has $(3,\infty)$. More generally, \cite{bss08} obtains $(2,\infty)$ as an endpoint for $d\geq 3$, e.g. $s=d/2-1$, whereas we have $(2,2(2d-1)/(2d-5))$ as our endpoint pair, with $s=1/(2d-1)$.  For $d=3$, our endpoint pair is $(2,10)$: that $10<+\infty$ allows us to adapt the argument from \cite{bgt04} and obtain well-posedness for the cubic equation, as alluded to earlier. We set $\Omega$ to be a compact manifold such that, in a local coordinate chart that intersects its boundary, the metric may be expressed as in our model. We will later provide examples.
\begin{thm}\label{gwp}
Let $d=3$ and $v_{0}\in H^{1}_{0}(\Omega)$. There exists a unique global in time solution $v\in C_{t}(H^{1}_{0}(\Omega))$ to \eqref{schrod} with $\kappa=1$ (defocusing equation), and its energy is conserved along the flow. For $\kappa=-1$ (focusing equation), the result holds locally in time, and globally provided the mass of $v_{0}$ is sufficiently small.
\end{thm}
Moreover, as in the boundaryless case, preservation of regularity holds and one may adapt the argument of \cite{PTV} to obtain exponential growth for the $H^{m}$ norm of the solution, where $m\in \mathbb{N}$, $m>1$.
\begin{thm}\label{growthHs}
  Let $d=3$ and $v_{0}\in H^{1}_{0}(\Omega)\cap H^{s}(\Omega)$ with $s>1$. Then the solution $v$ from Theorem \ref{gwp} is $C_{t}(H^{s}(\Omega))$, and for $s=m\in \mathbb{N}$, its norm grows at most exponentially: there exists $C=C(m,\|v_{0}\|_{H^{m}(\Omega)})$ such that,
  \begin{equation}
    \label{eq:3}
      \| v(t)\|_{H^{m}(\Omega)}\leq C \exp( C t)\,.
    \end{equation}
  \end{thm}
Therefore, well-posedness for the defocusing cubic equation on such model convex domains is similar to that of generic boundaryless manifolds, and we expect it will hold on any generic 3D compact manifold with strictly convex boundary once Theorem \ref{thmStrichartz} is generalized to such manifolds.
  
We conclude this introduction with a discussion on linear Strichartz estimates and their optimality. In \cite{doi}, we proved that there must be a loss of at least $\frac 16$ derivatives in Strichartz estimates for \eqref{scl}, which is obtained when the data is a gallery mode. Whether or not this result is sharp is unknown at present, nor even if a loss in the semi-classical setting {\it should} provide losses in classical time in the case of a generic non-trapping domain where concave portions of the boundary could act like mirrors and refocus wave packets (yielding unavoidable losses in dispersion).
In fact, understanding Strichartz estimates in exterior domains seems to be a very delicate task: obstructions from the compact case no longer apply, at least in the case of non-trapping obstacles.
Thus, one may ask if {\it all Strichartz estimates hold}. The conflict between this questioning and the failure of semi-classical Strichartz (and dispersion) near the boundary is only apparent: for non trapping domains, a wave packet would spend too short a time in a too narrow region near the boundary to be a contradiction by itself. 

For the wave equation, Strichartz estimates with losses were obtained in \cite{bss08} using short time parametrices constructions from \cite{smso06}. As already noticed, the main advantage of \cite{bss08} is also its main weakness: by considering only time intervals that allow for no more than one reflection of a given wave packet, one may handle any boundary but one does not see the full effect of dispersion in the tangential variables. New results in both positive and negative directions were obtained recently, for strictly convex domains: \cite{ILP3} proves Strichartz estimates for the wave equation to hold true on the domain $(\Omega_{d=2},g_F)$ with at most $1/9$ loss. For $d=2$, \cite{bss08} obtained $\frac 16$ instead of $\frac 19$ (but for any boundary), while \cite{ilp12} provides $\frac 14$. Arguments from \cite{ILP3} rely on improving the parametrix construction of \cite{ilp12} and the resulting bounds on the Green function : degenerate stationary phase estimates in \cite{ilp12} may be refined to pinpoint the space-time location of swallowtail singularities (worst case scenario). It turns out that, for the wave equation, such singularities only happen at an exceptional, discrete set of times. The proof of Theorem \ref{thmdispSchrodinger} will rely on similar refinements of degenerate stationary phase estimates together with refined estimates on gallery modes from \cite{doi}, all of which are of independent interest.

\begin{rmq}\label{rmqcasgeneral}
Adapting the parametrix construction for the wave flow from \cite{ILLP}, one may extend Theorem \ref{thmdispSchrodinger} to a domain $\Omega$ whose boundary is everywhere strictly (geodesically) convex: for every point $(0,y_0)\in \partial \Omega$ there exists $(0,y_0,\xi_0,\eta_0)\in T^*\Omega$ where the boundary is micro-locally strictly convex, i.e. such that  there exists a bicharacteristic passing through $(0,y_0,\xi_0,\eta_0)$ that intersects $\partial\Omega$ tangentially having exactly second order contact  with the boundary and remaining in the complement of $\partial\overline{\Omega}$. This will be addressed elsewhere.%
\end{rmq}

\begin{rmq}
One expects the interior of a strictly convex domain to be a worst case scenario. At the opposite end, we now have a much better understanding outside a strictly convex obstacle, where the full set of Strichartz estimates are known to hold (\cite{doi08}) and where dispersion was recently addressed in \cite{ildispext}, where diffraction effects related to the Arago-Poisson spot turn out to be significant for $d\geq 4$. 
\end{rmq}

In the remaining of the paper, $A\lesssim B$ means that there exists a constant $C$ such that $A\leq CB$; this constant may change from line to line and is independent of all parameters but the dimension $d$. It will be explicit when (very occasionally) needed. Similarly, $A\sim B$ means both $A\lesssim B$ and $B\lesssim A$.

\subsection*{Acknowledgments} The author would like to thank Daniel Tataru for helpful and stimulating discussions on this problem and Nikolay Tzvetkov for suggesting that the result of \cite{PTV} may be relevant for our nonlinear application.%
\section{The semi-classical Schrödinger propagator: spectral analysis and parametrix construction}
We recall a few notations, where $Ai$ denotes the standard Airy function (see e.g. \cite{AFbook} for well-known properties of the Airy function), $Ai(x)=\frac 1 {2\pi } \int_{\R} e^{ i  (\frac{\sigma^{3}}{3}+\sigma x)} \,d\sigma$ : define
\begin{equation}
  \label{eq:Apm}
  A_\pm(z)=e^{\mp i\pi/3} Ai(e^{\mp i\pi/3} z)=-e^{\pm 2i\pi/3} Ai(e^{\pm 2i\pi/3} (-z))\,,\,\,\text{ for } \,
  z\in \mathbb{C}\,,
\end{equation}
then one checks that $Ai(-z)=A_+(z)+A_-(z)$ (see \cite[(2.3)]{AFbook}). The next lemma is proved in \cite[Lemma 1]{ilpCE} and requires the classical notion of asymptotic expansion: a function $f(w)$ admits an asymptotic expansion for $w\rightarrow 0$ when there exists a (unique) sequence $(c_{n})_{n}$ such that, for any $n$, $\lim_{w\rightarrow 0} w^{-(n+1)}(f(w)-\sum_{0}^{n} c_{n} w^{n})=c_{n+1}$. We denote $ f(w)\sim_{w} \sum_{n} c_{n} w^{n}$.
\begin{lemma}\label{lemL} (see  \cite[Lemma 1]{ilpCE})
Define, for $\omega \in \R$, $  L(\omega)=\pi+i\log \frac{A_-(\omega)}{A_+(\omega)}$,
then $L$ is real analytic and strictly increasing. We also
have
\begin{equation}
  \label{eq:propL}
  L(0)=\pi/3\,,\,\,\lim_{\omega\rightarrow -\infty} L(\omega)=0\,,\,\,
  L(\omega)=\frac 4 3 \omega^{\frac 3 2}+\frac{\pi}{2}-B(\omega^{\frac 3
    2})\,,\,\,\text{ for } \,\omega\geq 1\,,
\end{equation}
with $ B(u)\sim_{1/u} \sum_{k=1}^\infty b_k u^{-k}$, $b_k\in\R$, $b_1> 0$.
Finally,
  $Ai(-\omega_k)=0 \iff L(\omega_k)=2\pi k$ and
 $ L'(\omega_k)=2\pi \int_0^\infty Ai^2(x-\omega_k) \,dx\,$
where here and thereafter, $\{-\omega_k\}_{k\geq 1}$ denote the zeros of the Airy function in decreasing order.
\end{lemma}
\subsection{Spectral analysis of the Friedlander model}\label{sec:spect}

Our domain is $\Omega_d=\{(x,y)\in\mathbb{R}^d|, x>0,y\in\mathbb{R}^{d-1}\}$ and Laplacian $\Delta_F$ given by \eqref{eq:LapM}.
As $\Delta_F$ has constant coefficients in $y$, taking the Fourier transform in the $y$ variable, it transforms into
$-\partial^2_x+|\theta|^2+xq(\theta)$. For $\theta\neq 0$, this operator is a positive self-adjoint operator 
on $L^2(\mathbb{R}_+)$, with compact resolvent. 

\begin{lemma}\label{lemorthog} (see \cite[Lemma 2]{ilpCE})
There exist eigenfunctions $\{e_k(x,\theta)\}_{k\geq 0}$ of $-\partial^2_x+|\theta|^2+xq(\theta)$ with corresponding eigenvalues $\lambda_k(\theta)=|\theta|^2+\omega_k q(\theta)^{2/3}$, that are an Hilbert basis for $L^{2}(\mathbb{R}_{+})$. These eigenfunctions are explicit in terms of Airy functions:
\begin{equation}\label{eig_k}
 e_k(x,\theta)=\frac{\sqrt{2\pi} q(\theta)^{1/6}}{\sqrt{L'(\omega_k)}}
Ai\Big(xq(\theta)^{1/3}-\omega_k\Big)\,,
\end{equation}
and $L'(\omega_k)$ (with $L$ from Lemma \ref{lemL}) is such that
$\|e_k(.,\theta)\|_{L^2(\mathbb{R}_+)}=1$.
\end{lemma}
For $x_0>0$, $\delta_{x=x_0}$ on $\mathbb{R}_+$ may be decomposed as
 $\delta_{x=x_0}=\sum_{k\geq 1} e_k(x,\theta)e_k(x_0,\theta)$.
At fixed $t_0$, consider $u(t_0,x,y)=\psi(hD_y)\delta_{x=x_0,y=y_0}$, where $h\in (0,1)$ is a small parameter and
$\psi\in C^{\infty}_0([\frac 12,\frac 32])$, then the (localized in $\theta$) 
  Green function for \eqref{scl} on $\Omega_{d}$ is
\begin{equation}\label{greenfct} 
G_{h}((t,x,y),(t_0,x_0,y_0))=\sum_{k\geq 1}
\int_{\mathbb{R}^{d-1}}e^{ih(t-t_0)\lambda_k(\theta)}
e^{i<y-y_0,\theta>} 
 \psi(h|\theta|)e_k(x,\theta)e_k(x_0,\theta)d\theta\,.
\end{equation}
In addition to the cut-off $\psi(h|\theta|)$, we may add a spectral cut-off $\psi_{1}(h\sqrt{\lambda_{k}(\theta)})$ under the $\theta$ integral, where $\psi_{1}$ is also such that $\psi_{1}\in  C^{\infty}_0([\frac 12,\frac 32])$. Indeed,
$$
-\Delta_{F} \Big(\psi(h|\theta|) e^{i <y,\theta> } e_{k}(x,\theta)\Big) =\lambda_{k}(\theta)\psi(h|\theta|) e^{i <y,\theta>} e_{k}(x,\theta)\,.
$$
On the flow, this is nothing but $\psi_1(hD_t)$ and this smoothes out the Green function.
\begin{rmq}\label{rmqcharset}
As remarked in \cite{ilp12} (see also \cite{ILP3}) for the wave propagator, after adding $\psi_1(h\sqrt{\lambda_k(\theta)})$, the significant part of the sum over $k$ in \eqref{greenfct} becomes a finite sum over $k\lesssim 1/h $. Indeed, with $\tau=\frac hi \partial_t=hD_t$, $\xi=\frac hi \partial_x=hD_x$, $\eta=\frac hi \nabla_y=hD_y$, the characteristic set of $ih\partial_t-h^2\Delta_F$ is $\tau=\xi^2+|\eta|^2+xq(\eta)$. %
Using $\tau=hD_t=h\lambda_k(D_y)$, one obtains (at the symbolic level) that on the micro-support of any gallery mode associated to $\omega_k$ we have
\begin{equation}\label{eq:charset}
h^{2/3}\omega_k q^{2/3}(\eta)=|\xi|^2+xq(\eta).
\end{equation}
We may assume that on the support of $\psi(\eta)\psi_1(h\sqrt{\lambda_k(\eta/h)})$ one has $h^{2/3}\omega_k\leq \varepsilon_0$ with $\varepsilon_0$ small. This is compatible with \eqref{eq:charset} since it is equivalent to $|\xi|^2\lesssim \varepsilon_0$. Considering the asymptotic expansion of $\omega_{k}\sim k^{2/3}$ the condition $h^{2/3}\omega_k\leq \varepsilon_0$ yields $k\lesssim \varepsilon_0/h$.
\end{rmq}
\begin{rmq}
As in \cite{ilp12}, the remaining part of the Green function (corresponding to larger values of $k$) will essentially be transverse: at most one reflection for $t\in [0,T_0]$ with $T_0$ small (depending on the above choice of $\varepsilon_0$). Hence, this regime can be dealt with as in \cite{bss08} to get the free space decay and we will ignore it in the upcoming analysis.
\end{rmq}
Reducing the sum to $k\leq \varepsilon_0/h$ is equivalent to adding a spectral cut-off $\phi_{\varepsilon_0}(x+h^2D_{x}^{2}/q(\theta))$ in the Green function, where $\phi_{\varepsilon_0}=\phi(\cdot/\varepsilon_0)$ for some smooth cut-off function $\phi\in C^{\infty}_0([-1,1])$: using that the eigenfunctions of the operator
$-\partial^2_{x}+xq(\theta)$ are also $e_k(x,\theta)$ but associated to the eigenvalues $\lambda_k(\theta)-|\theta|^2=\omega_k q^{2/3}(\theta)$, we can localize with respect to $x+h^2D_{x}^{2}/q(\theta)$ : notice
$(x+h^2D_{x}^{2}/q(\theta))e_{k}(x,\theta)=(\omega_k q^{2/3}(\theta)/q(\theta)) e_{k}(x,\theta)$ and this new localization operator is exactly associated by symbolic calculus to the cut-off $\phi_{\varepsilon_0}(\omega_{k}/ q(\theta)^{1/3})$. We therefore set, for $(t_0,x_0,y_0)=(0,a,0)$,
\begin{multline}\label{greenfctbiseps0} 
G_{h,\varepsilon_0}(t,x,y,0,a,0)  =  \sum_{k\geq 1}
 \int_{\mathbb{R}^{d-1}}e^{iht\lambda_k(\theta)}
e^{i<y,\theta>}  \psi(h|\theta|)\psi_{1}(h\sqrt{\lambda_{k}(\theta)})\\
\times \phi_{\varepsilon_0}(\omega_{k}/q(\theta)^{1/3})
 e_k(x,\theta)e_k(a,\theta)d\theta\,.
\end{multline}
In the following we introduce a new, small parameter $\gamma$ satisfying $\sup{(a,h^{2/3})}\lesssim \gamma\leq \varepsilon_0$ and then split the (tangential part of the) Green function into a dyadic sum $G_{h,\gamma}$ corresponding to a dyadic partition of unity supported for $\omega_k /q(\theta)^{1/3}\sim \gamma \sim 2^j\sup{(a,h^{2/3})}\leq \varepsilon_0$.
Let $\psi_2(\cdot/\gamma):=\phi_{\gamma}(\cdot)-\phi_{\gamma/2}(\cdot)$ and decompose $\phi_{\varepsilon_0}$ as follows
\begin{equation}\label{partunitpsi2}
\phi_{\varepsilon_0}(\cdot) %
=\phi_{\sup{(a,h^{2/3})}}(\cdot)+\sum_{\gamma=2^j \sup{(a,h^{2/3})}, 1\leq j<\log_2(\varepsilon_0/\sup{(a,h^{2/3})})}\psi_2(\cdot/\gamma),
\end{equation}
which allows to write $G_{h,\varepsilon_0}=\sum_{\sup{(a,h^{2/3})}\leq \gamma<1}G_{h,\gamma}$ where (rescaling the $\theta$ variable for later convenience) $G_{h,\gamma}$ takes the form
\begin{multline}
\label{greenfctbis} G_{h,\gamma}(t,x,a,y)  =  \sum_{k\geq 1}
\frac{1}{h^{d-1}} \int_{\mathbb{R}^{d-1}}e^{iht\lambda_k(\eta/h)}
e^{\frac ih <y,\eta>} \psi(|\eta|)\psi_{1}(h\sqrt{\lambda_{k}(\eta/h)})\\
\times \psi_{2}(h^{2/3}\omega_{k}/(q(\eta)^{1/3}\gamma))
e_k(x,\eta/h)e_k(a,\eta/h)d\eta\,.
\end{multline}
\begin{rmq}
When $\gamma=\sup{(a,h^{2/3})}$, according to \eqref{partunitpsi2}, we should, in \eqref{greenfctbis}, write $\phi_{\sup{(a,h^{2/3})}}$ instead of $\psi_2(\cdot/\sup{(a,h^{2/3})})$. However, for values $h^{2/3}\omega_k\lesssim \frac 12 \sup{(a,h^{2/3})}$, the corresponding Airy factors are exponentially decreasing and provide an irrelevant contribution:  writing $\phi_{\sup{(a,h^{2/3})}}$ or $\psi_2(\cdot/\sup{(a,h^{2/3})})$ yields the same contribution in $G_{h,\sup{(a,h^{2/3})}}$ modulo $O(h^{\infty})$. In fact, when $a<h^{2/3}$ is sufficiently small, there are no $\omega_k$ satisfying $h^{2/3}\omega_k/q^{1/3}(\eta)<h^{2/3}/2$ as $\omega_k\geq \omega_1> 2.33$ and $|\eta|\in [\frac 12, \frac 32]$; on the other hand, when $a\gtrsim h^{2/3}$ and $h^{2/3}\omega_k/q^{1/3}(\eta)\leq a/2$ then the Airy factor of $e_k(a,\eta/h)$ is exponentially decreasing (see \cite[Section 2.1.4.3]{AFbook} for details). In order to streamline notations, we use the same formula \eqref{greenfctbis} for each $G_{h,\gamma}$.  
\end{rmq}
From an operator point of view, with $G_h(\cdot)$ the semi-classical Schr\"odinger propagator, we are considering (with $i D=\partial$) $G_{h,\gamma}=\psi(hD_{y})\psi_{1}(h\sqrt{-\Delta_{F}})\psi_{2}((x+h^2D_{x}^{2}/ q(hD_y))/\gamma) G_h$.
 \begin{rmq}
 For $a\lesssim h^{2/3}$, \cite{doi} proved $ \|G_{h,h^{2/3}}(t,\cdot,a\lesssim h^{2/3},\cdot)\|_{L^{\infty}}\lesssim \frac{1}{h^d}\bigl(\frac{h}{t}\bigr)^{(d-1)/2}h^{1/3}$.
 The proof in \cite{doi} has $q(\eta)=|\eta|^2$ but easily extends to a positive definite quadratic form $q$. The subsequent $1/6$ loss in homogeneous Strichartz estimates is optimal for $a\lesssim h^{2/3}$: in \cite[Theorem 1.8]{doi} we suitably chose Gaussian data whose associated semi-classical Schr\"odinger flow saturates the above bound. Those are the so-called whispering gallery modes.
 \end{rmq}
We briefly recall a variant of the Poisson summation formula that will be crucial to analyze the spectral sum defining $G_{h,\gamma}$ (see \cite[Lemma 3]{ilpCE} for the proof.)
\begin{lemma}
  In $\mathcal{D}'(\R_\omega)$, one has
$ \sum_{N\in \Z} e^{-i NL(\omega)}= 2\pi \sum_{k\in \N^*} \frac 1
    {L'(\omega_k)} \delta(\omega-\omega_k)\,$, e.g.  $\forall \phi\in C_{0}^{\infty}$, 
  \begin{equation}
    \label{eq:AiryPoissonBis}
        \sum_{N\in \Z} \int e^{-i NL(\omega)} \phi(\omega)\,d\omega = 2\pi \sum_{k\in \N^*} \frac 1
    {L'(\omega_k)} \phi(\omega_k)\,.
  \end{equation}
\end{lemma}
Using \eqref{eq:AiryPoissonBis} on $G_{h,\gamma}$, we transform the sum over $k$ into a sum over $N\in \Z$, as follows
\begin{multline}\label{vhNgamN}
  \hat{G}_{h,\gamma}(t,x,a,\eta/h)= \frac 1 {2\pi} \sum_{N\in \Z} \int_{\R}e^{-i NL(\omega)}(|\eta|/h)^{2/3}q^{1/3}(\eta/|\eta|) e^{\frac ih t|\eta|^2(1+h^{2/3}\omega q^{1/3}(\eta/|\eta|)/|\eta|^{2/3})} \\
{}\times \psi_{1}\Big(|\eta|\sqrt{1+h^{2/3}\omega q^{2/3}(\eta/|\eta|)/|\eta|^{2/3}}\Big)\psi_{2}(h^{2/3}\omega/(q^{1/3}(\eta)\gamma))\\
{}\times Ai (xq^{1/3}(\eta)/h^{2/3}-\omega) Ai(aq^{1/3}(\eta)/h^{2/3}-\omega)d\omega,
\end{multline}
where $ \hat{G}_{h,\gamma}$ is the Fourier transform in $y$.
For $\sup{(a,h^{2/3})}\leq \gamma <1$, we let $\lambda_{\gamma} =\frac{\gamma^{3/2}}{h}$; when $h^{2/3}\lesssim a$ and $\gamma\sim a$ we write $\lambda:=\frac{a^{3/2}}{h}$. Airy factors are (after rescaling)
\begin{equation}
  \label{eq:bis47}
  Ai(xq^{1/3}(\eta)/h^{2/3}-\omega) %
  =\frac{q^{1/6}(\eta)\lambda_{\gamma}^{1/3}}{2\pi } \int e^{i q^{1/2}(\eta)\lambda_{\gamma}(\frac{\sigma^{3}}{3}+\sigma (\frac{x}{\gamma} -\omega/(q^{1/3}(\eta)\lambda_{\gamma}^{2/3}))} \,d\sigma.
\end{equation}
Rescaling $\omega=q^{1/3}(\eta)\lambda_{\gamma}^{2/3} \alpha=q^{1/3}(\eta)\gamma \alpha /h^{2/3}$ in \eqref{vhNgamN} yields
\begin{multline}
  \label{eq:bis48}
    \hat{G}_{h,\gamma}(t,x,a,\eta/h)= \frac{\lambda_{\gamma}^{4/3}}{(2\pi)^{3}h^{2/3}} \sum_{N\in \Z} \int_{\R}\int_{\R^{2}} e^{\frac i h  \tilde\Phi_{N,a,\gamma}(\eta,\alpha,s,\sigma,t,x)} q(\eta)\\
{}\times  \psi_{1}\Big(|\eta|\sqrt{1+\gamma\alpha q(\eta/|\eta|)}\Big) \psi_{2}(\alpha)\, ds d\sigma d\alpha\,,
\end{multline}
\begin{multline}
  \label{eq:bis49}
     \tilde\Phi_{N,a,\gamma}(\eta,\alpha,s,\sigma,t,x)=t|\eta|^2(1+\gamma\alpha q(\eta/|\eta|))-Nh L(q^{1/3}(\eta)\lambda_{\gamma}^{2/3} \alpha)\\
     +\gamma^{3/2} q^{1/2}(\eta)\Big(\frac{\sigma^{3}} 3+\sigma(\frac{x}{\gamma}-\alpha)
     +\frac {s^{3}} 3+s(\frac{a}{\gamma}-\alpha)\Big)\,.
\end{multline}
Here $Nh L(q^{1/3}(\eta)\lambda_{\gamma}^{2/3} \alpha)=\frac 43 Nq^{1/2}(\eta)(\gamma\alpha)^{3/2}-NhB(q^{1/2}(\eta)\lambda_{\gamma} \alpha^{3/2})$ and we recall that, asymptotically, $B(q^{1/2}(\eta)\lambda_{\gamma} \alpha^{3/2})\sim_{1/(\lambda_{\lambda}\alpha^{3/2}} \sum_{k\geq 1}\frac{b_k}{(q^{1/2}(\eta)\lambda_{\gamma}\alpha^{3/2})^k}$, where
on the support of $\psi_2(\alpha)$ we have $\alpha\sim 1$. 
At this point, notice that, as $|\eta|\in [1/2,3/2]$, we may drop the $\psi_{1}$ localization in \eqref{eq:bis48} by support considerations (slightly changing any cut-off support if necessary). Therefore,
\begin{equation}
  \label{eq:bis48bis}
    G_{h,\gamma}(t,x,a,y)= \frac 1 {(2\pi)^{3}} \frac{\gamma^2}{h^{d+1}}\sum_{N\in \Z} \int_{\R^{d}}\int_{\R^{2}} e^{\frac i h  (<y,\eta>+\tilde\Phi_{N,a,\gamma})}  q(\eta) \psi(|\eta|)\\
{}\times   \psi_{2}(\alpha)  \, ds d\sigma  d\alpha d\eta\,.
\end{equation}
\begin{rmq} 
Both formulas \eqref{eq:bis48bis} and \eqref{greenfctbis} define exactly the same object and both will be necessary to prove the dispersive estimates. The sum over the eigenmodes $e_k$ will be particularly useful for small values of $a\lesssim (ht)^{1/2}$, while for large values of the initial distance to the boundary the sum over $N$ will take over. While both formulas coincide, there is a duality between the two: when $a$ is small, there are less terms in the sum over $k$ in \eqref{greenfctbis}, while when $a>(ht)^{1/2}$ there are less terms in the sum over the reflections $N$. 
\end{rmq}

\begin{rmq}
In order to generalize Theorem \ref{thmdispSchrodinger} to a convex domain as in Remark \ref{rmqcasgeneral}, our construction of gallery modes from \cite{ILLP} will turn out to be crucial. 
Notice that in the general situation even the regime $a\leq h$ has its own difficulties: even deciding how the initial data should be chosen in order the Dirichlet condition to be satisfied on the boundary becomes a non trivial issue. %
In \cite{ILLP}, we bypass our lack of understanding of the eigenfunctions for the Laplace operator and use spectral theory for the model Laplace operator \eqref{eq:LapM} in order to construct a suitable initial data for very small $a$. Thus, constructing a parametrix in the model case (in terms of both eigenmodes and sum over reflections) and obtaining its best possible decay properties is %
important in order to further generalize Theorem \ref{thmdispSchrodinger}.
\end{rmq}

\begin{rmq}\label{rmqxgama}
As noticed in \cite{ilp12}, the symmetry of the Green function (or its suitable spectral truncations) with respect to $x$ and $a$ allows to restrict the computations of the $L^{\infty}$ norm to the region $0\leq x\leq a$. In other words, instead of evaluating $\|G_{h,\varepsilon_0}\|_{L^{\infty}(0\leq x,y)}(t,\cdot)$ it would be enough to bound $\|G_{h,\varepsilon_0}\|_{L^{\infty}(0\leq x\leq a,y)}(t,\cdot)$.
\end{rmq}
\section{Dispersive estimates for the semi-classical Schrödinger flow}
We now prove dispersive bounds for $G_{h,\varepsilon_0}(t,x,a,y)$ on $\Omega_d$ for fixed $|t|\in [h,T_0]$, with small $T_0>0$. We will estimate separately $\|G_{h,\gamma}(t,\cdot)\|_{L^{\infty}(\Omega_d)}$ for every $\gamma$ such that $\sup{(a,h^{2/3})}\lesssim \gamma\leq \varepsilon_0$. 
Henceforth we assume $t>0$. We sort out several situations, with a fixed (small) $\epsilon>0$. Firstly, $\sup{(h^{2/3-\epsilon},(ht)^{1/2})}\leq a\leq \varepsilon_0$: in this case, for  all $\gamma$ such that $\sup{(a,h^{2/3})}\lesssim \gamma\leq \varepsilon_0$ we have  $\sup{(h^{2/3-\epsilon},(ht)^{1/2})}\leq a\lesssim\gamma\leq \varepsilon_0$. This is our main case, where only formula \eqref{eq:bis48bis} is useful; integrals with respect to $\sigma,s$ have up to third order degenerate critical points and we need to perform a very detailed analysis of these integrals. In particular, the "tangential" case $\gamma\sim a$ provides the worst decay estimates. When $8a\leq \gamma$, integrals in \eqref{eq:bis48bis} have degenerate critical points of order at most two. We call this regime "transverse": summing up $\sum_{8a\leq \gamma}\|G_{h,\gamma}(t,\cdot)\|_{L^{\infty}}$ still provides a better contribution than $\|G_{h,a}(t,\cdot)\|_{L^{\infty}}$.
Secondly, for $a\lesssim \sup{(h^{2/3-\epsilon},(ht)^{1/2})}$, we further subdivide:  $\sup{(h^{2/3-\epsilon},(ht)^{1/2})}\leq \gamma \leq \varepsilon_0$, which is similar to the previous "transverse" regime, and estimates follow using \eqref{eq:bis48bis} ; and $\sup{(a,h^{2/3})}\lesssim \gamma\lesssim \sup{(h^{2/3-\epsilon},(ht)^{1/2})}$,  where we use \eqref{greenfctbis} to evaluate its $L^{\infty}$ norm.
\subsection{Case $\sup{(h^{2/3-\epsilon},(ht)^{1/2})}\leq a \leq \varepsilon_0$, with  (small) $\epsilon>0$}
Here we use \eqref{eq:bis48bis}. As $\sup{(a,h^{2/3})}=a$, we consider $\gamma$ such that $a\lesssim \gamma\leq \varepsilon_0$. Let $\lambda_{\gamma}:=\gamma^{3/2}/h$, then $\lambda_{\gamma}\geq h^{-3\epsilon/2}$. 
\begin{rmq}
The approach below applies for all $h^{2/3-\epsilon}\lesssim a\leq \varepsilon_0$, providing sharp estimates for each $G_{h,\gamma}$ for all $h^{2/3-\epsilon}\lesssim a\lesssim \gamma\leq \varepsilon_0$; however, when summing up over $a\l\lesssim \gamma\leq (ht)^{1/2}$, bounds for $G_{h,\varepsilon_0}$  get worse than those from Theorem \ref{thmdispSchrodinger}. Hence we restrict to values $\sup{(h^{2/3-\epsilon},(ht)^{1/2})}\leq a \leq \varepsilon_0$, while lesser values will be dealt with differently later.
\end{rmq}
First, we prove that the sum defining $G_{h,\gamma}$ in \eqref{eq:bis48bis} over $N$ is essentially finite and we estimate of the maximum number of terms in this sum.
\begin{prop}\label{propcardN}
For a fixed $t\in (h,T_0]$ the sum \eqref{eq:bis48bis} over $N$ is essentially finite and $|N|\lesssim \frac{1}{\sqrt{\gamma}}$. In other words, if $M$ is a sufficiently large constant (depending only on $q$), then
\[
 \frac 1 {(2\pi)^{3}} \frac{\gamma^2}{h^{d+1}}\sum_{N\in \Z, |N|\geq \frac{M |t|}{\sqrt{\gamma}}} \int_{\R\times\R^{d-1}}\int_{\R^{2}} e^{\frac i h  (<y,\eta>+\tilde\Phi_{N,a,\gamma})}  q(\eta) \psi(|\eta|)\psi_{2}(\alpha)  \, ds d\sigma  d\alpha d\eta=O(h^{\infty}).
\]
\end{prop}
\begin{proof}
The proof follows easily using non-stationary phase arguments for $N\geq M  \frac{t}{\sqrt{\gamma}}$ for some $M$ sufficiently large. Critical points with respect to $\sigma,s$ are such that
\begin{equation}\label{statssigma}
\sigma^2=\alpha-{x}/{\gamma},\quad s^2=\alpha-{a}/{\gamma},
\end{equation}
and as $x\geq 0$, $\tilde\Phi_{N,a,\gamma}$ may be stationary in $\sigma$, $s$ only if $|(\sigma,s)|\leq \sqrt{\alpha}$. As $\psi_2(\alpha)$ is supported near $1$, it follows that we must also gave $x\leq 2\gamma$, otherwise $\tilde\Phi_{N,a,\gamma}$ is non-stationary with respect to $\sigma$.
If $|(\sigma,s)|\geq (1+|N|^{\epsilon})\sqrt{\alpha}$ for some $\epsilon>0$ we can perform repeated integrations by parts in $\sigma,s$ to obtain $O(((1+N^{\epsilon})\lambda_{\gamma})^{-n})$ for all $n\geq 1$. Let $\chi$ a smooth cutoff supported in $[-1,1]$ and write $1=\chi(\sigma/(N^{\epsilon}\sqrt{\alpha}))+(1-\chi)(\sigma/(N^{\epsilon}\sqrt{\alpha}))$, then
\begin{multline}
\psi(|\eta|)\sum_{N\in \mathbb{Z}} \int_{\R}\int_{\R^{2}} e^{\frac i h  \tilde\Phi_{N,a,\gamma}}  
\psi_{2}(\alpha)\chi(s/(N^{\epsilon}\sqrt{\alpha})) (1-\chi)(\sigma/(N^{\epsilon}\sqrt{\alpha})) \, ds d\sigma  d\alpha \\
\lesssim \lambda_{\gamma}^{-1/3} \sup_{\alpha,|\eta|\in[1/2,3/2]} \Big|Ai\Big((a-\gamma \alpha)q^{1/3}(\eta)/h^{2/3}\Big)\Big|\sum_{N\in \mathbb{Z}} \Big((1+N^{\epsilon})\lambda_{\gamma})^{-n}\Big) %
=O(h^{\infty})\,,
\end{multline}
where in the last line we used $\lambda_{\gamma}\geq h^{-3\epsilon/2}$, $\epsilon>0$. In the same way, we can sum on the support of $(1-\chi)(s/(N^{\epsilon}\sqrt{\alpha}))$ and obtain a $O(h^{\infty})$ contribution. Therefore, we may add cut-offs $\chi(\sigma/(N^{\epsilon}\sqrt{\alpha}))$ and $\chi(s/N^{\epsilon}\sqrt{\alpha}))$ in $G_{h,\gamma}$ without changing its contribution modulo $O(h^{\infty})$. Using again \eqref{eq:bis49}, we have, at the critical point of $\tilde\Phi_{N,a,\gamma}$ with respect to $\alpha$
\begin{equation}\label{statalpha}
\frac{t}{{\gamma}^{1/2}}q(\eta)-q^{1/2}(\eta)(s+\sigma)=2Nq^{1/2}(\eta)\sqrt{\alpha}\Big(1-\frac 34 B'(\eta\lambda \alpha^{3/2})\Big),
\end{equation}
and as $|(\sigma,s)|\lesssim (1+|N|^{\epsilon})\sqrt{\alpha}$ on the support of $\chi(\sigma/(N^{\epsilon}\sqrt{\alpha}))\chi(s/(N^{\epsilon}\sqrt{\alpha}))$, $\tilde\Phi_{N,a,\gamma}$ may be stationary with respect to $\alpha$ only when $\frac{t}{\sqrt{\gamma}}\sim 2N$. As  $B'(\eta\lambda \alpha^{3/2})=O(\lambda_{\gamma}^{-3})=O(h^{9\epsilon/2})$, its contribution is irrelevant. From \eqref{statssigma} and \eqref{statalpha}, if
\begin{equation}\label{condstatalpha}
\frac{t}{{\gamma}^{1/2}}\frac{|\eta|}{\sqrt{\alpha}}q^{1/2}(\eta/|\eta|)\notin [2(N-1),2(N+1)],
\end{equation}
then the phase is non-stationary in $\alpha$. Recall that $q$ is positive definite and let
\begin{equation}\label{limitshalfq}
m_0:=\inf_{\Theta\in \mathbb{S}^{d-2}}q^{1/2}(\Theta), \quad M_0=\sup_{\Theta\in \mathbb{S}^{d-2}}q^{1/2}(\Theta).
\end{equation}
As $|\eta|,\alpha\in[\frac 12,\frac 32]$ on the support of the symbol, if $2(N-1)>\frac{t}{\sqrt{\gamma}}\times M_0\frac{3/2}{\sqrt{1/2}}$ or if $2(N+1)<\frac{t}{\sqrt{\gamma}}\times m_0\frac{1/2}{\sqrt{3/2}}$, then the phase is non-stationary in $\alpha$ as its first order derivative behaves like $N$. Repeated integrations by parts allow to sum up in $N$ as above, and conclude.
\end{proof}

\begin{rmq}\label{rmqA}
We can now add an even better localization with respect to $\sigma$ and $s$: on the support of $(1-\chi)(\sigma/(2\sqrt{\alpha}))$ and $(1-\chi)(s/(2\sqrt{\alpha}))$ the phase is non-stationary in $\sigma$ or $s$, and integrations by parts yield an $O(\lambda_{\gamma}^{-\infty})$ contribution. According to Proposition \ref{propcardN}, the sum over $N$ has finitely many terms, and therefore summing yields an $O(h^{\infty})$ contribution.
\end{rmq}

\begin{rmq}\label{rmqB}
We can (and will) also move the factor $e^{iNB(q^{1/2}(\eta) \lambda_{\gamma}\alpha^{3/2})}$ into the symbol as it does not oscillate: indeed, $\alpha,q(\eta)\in [\frac 12,\frac 32]$ on the support of $\psi_2$, $\psi$ and $N\sim \frac{t}{\sqrt{\gamma}}$, we obtain, 
\[
NB(q^{1/2}(\eta)\lambda_{\gamma} \alpha^{3/2})\sim N\sum_{k\geq 1}\frac{b_k}{(q^{1/2}(\eta)\lambda_{\gamma} \alpha^{3/2})^k} \sim \frac{N b_1}{q^{1/2}(\eta)\lambda_{\gamma}}\sim \frac{ht}{\gamma^2},
\]
using Lemma \ref{lemL}, and as we consider here $(ht)^{1/2}\lesssim \gamma$, this term remains bounded.
\end{rmq}
We set $\Phi_{N,a,\gamma}=<y,\eta>+\tilde\Phi_{N,a,\gamma}-NhB(q^{1/2}(\eta)\lambda_{\gamma}\alpha^{3/2})$: from Remark \ref{rmqB}, in this regime, $\Phi_{N,a,\gamma}$ are the phase functions in the sum of $G_{h,\gamma}$ defined by \eqref{eq:bis48bis}. We have
\begin{multline}
  \label{eq:bis49PhiN}
     \Phi_{N,a,\gamma}(\eta,\alpha,s,\sigma,t,x,y)=<y,\eta>+t|\eta|^2(1+\gamma\alpha q(\eta/|\eta|))\\
     +\gamma^{3/2} q^{1/2}(\eta)\Big(\frac{\sigma^{3}} 3+\sigma(\frac{x}{\gamma}-\alpha)
     +\frac {s^{3}} 3+s(\frac{a}{\gamma}-\alpha)-\frac 43 N\alpha^{3/2}\Big).
\end{multline}
In the following we study, for each $|N|\lesssim \frac{1}{\sqrt{\gamma}}$, the integrals in the sum \eqref{eq:bis48bis}. Notice that when $N=0$ we deal with the free semi-classical Schr\"odinger flow.
\begin{prop}\label{propfreeN=0}
For all $a\in (0,a_0]$, $h\in (0,1)$ and $t\in (h,T_0]$, 
\[
\Big|\sum_{\gamma =2^j a, 0\leq j\leq \log(\frac{\varepsilon_0}{a})}V_{0,h,\gamma}(t,x,y)\Big|\lesssim \frac{1}{h^{d}}\Big(\frac ht\Big)^{d/2}.
\]
\end{prop}
\begin{proof}
In this case ($N=0$) we use \eqref{greenfctbiseps0}, \eqref{partunitpsi2} and \eqref{eq:bis48bis} to write the sum over $\gamma$ as follows
\begin{multline}
\sum_{\gamma =2^j a, 0\leq j\leq \log(\frac{\varepsilon_0}{a})}V_{0,h,\gamma}(t,x,y)=\frac{1}{(2\pi)^3}\frac{1}{h^{d+1}}\int \psi(|\eta|)q(\eta)\phi_{\varepsilon_0}(\alpha)\\
\times e^{\frac ih(<y,\eta>+t|\eta|^2(1+\alpha q(\eta/|\eta|))+q^{1/2}(\eta)(\frac{\sigma^3}{3}+\sigma(x-\alpha)+\frac{s^3}{3}+s(a-\alpha)))} \,d\sigma ds d\alpha d\eta.
\end{multline}
Set $\xi_1=\frac{s+\sigma}{2}$ and $\xi_2=\frac{\sigma-s}{2}$, then $\sigma=\xi_1+\xi_2$ and $s=\xi_1-\xi_2$; the phase in the above integral becomes $<y,\eta>+t|\eta|^2(1+\alpha q(\eta/|\eta|))+q^{1/2}(\eta)(\frac 23 \xi_1^3+2\xi_1\xi_2^2+\xi_1(x+a-2\alpha)+\xi_2(x-a))=\Phi_{0,a,1}$. As $\partial^2_{\alpha}\Phi_{0,a,1}=0$ and $\partial^2_{\xi_1,\alpha}\Phi_{0,a,1}=-2q^{1/2}(\eta)$, the usual stationary phase applies in both $\xi_1,\alpha$ and yields a factor $h$. The critical points are $\xi_{1,c}=\frac{tq^{1/2}(\eta)}{2}$, $\alpha_c=\xi_{1,c}^2+\xi_2^2+\frac{x+a}{2}$. The critical point with respect to $\xi_2$ satisfies $\partial_{\xi_2}\Phi_{0,a,1}|_{\xi_{1,c},\alpha_c}=q^{1/2}(\eta)(4\xi_{1,c}\xi_2+x-a)$ and the second derivative equals $\partial^2_{\xi_2}\Phi_{0,a,1}|_{\xi_{1,c},\alpha_c}=q^{1/2}(\eta)\times 4\xi_{1,c}=2tq(\eta)$. For $t/h\gg 1$, the stationary phase applies and yields a factor $(h/t)^{1/2}$. We are left with the integration with respect to $\eta$. Using $\alpha\leq \varepsilon_0$ on the support of $\phi_{\varepsilon_0}(\alpha)$ and $x\geq 0$, it follows that $\xi_{1,c}^2+\xi_{2,c}^2\leq \varepsilon_0$. Writing $t|\eta|^2q(\eta/|\eta|)=tq(\eta)=2q^{1/2}(\eta)\xi_{1,c}$, the critical value equals 
\[
t|\eta|^2(1+\alpha_c q(\eta/|\eta|))-q^{1/2}(\eta)(\frac 43 \xi_{1,c}^3+4\xi_{1,c}\xi^2_{2,c})=t|\eta|^2+2q^{1/2}(\eta)\xi_{1,c}(\alpha_c-\frac 23 \xi_{1,c}^2-2\xi_{2,c}^2),
\]
and a derivative with respect to $\eta_j$ equals $y_j+2t\eta_j+\partial_{\eta_j}(q^{1/2}(\eta))\xi_{1,c}(\frac 43 \xi_{1,c}^2+x+a)$. We conclude by stationary phase as this yields $\nabla^2_{\eta}\Phi_{0,a,1}|_{\xi_{1,c},\xi_{2,c},\alpha_c}=2t\mathbb{I}_{d-1}(1+O(\varepsilon_0))$. The proof above applies also separately yielding dispersive bounds without loss for each $V_{0,h,\gamma}$.
\end{proof}

As we set $t>0$, from now on we only consider $N\geq 1$.
\begin{prop}\label{propcritptsalphaeta}
Let $N\geq 1$. The phase function $\Phi_{N,a,\gamma}$ can have at most one critical point $(\alpha_c,\eta_c)$ on the support $[\frac 12, \frac 32]$ of the symbol. At critical points in $(\alpha,\eta)$,  the determinant of the Hessian  is comparable to $ t^{d-1}\times \gamma^{3/2}N$. Stationary phase applies in both $\alpha\in [1/2,3/2]$ and $\eta\in\mathbb{R}^{d-1}$ and yields a decay factor $(h/t)^{(d-1)/2}\times{(\lambda_{\gamma}N)}^{-1/2}$.
\end{prop}
\begin{proof}
The derivatives of the phase $\Phi_{N,a,\gamma}$ with respect to $\alpha,\eta$ are
\begin{gather*}
\partial_{\alpha}\Phi_{N,a,\gamma}=\gamma^{3/2}q^{1/2}(\eta)\Big(\frac{t}{\sqrt{\gamma}}q^{1/2}(\eta)-(\sigma+s)-2N\sqrt{\alpha}\Big),\\
\nabla_{\eta}\Phi_{N,a,\gamma}=y+2\eta t+\frac{\gamma^{3/2}\nabla q(\eta)}{2q^{1/2}(\eta)}\Big(\frac{\sigma^{3}} 3+\sigma(\frac{x}{\gamma}-\alpha)
     +\frac {s^{3}} 3+s(\frac{a}{\gamma}-\alpha)-\frac 43 N \alpha^{3/2}+\frac{2\alpha t}{\sqrt{\gamma}}q^{1/2}(\eta) \Big).
\end{gather*}
At $\partial_{\alpha}\Phi_{N,a,\gamma}=0$ and $\nabla_{\eta}\Phi_{N,a,\gamma}=0$, critical points are such that
\begin{equation}\label{eq:alphac}
\sqrt{\alpha}=\frac{tq^{1/2}(\eta)}{2N\sqrt{\gamma}}-\frac{s+\sigma}{2N}
\end{equation}
and also (replacing $2N\sqrt{\alpha}$ by $\frac{t}{\sqrt{\gamma}}q^{1/2}(\eta)-(\sigma+s)$ in the expression of $\nabla_{\eta}\Phi_{N,a,\gamma}$)
\begin{equation}\label{eq:etac}
2t\Big(\eta+\frac 12 \gamma \alpha\nabla q(\eta)\Big)=-y-\gamma^{3/2}\frac{\nabla q(\eta)}{2q^{1/2}(\eta)}\Big[\frac{\sigma^3}{3}+\sigma\frac{x}{\gamma}+\frac{s^3}{s}+s\frac{a}{\gamma}-\frac{(s+\sigma)\alpha}{3}\Big].
\end{equation}
From \eqref{condstatalpha} (and support condition on $\eta,\alpha$), a critical point $\alpha_c\in [\frac 12,\frac 32]$ does exist only if 
\begin{equation}\label{suppcondt}
(1-1/N)\frac{\sqrt{1/2}}{3M_0/2}\leq \frac{t}{2N\sqrt{\gamma}}\leq (1+1/N)\frac{\sqrt{3/2}}{m_0/2}\,.
\end{equation}
For $N\geq 2$, fix $M$ sufficiently large such that $[(1-1/2)\frac{\sqrt{1/2}}{3M_0/2}, (1+1/2)\frac{\sqrt{3/2}}{m_0/2}]\subset [1/M,M]$, then \eqref{eq:alphac} may have a solution on the support of $\psi_2$ only when $\frac{t}{2N\sqrt{\gamma}}\in [1/M,M]$. %
 For $N=1$, we obtain the upper bound $\frac{t}{2\sqrt{\gamma}}\leq \frac{4}{m_0}\sqrt{3/2}$ but also, using \eqref{statssigma}, the following lower bounds : either $s+\sigma\geq -\frac 32\sqrt{\alpha}$, in which case $\frac{t}{2\sqrt{\gamma}}\geq\frac{\sqrt{\alpha}}{4|\eta|M_0}$, or $(s+\sigma)\leq -\frac 32\sqrt{\alpha}$ in which case both $s$ and $\sigma$ must take non positive values and in this case
\begin{equation}\label{boundfortstN}
q^{1/3}(\eta)\frac{t}{2\sqrt{\gamma}}\geq \sqrt{\alpha}+\frac{s+\sigma}{2}\geq \frac{a/\gamma}{2(\sqrt{\alpha}-s)}+ \frac{x/\gamma}{2(\sqrt{\alpha}-\sigma)}\geq \frac{a/\gamma}{4\sqrt{\alpha}}.
\end{equation}
Hence, for $t\leq \frac{a/\sqrt{\gamma}}{2\sqrt{3/2}M_0^{2/3}}$ the flow does not reach the boundary (no reflections). 

Let $N\geq 1$ and $t\geq \frac{a/\sqrt{\gamma}}{2\sqrt{3/2}M_0^{2/3}}$ (otherwise the phase is non-stationary). As $\alpha\in [\frac 12,\frac 32]$ and $\gamma\leq \varepsilon_0$, \eqref{eq:etac} may have a critical point $\eta_c$ only when $|y|/2t\in[\frac 12+O(\varepsilon_0),\frac 32+O(\varepsilon_0)]$.  Using $\partial_{\eta_j}q(\eta)=2q_{j,j}\eta_j+\sum_{k\neq j}q_{j,k}\eta_k$, $q_{j,k}=q_{k,j}$ the second order derivatives become
\begin{gather}
\partial^2_{\alpha,\alpha}\Phi_{N,a,\gamma}=-\gamma^{3/2}q^{1/2}(\eta)\frac{N}{\sqrt{\alpha}}\,,\quad\quad%
\partial_{\eta_j}\partial_{\alpha}\Phi_{N,a,\gamma}=\frac{\partial_{\eta_j} q(\eta)}{2q(\eta)}\partial_{\alpha}\Phi_{N,a,\gamma}+\gamma^{3/2}\frac{t}{2\sqrt{\gamma}}\partial_{\eta_j} q(\eta),\\
\partial^2_{\eta_j,\eta_j}\Phi_{N,a,\gamma}=
\begin{multlined}[t]
  2t\Big(1+\gamma \alpha \frac{(\partial_{\eta_j}q(\eta))^2}{4q(\eta)}\Big)+\frac{\gamma^{3/2}}{q^{1/2}(\eta)}\Big(q_{j,j}-\frac{(\partial_{\eta_j}q(\eta))^2}{4q(\eta)}\Big)\\
{}  \times \Big(\frac{\sigma^{3}} 3+\sigma(\frac{x}{\gamma}-\alpha)
     +\frac {s^{3}} 3+s(\frac{a}{\gamma}-\alpha)-\frac 43 N \alpha^{3/2}+2\alpha\frac{t}{\sqrt{\gamma}}q^{1/2}(\eta) \Big)\,, \end{multlined}\\
   \partial^2_{\eta_j,\eta_k}\Phi_{N,a,\gamma}=\begin{multlined}[t]
     2t\gamma\alpha \frac{\partial_{\eta_j} q(\eta)}{2q^{1/2}(\eta)}\frac{\partial_{\eta_k} q(\eta)}{2q^{1/2}(\eta)}+\frac{\gamma^{3/2}}{q^{1/2}(\eta)}\Big(q_{j,k}-\frac{\partial_{\eta_j}q(\eta)\partial_{\eta_k}q(\eta)}{4q(\eta)}\Big)\\
     {}\times \Big(\frac{\sigma^{3}} 3+\sigma(\frac{x}{\gamma}-\alpha)
     +\frac {s^{3}} 3+s(\frac{a}{\gamma}-\alpha)-\frac 43 N \alpha^{3/2}+2\alpha\frac{t}{\sqrt{\gamma}}q^{1/2}(\eta) \Big)\,.\end{multlined}
\end{gather}
At the stationary points, $\nabla^2_{\eta,\eta}\Phi_{N,a,\gamma}\sim 2t(1+O(\gamma))\mathbb{I}_{d-1}+O(\gamma^{3/2})$ where $\mathbb{I}_{d-1}$ denotes the identity matrix in dimension $d-1$ ; as $\varepsilon_0<1$ is small we deduce $\nabla^2_{\eta,\eta}\Phi_{N,a,\gamma}\sim 2t \mathbb{I}_{d-1}$. Hence, stationary phase with respect to $\eta$ yields a factor $( h/t)^{\frac{d-1}{2}}$, while  stationary phase in $\alpha$ yields a factor 
$({\lambda_{\gamma}N})^{-1/2}$ for $N\geq 1$.
\end{proof}

\begin{lemma}\label{lemalphacetac}
Let $N\geq 1$ and $a\lesssim \gamma\leq \varepsilon_0$. The critical point $\eta_c$ of $\Phi_{N,a,\gamma}$ is a function of $s+\sigma$, $(\sigma-s)^2$, $(\sigma-s)\frac{(x-a)}{\gamma}$, $\frac{y}{2t}$ and $\frac{t}{2N\sqrt{\gamma}}$. There exists smooth, uniformly bounded (vector valued) functions $\Theta, \tilde\Theta$ depending on the small parameter $\gamma$, such that
\begin{gather}
\eta^0_c:=\eta_c|_{\sigma=s=0}=-\frac{y}{2t}+\gamma \Theta\Big(\frac{y}{2t},\frac{t}{2N\sqrt{\gamma}},\gamma\Big)\,,\\
\Theta\Big(\frac{y}{2t},\frac{t}{2N\sqrt{\gamma}},\gamma\Big)=-\frac 12 (\frac{t}{2N\sqrt{\gamma}})^2 (q\nabla q)(-\frac{y}{2t})+\gamma\tilde \Theta_{}(\frac{y}{2t},\frac{t}{2N\sqrt{\gamma}},\gamma)\,.
\end{gather}
Moreover, ${\Theta}_1:=\frac{t}{\gamma^{3/2}}\partial_{\sigma}\eta_c$ and ${\Theta}_2:=\frac{t}{\gamma^{3/2}}\partial_{s}\eta_c$ are smooth, uniformly bounded functions.
\end{lemma}
\begin{proof}
We start with the second statement.
Let first $N\geq 2$ and define $M$ as follows
\begin{equation}\label{defM}
M:=4\sup\Big\{\frac{\sqrt{3/2}}{m_0-\varepsilon_0},\frac{M_0+\varepsilon_0}{\sqrt{1/2}}\Big\}, \text{ with } m_0,M_0 \text{ introduced in \eqref{limitshalfq}},
\end{equation}
and assume, without loss of generality, $0<\varepsilon_0<m_0/2$. 
Then $M$ is large enough so that $\Big[(1-1/2)\frac{\sqrt{1/2}}{3M_0/2}, (1+1/2)\frac{\sqrt{3/2}}{m_0/2}\Big]\subset [1/M,M]$ and for $\frac{t}{2N\sqrt{\gamma}}\in [1/M,M]$ and $\frac{|y|}{2t}\in[\frac 14,2]$, 
the critical points $\alpha_c$ and $\eta_c$ of $\Phi_{N,a,\gamma}$ solve \eqref{eq:alphac} and \eqref{eq:etac}. Let $\eta^0_c:=\eta_{c}|_{\sigma=s=0}$ denote the value of $\eta_c$ at $\sigma=s=0$, then, using \eqref{eq:etac}, $\eta^0_c$ solves the following equation,
\[
\eta^0_c+\frac 12\gamma \Big(\frac{t}{2N\sqrt{\gamma}}\Big)^2q(\eta^0_c)\nabla q(\eta^0_c)=-\frac{y}{2t}\,.
\]
For $\frac{t}{2N\sqrt{\gamma}}\in [1/M,M]$, writing $\eta^0_c=-\frac{y}{2t}+\gamma \Theta(\frac{y}{2t},\frac{t}{2N\sqrt{\gamma}},\gamma)$, yields, for $\Theta(\frac{y}{2t},\frac{t}{2N\sqrt{\gamma}},\gamma)$ 
\begin{equation}\label{eqThetaetac}
\Theta+\frac 12(\frac{t}{2N\sqrt{\gamma}})^2 (q\nabla q)(-\frac{y}{2t}+\gamma \Theta)=0\,,
\end{equation}
which further reads as follows, with $\Theta=(\Theta^{(1)},...,\Theta^{{(d-1)}})$ and for all $1\leq l\leq d-1$
\[
\Theta^{(l)}+(\frac{t}{2N\sqrt{\gamma}})^2\sum_{j,k,p}q_{j,k}q_{p,l}(-\frac{y_j}{2t}+\gamma \Theta^{(j)})(-\frac{y_k}{2t_k}+\gamma \Theta^{(k)})(-\frac{y_p}{2t}+\gamma \Theta^{(p)})=0\,.
\]
As $\gamma\leq \varepsilon_0$, this equation has an unique solution, which is a smooth function of $(\frac{y}{2t},\frac{t}{2N\sqrt{\gamma}},\gamma)$ and $\Theta^{(l)}=(\frac{t}{2N\sqrt{\gamma}})^2\Big(\sum_{j,k,p}q_{j,k}q_{p,l}(\frac{y_j}{2t})(\frac{y_k}{2t})(\frac{y_p}{2t})\Big)+\gamma \tilde \Theta^{(l)}$, where $\tilde \Theta=(\tilde \Theta^{(1)},..,\tilde\Theta^{{(d-1)}})$ is a smooth function of $(\frac{y}{2t},\frac{t}{2N\sqrt{\gamma}},\gamma)$. For $N=1$, $t$ may take (very) small values but does not vanish where $\Phi_{N=1,a,\gamma}$ may be stationary and therefore \eqref{eqThetaetac} still holds and $\frac{|y|}{2t}\in[\frac 14,2]$, hence we obtain $\Theta$ in the same way. We now prove that for all $N\geq 1$, $\eta_c$ is a function of $s+\sigma$, $(\sigma-s)^2$, $(\sigma-s)\frac{(x-a)}{\gamma}$, $\frac{y}{2t}$ and $\frac{t}{2N\sqrt{\gamma}}$. This will be useful later on, especially in the proof of upcoming Proposition \ref{dispNpetitpres}. Inserting \eqref{eq:alphac} in \eqref{eq:etac} yields
\begin{multline}\label{eqetacc}
  \eta_c+\frac{\gamma}{2}\Big(\frac{t}{2N\sqrt{\gamma}}q^{1/2}(\eta_c)-\frac{\sigma+s}{2N}\Big)^2\nabla q(\eta_c)=-\frac{y}{2t}-\frac{\gamma^{3/2}}{2t}\frac{\nabla q(\eta_c)}{2q^{1/2}(\eta_c)}\\
  {}\times \Big[\frac{\sigma^3}{3}+\sigma\frac{x}{\gamma}+\frac{s^3}{3}+s\frac{a}{\gamma}-\frac{(s+\sigma)}{3}\Big(\frac{t}{2N\sqrt{\gamma}}q^{1/2}(\eta_c)-\frac{\sigma+s}{2N}\Big)^2\Big].
\end{multline}
It follows that $\eta_c$ is a function of $(s+\sigma)$ and $\frac{\sigma^3}{3}+\sigma\frac{x}{\gamma}+\frac{s^3}{3}+s\frac{a}{\gamma}$ and writing the last term under the form $\frac{(s+\sigma)^3}{3}-4(s+\sigma)\Big((s+\sigma)^2-(s-\sigma)^2\Big)
+(s+\sigma)\frac{(x+a)}{2\gamma}+(\sigma-s)\frac{(x-a)}{2\gamma}$
allows to conclude. Taking now the derivative with respect to $\sigma$ in \eqref{eqetacc} yields
\begin{equation}\label{etacderivsmall}
\partial_{\sigma}\eta_c\Big(\mathbb{I}_{d-1}+O(\gamma)+O\Big(\frac {\gamma^{\frac 32}}{t}\Bigr)\Big) %
=\frac{\gamma\nabla q(\eta_c)}{2N}+\frac{\gamma^{\frac 32}\nabla q(\eta_c)}{4 t q^{\frac 12}(\eta_c)}\Big[\sigma^2+\frac{x}{\gamma}+\frac{{\alpha_c^{\frac 1 2 }}}{3}\Big(\frac{s+\sigma}{N}-{\alpha_c^{\frac 1 2}}\Big)\Big],
\end{equation}
where the second and third terms in brackets in the first line of \eqref{etacderivsmall} are smooth, bounded functions of $\eta_c, \frac{t}{2N\sqrt{\gamma}},(s+\sigma)$ and $\frac{\sigma^3}{3}+\sigma\frac{x}{\gamma}+\frac{s^3}{3}+s\frac{a}{\gamma}$ with coefficients $\gamma$ and $\gamma^{3/2}/t$, respectively. Let first $N\geq 2$, then using $\frac{t}{2N\sqrt{\gamma}}\in [1/M,M]$ we find $\gamma^{3/2}/t\sim \gamma/N$ and therefore $\partial_{\sigma}\eta_c=O(\gamma^{3/2}/t)$. In the same way we obtain $\partial_{s}\eta_c=O(\gamma^{3/2}/t)$. 
Let now $N=1$, then $\gamma^{3/2}/t\gtrsim \gamma$ whenever the phase may be stationary, and therefore we still find $\partial_{\sigma}\eta_c=O(\gamma^{3/2}/t)$ and $\partial_{s}\eta_c=O(\gamma^{3/2}/t)$. Therefore, ${\Theta}_1:=\frac{t}{\gamma^{3/2}}\partial_{\sigma}\eta_c$ (and ${\Theta}_2:=\frac{t}{\gamma^{3/2}}\partial_{s}\eta_c$), respectively)  is a smooth and uniformly bounded vector valued function depending on 
 $\sigma+s,\sigma^2+\frac{x}{\gamma},\sigma^3/3+\sigma\frac{x}{\gamma}+s^3/3+s\frac{a}{\gamma}$ and $(\frac{t}{2N\sqrt{\gamma}},\frac{y}{2t},\gamma)$ (and, respectively, on $\sigma+s,s^2+\frac{a}{\gamma},\sigma^3/3+\sigma\frac{x}{\gamma}+s^3/3+s\frac{a}{\gamma}$ and $(\frac{t}{2N\sqrt{\gamma}},\frac{y}{2t},\gamma)$). In the following  we write $\Theta_j=\Theta_j\Big(\sigma,s,\frac{t}{2N\sqrt{\gamma}},\frac{x}{\gamma},\frac{a}{\gamma},\frac{y}{2t},\gamma\Big)$ for $j\in\{1,2\}$.
\end{proof}
\begin{lemma}\label{lemalphacetacalphac}
For all $N\geq 1$, the critical point $\alpha_c$ is such that
\begin{equation}\label{eq:alpha_c21}
\sqrt{\alpha_c}=
\frac{t}{2N\sqrt{\gamma}}q^{1/2}\Big(\eta^0_c\Big)-\frac{\sigma}{2N}(1-\gamma\mathcal{E}_1)-\frac{s}{2N}(1-\gamma\mathcal{E}_2),
\end{equation}
where $\mathcal{E}_j$ are smooth, uniformly bounded functions:
\begin{gather}\label{f-1}
\mathcal{E}_1:=<\int_0^1{\Theta}_1\Big(o\sigma,os,\frac{t}{2N\sqrt{\gamma}},\frac{x}{\gamma},\frac{a}{\gamma},\frac{y}{2t},\gamma\Big)do,\int_0^1\frac{\nabla q}{2q^{1/2}}(o\eta^0_c+(1-o)\eta_c)do>,\\
\label{f-2}
\mathcal{E}_2:<=\int_0^1{\Theta}_2\Big(o\sigma,os,\frac{t}{2N\sqrt{\gamma}},\frac{x}{\gamma},\frac{a}{\gamma},\frac{y}{2t},\gamma\Big)do,\int_0^1\frac{\nabla q}{2q^{1/2}}(o\eta^0_c+(1-o)\eta_c)do>.
\end{gather}
\end{lemma}
\begin{proof}
Rewrite \eqref{eq:alphac} as %
$\sqrt{\alpha_c}=\frac{t}{2N\sqrt{\gamma}}q^{1/2}\Big(\eta^0_c\Big)-\frac{(\sigma+s)}{2N}+\frac{t}{2N\sqrt{\gamma}}(q^{1/2}(\eta_c)-q^{1/2}(\eta^0_c))$.
As we have $\eta_c-\eta^0_c=\frac{\gamma^{3/2}}{t}<(\sigma,s),\int_0^1 (\Theta_1,\Theta_2)\Big(o\sigma,os,\frac{t}{2N\sqrt{\gamma}},\frac{x}{\gamma},\frac{a}{\gamma},\frac{y}{2t},\gamma\Big)do>$ and 
\begin{equation}\label{etacvsseta0c}
q^{1/2}(\eta_c)-q^{1/2}(\eta^0_c)=(\eta_c-\eta^0_c)\int_0^1\Big(\frac{\nabla q}{2q^{1/2}}\Big)(o\eta^0_c+(1-o)\eta_c)do,
\end{equation}
defining $\mathcal{E}_j$ as in \eqref{f-1} and \eqref{f-2} yields \eqref{eq:alpha_c21}.
\end{proof}
\begin{cor}\label{coruhgam}
There exist $C\neq 0$ (independent of $h,a,\gamma$), $\tilde \psi\in C^{\infty}_0([\frac 14,2])$ with $\tilde \psi=1$ on the support of $\psi$ such that
\begin{gather}
G_{h,\gamma}(t,x,y)=\frac{C}{h^{d}}\Big(\frac ht\Big)^{(d-1)/2}\tilde\psi\Big(\frac{|y|}{2t}\Big)\sum_{\frac{t}{\sqrt{\gamma}}\sim N\lesssim \frac{1}{\sqrt{\gamma}}} V_{N,h,\gamma}(t,x,y) +O(h^{\infty})\,,\\
V_{N,h,\gamma}(t,x,y)=\frac{\gamma^2}{h}\frac{1}{\sqrt{\lambda_{\gamma}N}}\int e^{\frac ih \phi_{N,a,\gamma}(\sigma,s,t,x,y)}\varkappa(\sigma,s,t,x,y;h,\gamma,1/N)d\sigma ds\,,
\end{gather}
with phase $\phi_{N,a,\gamma}(\sigma,s,t,x,y)=\Phi_{N,a,\gamma}(\eta_c,\alpha_c,\sigma,s,t,x,y)$ and symbol $\varkappa(\cdot;h,\gamma,1/N)$.
\end{cor}
This immediately follows from stationary phase in $\alpha$ and $\eta$,  with a leading order term for $\varkappa$ being $q(\eta_{c})\psi(|\eta_{c}|)\psi_2(\alpha_{c})e^{iNB(q^{1/2}(\eta_{c})\lambda_{\gamma}\alpha_{c}^{3/2})}$.
\begin{rmq}
This main contribution for the symbol $\varkappa(\cdot;h,\gamma,1/N)$ has an harmless dependence on the parameters $h,a,\gamma, 1/N$, as $\varkappa(\cdot,h,\gamma,1/N)$ reads as an asymptotic expansion with small parameters $(\lambda_{\gamma}N)^{-1}=h/(N\gamma^{3/2})$ in $\alpha$ and $(h/t)$ in $\eta$, and all terms in the expansions are smooth functions of $\alpha_c, \eta_c$.
\end{rmq}
\begin{rmq}\label{rmkchi}
From Remark \ref{rmqA}, we may introduce cut-offs $\chi(\sigma/(2\sqrt{\alpha_c}))$ and $\chi(s/(2\sqrt{\alpha_c}))$, supported for $|(\sigma,s)|\leq 2\sqrt{\alpha_c}$ in $V_{N,h,\gamma}$ without changing its contribution modulo $O(h^{\infty})$. 
\end{rmq}
We are left with integrals with respect to the variables $s,\sigma$ to estimate $\|V_{N,h,\gamma}(t,\cdot)\|_{L^{\infty}}$.  We first compute higher order derivatives of the critical value $\Phi_{N,a,\gamma}(\eta_c,\alpha_c,s,\sigma,t,y,x)$, with
\begin{gather}\label{derphisig1}
\partial_{\sigma}\Big(\Phi_{N,a,\gamma}(\eta_c,\alpha_c,s,\sigma,\cdot)\Big)=\gamma^{3/2}q^{1/2}(\eta_c)(\sigma^2+\frac{x}{\gamma}-\alpha_c), \\ 
\label{derphis1}
\partial_{s}\Big(\Phi_{N,a,\gamma}(\eta_c,\alpha_c,s,\sigma,\cdot)\Big)=\gamma^{3/2}q^{1/2}(\eta_c)(s^2+\frac{a}{\gamma}-\alpha_c).
\end{gather}
Higher order derivatives of $\phi_{N,a,\gamma}(\sigma,s,\cdot):=\Phi_{N,a,\gamma}(\eta_c,\alpha_c,\sigma,s,\cdot)$ involve derivatives of critical points $\alpha_c,\eta_c$ with respect to $\sigma,s$ :
\begin{align}\label{secdersig}
\partial^2_{\sigma,\sigma}\Big(\Phi_{N,a,\gamma}(\eta_c,\alpha_c,\cdot)\Big)&=\partial_{\sigma}\eta_c\frac{\nabla q(\eta)}{2q(\eta)}|_{\eta=\eta_c}\partial_{\sigma}\phi_{N,a,\gamma}
+\gamma^{3/2}q^{1/2}(\eta_c)(2\sigma-2\sqrt{\alpha_c}\partial_{\sigma}\sqrt{\alpha_c}),\\
\label{secders}
\partial^2_{s,s}\Big(\Phi_{N,a,\gamma}(\eta_c,\alpha_c,\cdot)\Big)&=\partial_{s}\eta_c\frac{\nabla q(\eta)}{2q(\eta)}|_{\eta=\eta_c}\partial_{s}\phi_{N,a,\gamma}
+\gamma^{3/2}q^{1/2}(\eta_c)(2s-2\sqrt{\alpha_c}\partial_{s}\sqrt{\alpha_c}),\\
\label{secderssig}
\partial^2_{\sigma,s}\Big(\Phi_{N,a,\gamma}(\eta_c,\alpha_c,\cdot)\Big)&=\partial_{\sigma}\eta_c\frac{\nabla q(\eta)}{2q(\eta)}|_{\eta=\eta_c}\partial_{s}\phi_{N,a,\gamma}-\gamma^{3/2}q^{1/2}(\eta_c)(2\sqrt{\alpha_c}\partial_{\sigma}\sqrt{\alpha_c}),
\end{align}
and therefore, when $\partial_{s}\phi_{N,a,\gamma}=\partial_{\sigma}\phi_{N,a,\gamma}=0$, we have
\begin{align*}
\partial^2_{\sigma,\sigma}\phi_{N,a,\gamma}(\eta_c,\alpha_c,s,\sigma,\cdot)|_{\partial_{s}\phi_{N,a,\gamma}=\partial_{\sigma}\phi_{N,a,\gamma}=0}&=
2\gamma^{3/2}q^{1/2}(\eta_c)(\sigma-\sqrt{\alpha_c}\partial_{\sigma}\sqrt{\alpha_c}),\\
\partial^2_{s,s}\phi_{N,a,\gamma}(\eta_c,\alpha_c,s,\sigma,\cdot)|_{\partial_{s}\phi_{N,a,\gamma}=\partial_{\sigma}\phi_{N,a,\gamma}=0}&=
2\gamma^{3/2}q^{1/2}(\eta_c)(s-\sqrt{\alpha_c}\partial_{s}\sqrt{\alpha_c}),\\
\partial^2_{\sigma,s}\phi_{N,a,\gamma}(\eta_c,\alpha_c,s,\sigma,\cdot)|_{\partial_{s}\phi_{N,a,\gamma}=\partial_{\sigma}\phi_{N,a,\gamma}=0}&=-2\gamma^{3/2}q^{1/2}(\eta_c)\sqrt{\alpha_c}\partial_{\sigma}\sqrt{\alpha_c}.
\end{align*}
\begin{rmq}\label{rmqcritssigderalphac}
At critical points we have $\partial_{\sigma}\sqrt{\alpha_c}=\partial_{s}\sqrt{\alpha_c}$ : derivatives of $\alpha_c$ depend on $\eta_c$ that solves \eqref{eq:etac} ; from \eqref{eq:etac}, $\partial_{\sigma}\eta_c$ (and $\partial_{s}\eta_{c}$, respectively) depend upon $(s+\sigma)$, $\sigma^2+\frac{x}{\gamma}$ and $\sigma^3/3+\sigma\frac{x}{\gamma}+s^3/3+s\frac{a}{\gamma}$ (and upon $(s+\sigma)$, $s^2+\frac{a}{\gamma}$ and $\sigma^3/3+\sigma\frac{x}{\gamma}+s^3/3+s\frac{a}{\gamma}$); at the critical points $\sigma,s$ we have $\sigma^2+\frac{x}{\gamma}=s^2+\frac{a}{\gamma}=\alpha_c$ and we find $\partial_{\sigma}\eta_c=\partial_{s}\eta_c$.
\end{rmq}

\subsubsection{"Tangential" waves $a\in [\frac18 \gamma, 8\gamma]$}\label{secttangcas}
We abuse notations  and write $G_{h,a}=G_{h,\gamma\sim a}$, $\lambda=a^{3/2}/h=\lambda_{\gamma\sim a}$ and from Corollary \ref{coruhgam}, with $\phi_{N,a}(\sigma,s,t,x,y)=\Phi_{N,a,a}(\eta_c,\alpha_c,\sigma,s,t,x,y)$,
\begin{gather}\label{uhsumNa}
G_{h,a}(t,x,y)=\frac{C}{h^{d}}\Big(\frac ht\Big)^{(d-1)/2}\tilde\psi\Big(\frac{|y|}{2t}\Big)\sum_{\frac{t}{\sqrt{a}}\sim N\lesssim \frac{1}{\sqrt{a}}} V_{N,h,a}(t,x,y) +O(h^{\infty})\,,\\
\label{defVNha}
V_{N,h,a}(t,x,y)=\frac{a^2}{h}\frac{1}{\sqrt{\lambda N}}\int e^{\frac ih \phi_{N,a}(\sigma,s,t,x,y)}\varkappa(\sigma,s,t,x,y,h,a,1/N)d\sigma ds\,.
\end{gather}
\begin{rmq}
From Remark \ref{rmqB}, only values $N\lesssim \lambda$ are of interest. It turns out that one needs to separate the cases $N<\lambda^{1/3}$ and $\lambda^{1/3}\lesssim N$. Fix $t$ and set $T=\frac{t}{\sqrt{a}}$ : if $\lambda^{1/3}\lesssim T\sim N$, then $\phi_{N,a}$ behaves like the phase of a product of two Airy functions and can be bounded using mainly their respective asymptotic behavior. 
When $N\sim T\lesssim \lambda^{1/3}$, $\phi_{N,a}$ may have degenerate critical points up to order $3$. We claim that {for any} $t$ such that $T:=\frac{t}{\sqrt{a}}\ll \lambda^{1/3}$ and {for any} $N\sim T$ there exists a locus of points $\mathcal{Y}_{N}(T):=\{Y\in\mathbb{R}^{d-1}| K_a(\frac{Y}{4N},\frac{T}{4N})=1\}$, where $K_a$ is the smooth function to be defined in \eqref{defKa} such that, for all $Y\in\mathcal{Y}_N(T)$ we have $\|G_{h,a}(t,\cdot)\|_{L^{\infty}(\Omega)}= |G_{h,a}(t,a,a,y)||_{y\in\sqrt{a}\mathcal{Y}_{N}(t/\sqrt{a})}\sim \frac{1}{h^d}(\frac ht)^{(d-1)/2}a^{1/4}(\frac ht)^{1/4}$, for all $(ht)^{1/2}\lesssim a\lesssim \varepsilon_0$. Optimality follows.
\end{rmq}
\begin{rmq}
When dealing with the wave flow in \cite{ILP3}, a parametrix is also obtained as a sum of reflected waves: due to finite speed of propagation, the main contribution at fixed $t$ is provided by waves located between the $(N-1)th$ and $(N+1)th$ reflections, where $N=[\frac{t}{\sqrt{a}}]$. For each $N\ll \lambda^{1/3}$, the worst bound occurs at a unique time $t_N$, at $x=a$ and for a unique $y_N$. For the Schr\"odinger flow, for any $t/\sqrt{a}\ll \lambda^{1/3}$ and any $N\sim t/\sqrt{a}$, 
$|V_{N,h,a}(t,a,y)||_{y\in \sqrt{a}\mathcal{Y}_N(t/\sqrt{a})}\sim \|G_{h,a}(t,\cdot)\|_{L^{\infty}}$, where $\mathcal{Y}_N(t/\sqrt{a})\cap \mathcal{Y}_{N'}(t/\sqrt{a})=\emptyset$ for $N\neq N'$.
\end{rmq}
We denote $\alpha^0_c=\alpha_c|_{s=\sigma=0}$ obtained in \eqref{eq:alpha_c21}. Recall from Lemma \ref{lemalphacetac} (with $\gamma$ replaced by $a$), that
$\eta^0_c=-\frac{y}{2t}+a\Theta(\frac{y}{2t},\frac{t}{2N\sqrt{a}},a)$ is a smooth function of $(\frac{y}{2t},\frac{t}{2N\sqrt{a}},a)$, hence so is $\sqrt{\alpha^0_c}=\frac{t}{2N\sqrt{a}}q^{1/2}(\eta^0_c)$.
Let $T=t/\sqrt{a}$, $Y=y/\sqrt{a}$ and define $K_a(\frac{Y}{4N},\frac{T}{2N})=\sqrt{\alpha^0_c(\frac{Y}{4N}\frac{2N}{T},\frac{T}{2N},a)}$. Then $K_a$ is smooth in all variables and
\begin{equation}\label{defKa}
K_a(\frac{Y}{4N},\frac{T}{2N})=\frac{|Y|}{4N}q^{1/2}\Big(-\frac{Y}{|Y|}+a\frac{T}{2N}\frac{4N}{|Y|} \Theta(\frac{Y}{4N}\frac{2N}{T},\frac{T}{2N},a)\Big)\,.
\end{equation}
\begin{prop}
\label{dispNgrand}
For $\lambda^{1/3}\lesssim T\sim N$, $\frac xa\leq 1$, we have
  \begin{equation}
    \label{eq:1ff}
       \left| V_{N,h,a}(t,x,y)\right|\lesssim  \frac {h^{1/3}} {(N/\lambda^{1/3})^{1/2} +\lambda^{1/6}\sqrt{4N}|K_a(\frac{Y}{4N},\frac{T}{2N})-1|^{1/2}}\,.
  \end{equation}
\end{prop}

\begin{prop}
\label{dispNpetitloin}
For $1\leq N<\lambda^{1/3}$ and $|K_a(\frac{Y}{4N},\frac{T}{2N})-1|\gtrsim 1/N^2$, $\frac xa\leq 1$ we have
  \begin{equation}
    \label{eq:2ff}
       \left| V_{N,h,a}(t,x,y)\right| \lesssim\frac{h^{1/3}}{(1+2N|K_a(\frac{Y}{4N},\frac{T}{2N})-1|^{1/2})}\,.
  \end{equation}
\end{prop}

\begin{prop}
\label{dispNpetitpres}
For $1\leq N<\lambda^{1/3}$ and $|K_a(\frac{Y}{4N},\frac{T}{2N})-1|\leq\frac{1}{4N^2}$, $\frac xa\leq 1$ we have
  \begin{equation}
    \label{eq:2hh}
       \left| V_{N,h,a}(t,x,y)\right| \lesssim \frac{h^{1/3}}{(N/\lambda^{1/3})^{1/4}+N^{1/3}|(K_a(\frac{Y}{4N},\frac{T}{2N})-1)|^{1/6}}\,.
  \end{equation}
Moreover, at $x=a$ and $K_a(\frac{Y}{4N},\frac{T}{2N})=1$ we have $ \left| V_{N,h,a}(t,a,y)\right| \sim \frac{h^{1/3}}{((N/\lambda^{1/3})^{1/4}}$.
\end{prop}
We postpone the proofs of Propositions \ref{dispNgrand}, \ref{dispNpetitloin} and \ref{dispNpetitpres} to Section \ref{sectproofsprops} and we complete the proof of Theorem \ref{thmdispSchrodinger} in the case $(ht)^{1/2}\lesssim a\sim \gamma \leq \varepsilon_0<1$. 
Let therefore $\sqrt{a} \lesssim t\lesssim 1$ be fixed and let $N_t\geq 1$ be the unique positive integer such that $T=\frac{t}{\sqrt{a}}\leq N_t<\frac{t}{\sqrt{a}}+1=T+1$, hence $N_t=[T]$, where $[T]$ denotes the integer part of $T$. If $N_t$ is bounded then the number of $V_{N,h,a}$ with $N\sim N_t$ in the sum \eqref{uhsumNa} is also bounded and we can easily conclude adding the (worst) bound \eqref{dispNpetitpres} a finite number of times.
Assume $N_t\geq 2$ is large enough. We introduce the following notation: for $k\in \mathbb{Z}$ let
$I_{N_t,k}:=[4(N_t+k)-2,4(N_t+k)+2)$.
As $\alpha_c,\eta_c\in [\frac 12,\frac 32]$ and $\sqrt{\alpha_c}=\frac{T}{2N}q^{1/2}(\eta_c)-\frac{(\sigma+s)}{2N}$ with $|(\sigma,s)|\leq 2\sqrt{\alpha_c}$ on the support of $\chi$ (see Remark \ref{rmkchi}), we deduce (using \eqref{suppcondt}) that, for $M$ defined in \eqref{defM}, we have $2N\in [\frac{T}{M},MT]\subset [\frac{N_t}{M},M(N_t+1)]$. Using \eqref{uhsumNa}, we then bound $G_{h,a}(t,\cdot)$ as follows
\[
\|G_{h,a}(t,.)\|_{L^{\infty}(0\leq x\leq a, y)}\lesssim \frac{1}{h^d}\Big(\frac ht\Big)^{(d-1)/2} \sup_{x\leq a,y}\sum_{N_t/M\leq 2N\leq M(N_t+1)} |V_{N,h,a}(t,x,y)|.
\]
It will follow from the proof of Propositions \ref{dispNpetitpres} that the worst dispersive bounds for $V_{N,h,a}$ occurs at $x=a$ (when $\phi_{N,a}$ may have a critical point of order $3$). Therefore, we will seek for bounds for $G_{h,a}$ especially at $x=a$. 

For a fixed $y$ on the support of $\tilde\psi\Big(\frac{|y|}{2t}\Big)$ we let $Y=\frac{y}{\sqrt{a}}$, then $\frac 14\leq \frac{|Y|}{2T}\leq 2$, and therefore $|Y|\in [T/2,4T]\subset [N_t/2,4(N_t+1)]$. Using \eqref{defKa} and the fact that $q^{1/2}$ is homogeneous of order $1$, it follows that $K_a(\frac{Y}{4N},\frac{T}{2N})$ is close to $1$ when $q^{1/2}(-Y+2aT\Theta(\frac{Y}{2T},\frac{T}{2N},a))$ is sufficiently close to $4N$. As $2<N_t\leq T\leq 1/\sqrt{a}$, $|Y|/T\in [1/2,4]$, $\Theta$ is bounded and $0<a\leq \varepsilon_0$ is small, then, for $m_0$ and $M_0$ defined in \eqref{limitshalfq},
\[
q^{1/2}\Big(-Y+2aT\Theta(\frac{Y}{2T},\frac{T}{2N},a)\Big)
 \subset [N_t(m_0-\varepsilon_0)/2,4(N_t+1)(M_0+\varepsilon_0)].
\]
Setting $k_1=-N_t(1-(m_0-\varepsilon_0)/8),\quad  k_2=(N_t+1)(M_0+\varepsilon_0-1)-1$,
we have $N_t+k\sim N_t$ and $[N_t(m_0-\varepsilon_0)/2,4(N_t+1)(M_0+\varepsilon_0)]\subset \cup_{k_1\leq k\leq k_2} I_{N_t,k}$. Let 
\[
\tilde I_{N_t,k}:=(4(N_t+k)-1,4(N_t+k)+1)\subset I_{N_t,k}.
\] 
Write
\begin{multline}\label{estimsupxyNtk}
\sup_{x,y}\sum_{N_t/M\leq 2N\leq M(N_t+1)} |V_{N,h,a}(t,x,y)|\\=\sup_{k_1\leq k\leq k_2} \sup_{q^{1/2}\Big(-Y+2aT\Theta(\frac{Y}{2T},\frac{T}{2N},a)\Big)\in I_{N_t,k}}\sum_{N_t/M\leq 2N\leq M(N_t+1)} |V_{N,h,a}(t,a,y)|\\
\geq \sup_{k_1\leq k\leq k_2} \sup_{q^{1/2}\Big(-Y+2aT\Theta(\frac{Y}{2T},\frac{T}{2N},a)\Big)\in \tilde I_{N_t,k}}\sum_{N_t/M\leq 2N\leq M(N_t+1)} |V_{N,h,a}(t,a,y)|.
\end{multline}
\begin{prop}\label{propsumNgrand}
There exists $C>0$ (independent of $h,a$) such that, if $N_t:=[\frac{t}{\sqrt{a}}]\gg \lambda^{1/3}$, 
\[
\|G_{h,a}(t,\cdot)\|_{L^{\infty}(\Omega_d)}\leq \frac{C}{h^{d}}\Big(\frac ht\Big)^{(d-1)/2}\Big(\frac{ht}{a}\Big)^{1/2}\,.
\]
\end{prop}

\begin{proof}
If $\lambda^{1/3}\ll N_t$, then $N_t+k\gg\lambda^{1/3}$ for all $k\in [k_1,k_2]$ and we estimate the $L^{\infty}$ norms of $G_{h,a}(t,\cdot)$ using the first equality in \eqref{estimsupxyNtk} and Proposition \ref{dispNgrand}:
if $k_y\in [k_1,k_2]$ is such that $q^{1/2}(-Y)\in I_{N_t,k_y}$, then,  $4NK_a(\frac{Y}{4N},\frac{T}{2N})=q^{1/2}\Big(-Y+2aT\Theta(\frac{Y}{2T},\frac{T}{2N},a)\Big)\in \cup_{|k'-k_y|\leq 1} I_{N_t,k'}$ (using $a$ small) and therefore the second line in \eqref{estimsupxyNtk} can be (uniformly) bounded as follows
\begin{multline}\label{boundwhenNislarge}
 \sup_{k_1\leq k\leq k_2}\sup_{4NK_a(\frac{Y}{4N},\frac{T}{2N})\in I_{N_t,k}} \sum_{2N\in[N_t/M, M(N_t+1)]} |V_{N,h,a}(t,a,y)| \\
 \leq
 \sup_{|k'-k_y|\leq 1}\sup_{4NK_a(\frac{Y}{4N},\frac{T}{2N})\in I_{N_t,k'}} \sum_{2N\in[N_t/M, M(N_t+1)]} |V_{N,h,a}(t,a,y)| \\ 
 \leq   \sup_{4NK_a(\frac{Y}{4N},\frac{T}{2N})\in \cup_{|k'-k_y|\leq 1}I_{N_t,k'}} \sum_{2N\in[N_t/M, M(N_t+1)}\frac {h^{1/3}} {(N/\lambda^{1/3})^{1/2} +\lambda^{1/6}|4NK_a(\frac{Y}{4N},\frac{T}{2N})-4N|^{1/2}}.
 \end{multline}
As $4NK_a(\frac{Y}{4N},\frac{T}{2N})\in \cup_{|k'-k_y|\leq 1} I_{N_t,k'}$, we find, for $N=N_t+k_y+j$ and $|j|\geq 2$, that $\Big|4NK_a(\frac{Y}{4N},\frac{T}{2N})-4N\Big|\geq |j|-1$,
 and therefore the last line in \eqref{boundwhenNislarge} can be bounded by
 \begin{equation}\label{estimsumNgrandj}
 \frac{h^{\frac 13}}{(N_t+k_y)^{\frac 1 2 }} \Big(3\lambda^{\frac 16}+ \sum_{|N-(N_t+k_y)|=|j|\geq 2}\frac {\lambda^{\frac 16}} {(1+j/(N_t+k_y))^{1/2}+\lambda^{\frac 13}|(|j|-1)/(N_t+k_y)|^{\frac 1 2}}\Big).
\end{equation}
The sums over $N=N_t+k_y\pm(j+1)$, $j\geq 1$, read as
\begin{multline}
\frac{h^{1/3}(N_t+k_y)^{1/2}}{\lambda^{1/6}(N_t+k_y)}\sum_{N= N_t+k_y\pm(j+1), j\geq 1}\frac {1} {(1\pm (j+1)/(N_t+k_y))^{1/2}\lambda^{-1/3} +|j/(N_t+k_y)|^{1/2}}\\
\leq h^{1/3}\frac{(N_t+k_y)^{1/2}}{\lambda^{1/6}}\sum_{\pm}\int_0^{1-\frac{1+N_t/(2M)}{N_t+k_y}}\frac{dx}{\sqrt{x}+\lambda^{-1/3}(1\pm(N_t+k_y)^{-1}\pm  x)^{1/2}},
\end{multline}
where the last integral is taken on $[0,1-\frac{1+N_t/(2M)}{N_t+k_y}]$ as $N=N_t+k_y\pm (j+1)\geq \frac{N_y}{2M}$. 
As $k_y\geq k_1$, we have $N_t+k_y\geq N_t(1+(m_0-\varepsilon_0)/8)$ and using \eqref{defM},
$\frac{N_t}{2M(N_t+k_y)}\leq \frac{4}{M(m_0-\varepsilon_0)}\leq \frac{1}{\sqrt{3/2}}$.
Both integrals (with $\pm$ signs) are bounded by $\frac 12$, so the contribution coming from the sum over $|N-(N_t+k_y)|\geq 2$ in \eqref{estimsumNgrandj} is $h^{1/3}(N_t+k_y)^{1/2}/\lambda^{1/6}$. As $N_t+k_y\leq (N_t+1)(M_0+\varepsilon_0-1)$ where $M_0$ is fixed, depending only on $q$, and $N_t\in[\frac{t}{\sqrt{a}}-1,\frac{t}{\sqrt{a}}]$, we obtain 
\begin{equation}\label{kbounds}
 \sup_{4NK_a(\frac{Y}{4N},\frac{T}{2N})\in \cup_{|k'-k_y|\leq 1}I_{N_t,k'}}\sum_{2N\in [N_t/M,M(N_t+1)]} |V_{N,h,a}(t,a,y)|\leq \sqrt{M_0} h^{1/3}\Big(\frac{t/\sqrt{a}}{\lambda^{1/3}}\Big)^{1/2}\lesssim \Big(\frac{ht}{a}\Big)^{1/2}\,,
\end{equation}
which concludes the proof.\end{proof}

We introduce one more notation. 
If $y$ is such that $q^{1/2}\Big(-Y+2aT\Theta(\frac{Y}{2T},\frac{T}{2N},a)\Big)\in \tilde I_{N_t,k}$ for some $k_1\leq k\leq k_2$, then $k$ is unique and we denote it $k^{\#}_y$. If $2(N_t+k^{\#}_y)\in [N_t/M, M(N_t+1)]$, we can either have $\lambda^{1/3}\lesssim N_t+k^{\#}_y$, or $N_t+k^{\#}_y< \lambda^{1/3}$.
\begin{rmq}\label{rmqprop5}
When $N_t+k^{\#}_y< \lambda^{1/3}$, Proposition \ref{dispNpetitpres} may apply only for $N=N_t+k^{\#}_y$, as for $k^{\#}_y\neq k\in [k_1,k_2]$ and $n=N_t+k$ we must have 
\begin{multline}
\Big|q^{1/2}\Big(-Y+2aT\Theta(\frac{Y}{2T},\frac{T}{2N},a)\Big)-4n\Big|\geq 4|n-(N_t+k^{\#}_y)|\\
-\Big|q^{1/2}\Big(-Y+2aT\Theta(\frac{Y}{2T},\frac{T}{2N},a)\Big)-4(N_t+k^{\#}_y)\Big|\geq 3 \gg \frac{1}{n}.
\end{multline}
\end{rmq}

\begin{prop}\label{proptangmain}
There exists $C>0$ (independent of $h,a$) such that, if $N_t:=[\frac{t}{\sqrt{a}}]\ll \lambda^{1/3}$, 
\begin{equation}\label{ptmaxGha}
\|G_{h,a}(t,\cdot)\|_{L^{\infty}(\Omega_d)}\sim \frac{C}{h^{d}}\Big(\frac ht\Big)^{(d-1)/2}\Big(\frac{ha}{t}\Big)^{1/4}.
\end{equation}
\end{prop}
\begin{proof}
If $y$ is such that $q^{1/2}(-Y)\in I_{N_t,k_y}$ for $k_y\in [k_1,k_2]$, then, using $a\leq \varepsilon_0$,
\begin{multline}
  \Big|q^{1/2}\Big(-Y+2aT\Theta(\frac{Y}{2T},\frac{T}{2N},a)\Big)-4n\Big|\geq 4|n-(N_t+k_y)|\\
  {}-\Big|q^{1/2}\Big(-Y+2aT\Theta(\frac{Y}{2T},\frac{T}{2N},a)\Big)-4(N_t+k_y)\Big|
\end{multline}
for all $n\neq N_t+k_y$; the second term in the right hand side is smaller than $2$, while the first one is at least $4$; therefore the assumption of Proposition \ref{dispNpetitpres} cannot hold for $n\neq N_t+k_y$. For all such $n$ we then use Proposition \ref{dispNpetitloin},
\begin{multline}\label{smalltermsinthesum}
 \sup_{q^{1/2}(-Y))\in I_{N_t,k_y}}\sum_{2n\in [N_t/M,M(N_t+1)], n \neq N_t+k_y} |V_{n,h,a}(t,a,y)|\\
 \lesssim  h^{1/3} \sum_{2n\in [N_t/M,M(N_t+1)], n\neq N_t+k_y}\frac {1} {(1 +|n(q^{1/2}(-Y+2aT\Theta(\frac{Y}{2T},\frac{T}{2n},a))-4n)|^{1/2})}\\
 \lesssim h^{1/3}\sum_{n=N_t+k_y+j, 1\leq |j|\lesssim N_t}\frac {1} {(1 +(N_t+k_y+j)^{1/2}|j|^{1/2})}\\
 \leq h^{1/3}\sum_{\pm}\int_0^{1-\frac{1+N_t/(2M)}{N_t+k_y}}\frac{dx}{x^{1/2}(1\pm x)^{1/2}+(N_t+k_y)^{-1}},
\end{multline}
where the last two integrals are uniform bounds for the sum over $N< N_t+k_y$ and $N>N_t+k_y$, respectively; when $N>N_t+k_y$, the integral over $[0,1]$ is bounded by a uniform constant ; when $N< N_t+k_y$, write $x=\sin^2 \theta$, $\theta\in [0,\pi/2)$, therefore $1-x=\cos^2\theta$, $dx=2\sin\theta\cos\theta$ : the corresponding integral is also bounded by at most $\pi$.

We are left with $N=N_t+k_y$.
If $q^{1/2}\Big(-Y+2aT\Theta(\frac{Y}{2T},\frac{T}{2N},a)\Big)\notin \tilde I_{N_t,k_y}$, then we use again Proposition \ref{dispNpetitloin}. If, on the contrary, $q^{1/2}\Big(-Y+2aT\Theta(\frac{Y}{2T},\frac{T}{2N},a)\Big)\in \tilde I_{N_t,k_y}$, then
$k_y^{\#}=k_y\in [k_1,k_2]$ and we may be able to apply Proposition \ref{dispNpetitpres} with $N=N_t+k^{\#}_y$ if moreover the following holds: $\Big|q^{1/2}\Big(-Y+2aT\Theta(\frac{Y}{2T},\frac{T}{2N},a)\Big)-4N\Big|\lesssim \frac{1}{N}$.
We then have
\begin{align*}
 \sup_{q^{\frac 12}(-Y))\in I_{N_t,k_y}}|V_{N_t+k_y,h,a}(t,a,y)| & \lesssim \frac{h^{\frac 13}}{(N/\lambda^{\frac 13})^{\frac 14}}+\frac {h^{\frac 1 3}} {(1 +|N(q^{\frac 12}(2aT\Theta(\frac{Y}{2T},\frac{T}{2N},a)-Y)-4N)|^{\frac 12})}\\
& \lesssim\Big(\frac{ha}{t}\Big)^{1/4}+ h^{1/3}\,.
\end{align*}
As for $N_t\sim \frac{t}{\sqrt{a}}\ll \frac{\sqrt{a}}{h^{1/3}}=\lambda^{1/3}$ we have $h^{1/3}\ll \Big(\frac{ha}{t}\Big)^{1/4}$, it follows that at fixed $t$, the supremum of the sum over $V_{N,h,a}(t,x,y)$ is reached for $y$ such that $q^{1/2}(-Y+2aT\Theta(\frac{Y}{2T},\frac{T}{2N},a))=4N$ with $N=N_t+k^{\#}_y$ and at $x=a$. As the contribution from \eqref{smalltermsinthesum} in the sum over $n\neq N_t+k_y$ is $\sim h^{1/3}$, we obtain an upper bound for $G_{h,a}(t,\cdot)$. The last line of \eqref{estimsupxyNtk} and $h^{1/3}\ll \Big(\frac{ha}{t}\Big)^{1/4}$ provide a similar lower bound for $G_{h,a}$ and therefore \eqref{ptmaxGha} holds true.
\end{proof}

\begin{prop}\label{propsumNlimit}
There exists $C>0$ (independent of $h,a$) such that, if $N_t:=[\frac{t}{\sqrt{a}}]\sim \lambda^{1/3}$, 
\begin{equation}\label{ptmaxGhabis}
\|G_{h,a}(t,\cdot)\|_{L^{\infty}(\Omega_d)}\leq \frac{C}{h^{d}}\Big(\frac ht\Big)^{(d-1)/2}\Big(\Big(\frac{ha}{t}\Big)^{1/4}+ \Big(\frac{ht}{a}\Big)^{1/2}+h^{1/3}\Big).
\end{equation}
\end{prop}
\begin{rmq}
When $N_t\sim\lambda^{1/3}$ we find $a\sim h^{1/3}t$ and all the terms in brackets in the right hand side of \eqref{ptmaxGhabis} behave like $h^{1/3}$, hence $\|G_{h,a}(t,\cdot)\|_{L^{\infty}(\Omega_d)}\leq \frac{C}{h^{d}}\Big(\frac ht\Big)^{(d-1)/2}h^{1/3}$.
\end{rmq}
\begin{proof}
If $N_t\sim \lambda^{1/3}$ and $k\sim N_t$, we split according to whether $y$ is such that $N_t+k_y<\lambda^{1/3}$ or $N_t+k_y\geq \lambda^{1/3}$ and proceed as in the previous cases using Propositions \ref{dispNgrand}, \ref{dispNpetitloin} and \ref{dispNpetitpres}. As, for such $N_t$, we have $(\frac{ha}{t})^{1/4}\sim h^{1/3}\sim (\frac{th}{a})^{1/2}$, we cannot deduce the supremum to be $(\frac{ha}{t})^{1/4}$ but obtain an uniform bound $h^{1/3}$ for $\gamma_{h,a}(t)$ in the statement of Theorem \ref{thmdispSchrodinger}. 
\end{proof}

\subsubsection{Transverse waves}\label{sectransv}
Let  $\gamma>8a$ and recall $\lambda_{\gamma}:=\frac{\gamma^{3/2}}{h}$.
\begin{prop}\label{proptransvmain}
Let $t>h$ %
and $\varepsilon_0>\gamma>8a$. %
\begin{equation}\label{transwave}
\|G_{h,\gamma}(t,\cdot)\|_{L^{\infty}(x\leq a, y)}\lesssim 
\left\{ \begin{array}{l} 
\frac{1}{h^d}\Big(\frac ht \Big)^{\frac{d-1}2}\Big({\frac{th}{\gamma}}\Big)^{1/2}, \text{ if }\frac{t}{\sqrt{\gamma}}\gtrsim \lambda_{\gamma}^{1/3}, \,\\
\frac{1}{h^d}\Big(\frac ht \Big)^{\frac{d-1}2}h^{1/3}, \text{ if }\frac{a}{\gamma}\lesssim \frac{t}{\sqrt{\gamma}}\lesssim\lambda_{\gamma}^{1/3},\\
\frac{1}{h^d}\Big(\frac ht \Big)^{\frac d2},\text{ if } h<t \text{ and } \frac{t}{\sqrt{\gamma}}\leq \frac{1}{2\sqrt{3/2}M_0^{2/3}}\frac{a}{\gamma}.
 \end{array} \right.
\end{equation}
Moreover, for $h<t<a$ we have $\|G_{h,\varepsilon_0}(t,\cdot)\|_{L^{\infty}(x\leq a, y)}\lesssim \frac{1}{h^d}\Big(\frac ht \Big)^{d/2}$, while for $a\lesssim  t\leq T_0$ 
\begin{equation}\label{estimtransvN>1}
\sum_{3\leq j\leq \log\Big(\frac{\varepsilon_0}{a}\Big),\gamma_j =2^j a} \|G_{h,\gamma_j}(t,\cdot)\|_{L^{\infty}(x\leq a,y)}\lesssim 
\left\{ \begin{array}{l} 
\frac{1}{h^d}\Big(\frac ht \Big)^{\frac{d-1}2}h^{1/3}\log \frac{\varepsilon_0}{a}, \text{ if } a\lesssim  t\leq \frac{a}{h^{1/3}}(<\frac{\gamma}{h^{1/3}}),\\
 \frac{1}{h^d}\Big(\frac ht \Big)^{\frac{d-1}2}\Big[\Big(\frac{ht}{a}\Big)^{\frac 12}+h^{\frac 13}\log\frac{\varepsilon_0}{a}\Big],\text{ if }  t\geq \frac{a}{h^{\frac 13}}.
  \end{array} \right.
\end{equation}
\end{prop}
\begin{proof}
According to Proposition \ref{propcritptsalphaeta}, if $\frac{t}{\sqrt{\gamma}}\leq \frac{1}{2\sqrt{3/2}M_0^{2/3}}\frac{a}{\gamma}$ then $V_{N,h,\gamma}(t,\cdot)=O(h^{\infty})$ for all $a\leq \gamma\leq \varepsilon_0$ and all $N\geq 1$, hence $G_{h,\gamma}(t,\cdot)=V_{0,h,\gamma}(t,\cdot)$. The last line in \eqref{transwave} follows using the proof of Proposition \ref{propfreeN=0} applied to $V_{0,h,\gamma}(t,\cdot)$. If $h<t\lesssim a$, then $\frac{t}{\sqrt{\gamma}}\ll \frac{a}{\gamma}$ for all $a\leq \gamma\leq \varepsilon_0$, so $G_{h,\varepsilon_0}(t,\cdot)=\sum_{\gamma} G_{h,\gamma}(t,\cdot)=\sum_{\gamma} V_{0,h,\gamma}(t,\cdot)$ and we use Proposition \ref{propfreeN=0}. 

Let $\frac{t}{\sqrt{\gamma}}\gtrsim \frac{a}{\gamma}$. Let $T=\frac{t}{\sqrt{\gamma}}$, $Y=\frac{y}{\sqrt{\gamma}}$ and let $K_{\gamma}$ be given by \eqref{defKa} (with $a$ replaced by $\gamma$). Let $V_{N,h,\gamma}$ as in Corollary \ref{coruhgam}, then $G_{h,\gamma}(t,x,y)=\sum_{N\sim \frac{t}{\sqrt{\gamma}}}V_{N,h,\gamma}(t,x,y)$. %
For $x\leq a$, $8a<\gamma$ and $1\leq N\sim T$ the following holds
  \begin{equation}
    \label{eq:1ffgam}
      \left| V_{N,h,\gamma}(t,x,y)\right|\lesssim 
       \frac{\gamma^2}{h}\times\frac{1}{\sqrt{N\lambda_{\gamma}}}\times \frac{1}{\lambda_{\gamma}}.
  \end{equation}
Indeed, as long as $x\leq a$, we easily see that, for each $N$, the phase function of $V_{N,h,\gamma}$ has non-degenerate critical points with respect to both $\sigma, s$ and the estimates \eqref{eq:1ffgam} follow.
Summing up over $N\gtrsim \lambda_{\gamma}^{1/3}$ as in the proof of Proposition \ref{propsumNgrand} yields the first line of \eqref{transwave}. Summing over $N\lesssim \lambda_{\gamma}^{1/3}$ as in the proof of Proposition \ref{proptangmain} yields the second line of \eqref{transwave}. %

Let $a\lesssim t\lesssim a/h^{1/3}$, then $t\leq \gamma/h^{1/3}$ for all $\sup(h^{2/3-\epsilon},a)\leq \gamma\leq \varepsilon_0$. 
Summing up for $\gamma_j=2^j a$, yields the first line in \eqref{estimtransvN>1}, as $j\leq \log\frac{\varepsilon_0}{a}$. Let now $a/h^{1/3}\lesssim t\leq T_0$, then for $a\leq\gamma \lesssim th^{1/3}$, $|G_{h,\gamma}(t,\cdot)|$ is bounded by the term in the first line of \eqref{transwave}, while for $th^{1/3}\leq \gamma \leq \varepsilon_0$, $|G_{h,\gamma}(t,\cdot)|$ is bounded by the term in the second line of \eqref{transwave}. The sum for $\gamma_j=2^j a$ over $0\leq j\leq \log \frac{\sup(\varepsilon_0,th^{1/3})}{a}$ yields the first contribution in the second line of \eqref{estimtransvN>1} and the sum over $\frac{\sup(\varepsilon_0,th^{1/3})}{a}< j\leq \log{\frac{\varepsilon_0}{a}}$ yields the second one.
\end{proof}

\subsubsection{Optimality for $\sqrt{a}\leq t\ll \frac{a}{h^{1/3}}(\leq \frac{\gamma}{h^{1/3}})$} 
Write, for $1\leq \frac{t}{\sqrt{a}}\ll\lambda^{1/3}=\frac{\sqrt{a}}{h^{1/3}}$, 
\[
\|G_{h,\varepsilon_0}(t,\cdot)\|_{L^{\infty}(\Omega_d)}\geq \|G_{h,a}(t,\cdot)\|_{L^{\infty}(\Omega_d)}-\sum_{0\leq j<\frac 12 \log({\varepsilon_0}/{a}),\gamma_j =2^{2j} a} \|G_{h,\gamma_j}(t,\cdot)\|_{L^{\infty}(\Omega_d)}.
\]
From \eqref{ptmaxGha} %
we have $\|G_{h,a}(t,\cdot)\|_{L^{\infty}(\Omega_d)}\sim \frac{1}{h^d}\Big(\frac ht \Big)^{(d-1)/2} \Big(\frac{ah}{t}\Big)^{1/4}$ and from the first line of \eqref{estimtransvN>1} %
we have $\sum_{0\leq j<\frac 12 \log({\varepsilon_0}/{a}),\gamma_j =2^{2j} a} \|G_{h,\gamma_j}(t,\cdot)\|_{L^{\infty}(\Omega_d)}\leq \frac{1}{h^d}\Big(\frac ht \Big)^{(d-1)/2}h^{1/3}\log \frac{\varepsilon_0}{a}$.
We will show that $\Big(\frac{ah}{t}\Big)^{1/4}\gg h^{1/3}(\log \frac{\varepsilon_0}{a})$ for all $t$ such that $1\leq \frac{t}{\sqrt{a}}\leq \lambda^{1/3-\epsilon}=\frac{\sqrt{a}}{h^{1/3}}\lambda^{-\epsilon}$, $\epsilon>0$. As in the regime we consider here we have $a\geq h^{2/3-\epsilon}$, then $\lambda=\frac{a^{3/2}}{h}>h^{-3\epsilon/2}$, hence $\lambda^{-\epsilon}\leq h^{3\epsilon^2/2}$ and we obtain $t\leq \frac{a}{h^{1/3}}h^{3\epsilon^2/2}$, which further yields $(\frac{ah}{t})^{1/4}\geq h^{1/3-3\epsilon^2/8}\gg h^{1/3}\log\frac 1h\gtrsim h^{1/3}\log \frac{\varepsilon_0}{a}$ (using
again $a\geq h^{2/3-\epsilon}$). We eventually find $\|G_{h,\varepsilon_0}(t,\cdot)\|_{L^{\infty}(\Omega_d)}\sim  \frac{1}{h^d}\Big(\frac ht \Big)^{(d-1)/2}\Big(\frac{ah}{t}\Big)^{1/4}$. 

\subsection{Case $a\lesssim \sup{(h^{2/3-\epsilon},(ht)^{1/2})}$ for (small) $\epsilon>0$}
\subsubsection{The sum over $8\sup{(h^{2/3-\epsilon},(ht)^{1/2})}\leq \gamma \leq \varepsilon_0$}
This part is easy to deal with as we can apply the estimates obtained in the previous section (with $a$ replaced by $(ht)^{1/2}$). As we have $8a\leq \gamma $ and as in this regime we can use the parametrix \eqref{eq:bis48bis}, we obtain 

\begin{equation}\label{sumgambig}
\|\sum_{8a\lesssim 8\sup{(h^{2/3-\epsilon},(ht)^{1/2})\leq \gamma \leq \varepsilon_0}}G_{h,\gamma}(t,\cdot)\|_{L^{\infty}(\Omega_d)}\lesssim  \frac{1}{h^d}\Big(\frac ht \Big)^{(d-1)/2}\frac{(ht)^{1/2}}{\sup{(h^{2/3-\epsilon},(ht)^{1/2}})^{1/2}}.
\end{equation}
When $t\geq h^{1/3-2\epsilon}$ then $\sup{(h^{2/3-\epsilon},(ht)^{1/2}})=(ht)^{1/2}$ and the last factor in \eqref{sumgambig} equals $(ht)^{1/4}$. When $t\leq h^{1/3-2\epsilon}$ the last factor in \eqref{sumgambig} is bounded by $(ht)^{1/2}/h^{(2/3-\epsilon)/2}\leq h^{1/3-\epsilon/2}$.

\subsubsection{The sum over $\sup{(a,h^{2/3})}\lesssim \gamma\lesssim \sup{(h^{2/3-\epsilon},(ht)^{1/2})}$ }
This part will be entirely dealt with using formula \eqref{greenfctbis} and next Lemma.

\begin{lemma}\label{lemsob}(see \cite{ilp12})
There exists $C_0$ such that for $L\geq 1$ the following holds true
\begin{equation}\label{estairy2}
\sup_{b\in \R}\Big (\sum_{1\leq k\leq L}\omega_k^{-1/2}Ai^2(b-\omega_{k})\Big) \leq C_{0}L^{1/3}\,,
\sup_{b\in\R_+}\Big (\sum_{1\leq k\leq L}\omega_k^{-1/2}Ai'^2(b-\omega_{k})\Big)  \leq C_{0}L\,.
\end{equation}
\end{lemma}
Write, for $\gamma_{\sup}:=\sup{(h^{2/3-\epsilon},(ht)^{1/2})}$, $\gamma_{\min}:=\sup{(a,h^{2/3})}$,
\begin{multline}\label{sumagam}
\sum_{\gamma_{\min}\leq\gamma\leq \gamma_{\sup}}G_{h,\gamma}(t,x,a,y)
=\sum_{k\sim \lambda_{\gamma},\gamma_{\min}\leq \gamma\leq \gamma_{max}}\frac{h^{1/3}}{h^d}\int e^{\frac ih<y,\eta>}\psi(|\eta|) e^{\frac ih t (|\eta|^2+\omega_{k}h^{2/3}q^{2/3}(\eta))}\\
\times \frac{q^{1/3}(\eta)}{L'(\omega_{k})} \psi_{2}(h^{2/3}\omega_{k}/(q^{1/3}(\eta)\gamma))
Ai(xq^{1/3}(\eta)/h^{2/3}-\omega_{k}) Ai(aq^{1/3}(\eta)/h^{2/3}-\omega_{k})d\eta +O(h^{\infty}),
\end{multline}
where we used that $\psi_2$ and $\psi$ are supported on $[\frac 12, \frac 32]$ to deduce $k\sim\omega_k^{3/2}\sim \lambda_{\gamma}q^{1/2}(\eta)\sim \lambda_{\gamma}$ on the support of $\psi_2(h^{2/3}\omega_{k}/(q^{1/3}(\eta)\gamma))\psi(|\eta|)$; the term $O(h^{\infty})$ comes from the (finite) sum over $1\leq k\ll \lambda_{\gamma}$ and $\lambda_{\gamma}\ll k\lesssim 1/h$.
Notice that if $t\leq h^{1/3-2\epsilon}$ then $(ht)^{1/2}\leq h^{2/3-\epsilon}$, which yields $\gamma_{\sup}=h^{2/3-\epsilon}$, hence for such $t$ we have to consider only values $a\leq h^{2/3-\epsilon}$. %
For $t\leq h^{1/3-2\epsilon}$ and $\gamma\leq \gamma_{\sup}=h^{2/3-\epsilon}$, $\lambda_{\gamma}\lesssim h^{-3/2\epsilon}$ for small $\epsilon>0$ and we cannot perform stationary phase arguments with parameter $\lambda_{\gamma}$ ; formula \eqref{eq:bis48bis} becomes useless and we have to resort to \eqref{greenfctbis}.
We consider separately the situations $t\geq h^{1/3-2\epsilon}$ and $t\leq h^{1/3-2\epsilon}$, although the arguments in the corresponding proofs are similar and relying on \eqref{greenfctbis}.
 
\subsubsection{Let $t\geq h^{1/3-2\epsilon}$, in which case $(ht)^{1/2}\geq h^{2/3-\epsilon}$}
We will bring the Airy functions into the symbol and apply stationary phase in $\eta\in\mathbb{R}^{d-1}$. The sum over $k$ is taken over $1\leq k\lesssim (ht)^{3/4}/h$ and on the support of $\psi_2$ we have $k^{2/3}\sim \omega_k\sim \lambda_{\gamma}^{2/3}$ with $\gamma\leq\gamma_{max}:=(ht)^{1/2}$.
\begin{prop}
For  $t\geq h^{1/3-2\epsilon}$, the following dispersive estimates hold
\begin{equation}
\|\sum_{\sup{(a,h^{2/3})}\leq\gamma\leq (ht)^{1/2}}G_{h,\gamma}(t,\cdot)\|_{L^{\infty}(\Omega_d)}\lesssim %
 \frac{1}{h^d}\Big(\frac{h}{t}\Big)^{(d-1)/2}(ht)^{1/4}\,.
\end{equation}
\end{prop}
\begin{proof}
Let $z=y/t$ and let $\frac th$ be the large parameter in the integrals in the fourth line of \eqref{sumagam} whose phase function is, for each $\omega_k\sim \lambda_{\gamma}^{2/3}$, of the form 
$<z,\eta>+|\eta|^2+\omega_kh^{2/3}q^{2/3}(\eta)$. For each $\omega_k\lesssim \gamma_{\sup}/h^{2/3}=(ht)^{1/2}/h^{2/3}$, the corresponding critical point $\eta_c$ satisfies $z+2\eta_c+O(\omega_kh^{2/3})=0$ and using $\omega_k h^{2/3}\leq \varepsilon_0$, we obtain that the Hessian behaves like
 $2\mathbb{I}_{d-1}+O(\varepsilon_0)$. In order to apply stationary phase with symbol
$$
q^{1/3}(\eta)\psi(|\eta|)\psi_2\Bigl(\frac{\omega_k}{q^{1/3}(\eta)\lambda_{\gamma}^{2/3}}\Bigr)\Ai\Big(q^{1/3}(\eta)\lambda_{\gamma}^{2/3}\frac{x}{\gamma}-\omega_k\Big)\Ai\Big(q^{1/3}(\eta)\lambda_{\gamma}^{2/3}\frac{a}{\gamma}-\omega_k\Big)
$$
we check that there exists some $\nu>0$ such that $\forall j\geq 1$ and $\forall \alpha$ with $|\alpha|=j$, 
\begin{equation}\label{condnu}
\Big|\partial^{\alpha}_{\eta}\Big(\Ai\big(q^{1/3}(\eta)\lambda_{\gamma}^{2/3}\frac{x}{\gamma}-\omega_k\big)\Big)\Big|\leq C_{j}\Big(\frac th\Big)^{j(1-2\nu)/2}\,.
\end{equation}
In particular, this allows to deduce that, for $\eta$ on the support of $\psi$ we have
\begin{equation}\label{phstatverification}
  \partial^2_{ij}\Big(q^{\frac 13}(\eta)\psi_2\Bigl(\frac{\omega_k}{q^{\frac 13}(\eta)\lambda_{\gamma}^{\frac 23}}\Bigr)\Ai\big(q^{\frac 13}(\eta)\lambda_{\gamma}^{\frac 23}\frac{x}{\gamma}-\omega_k\big)\Ai\big(q^{\frac 13}(\eta)\lambda_{\gamma}^{\frac 23}\frac{a}{\gamma}-\omega_k\big)\Big)\lesssim \Bigl(\frac{t}{h}\Bigr)^{1-2\nu}
\end{equation} 
and assures that the stationary phase can be applied with the Airy factors as part of the symbol. 
As one has, for all $l\geq 0$, $\sup_{b\geq 0}\Big|b^l Ai^{(l)}(b-\omega_k)\Big|\leq C_l \omega_k^{3l/2}$,
it is sufficient to check that for $t\geq h^{1/3-2\epsilon}$ and $k\leq (ht)^{3/4}/h$ the following holds
\begin{equation}\label{toprove}
\omega_k^{3/2}\lesssim \Big(\frac th\Big)^{(1-2\nu)/2}.
\end{equation}
As $\omega_k\sim k^{2/3}\lesssim \lambda^{2/3}_{\gamma_{\sup}}\sim ((ht)^{3/4}/h)^{2/3}$ for $k\leq (ht)^{3/4}/h$, \eqref{toprove} holds if we prove $t^{1/2}( t/h )^{1/4}={(ht)^{3/4}}/{h}\lesssim (t/h)^{(1-2\nu)/2}$ which is obviously true as it reduces to $t\lesssim (t/h)^{1/2-2\nu}$ for some $\nu>0$ (recall that we consider here only values $t\lesssim 1$). The sum of the main contributions of the symbols obtained after applying stationary phase in $\eta$ equals
\begin{multline}\label{gchiapetittgrand} 
\Big|\sum_{k\lesssim (ht)^{3/4}/h}\omega_k^{-1/2}Ai\Big(x\frac{q^{1/3}(\eta_c(z,\omega_kh^{2/3}))}{h^{2/3}}-\omega_k\Big)Ai\Big(a\frac{q^{1/3}(\eta_c(z,\omega_kh^{2/3}))}{h^{2/3}}-\omega_k\Big)\Big|\\
\leq \Big|\sum_{k\lesssim (ht)^{3/4}/h}\omega_k^{-1/2}Ai^2\Big(x\frac{q^{1/3}(\eta_c(z,\omega_kh^{2/3}))}{h^{2/3}}-\omega_k\Big)\Big|^{1/2}\\
\times  \Big|\sum_{k\lesssim (ht)^{3/4}/h}\omega_k^{-1/2}Ai^2\Big(a\frac{q^{1/3}(\eta_c(z,\omega_kh^{2/3}))}{h^{2/3}}-\omega_k\Big)\Big|^{1/2}
\lesssim (\lambda_{\gamma_{\sup}})^{1/3},
\end{multline} 
where we applied Cauchy-Schwarz followed by \eqref{estairy2} from Lemma \ref{lemsob} with $L\sim  \lambda_{\gamma_{\sup}}=(ht)^{3/4}/h$. 
However, this is not enough to conclude : we also need to prove that lower order terms in the symbol obtained after stationary phase do sum up and provide smaller contributions. This can be done using the second inequality in \eqref{estairy2} as well as the
equation satisfied by the Airy function.
\end{proof}
\subsubsection{Let $t\leq h^{1/3-2\epsilon}$, with (small) $\epsilon>0$} Then $\sup{(h^{2/3-\epsilon},(ht)^{1/2})}=h^{2/3-\epsilon}$ and we consider only $\gamma$ such that $\sup{(h^{2/3},a)}\lesssim \gamma\lesssim h^{2/3-\epsilon}$, as the sum over $\gamma>h^{2/3-\epsilon}>(ht)^{1/2}$ can be handled as in \eqref{sumgambig}. Then $\lambda_{\gamma_{\sup}}=(h^{2/3-\epsilon})^{3/2}/h=h^{-3\epsilon/2}$.
\begin{prop}
Let $0<\epsilon<\frac{2}{9(d+1)}(<\frac 16)$.
For $h^{1/3+\epsilon}\lesssim t\leq h^{1/3-2\epsilon}$ we have
\[
\|\sum_{\sup{(a,h^{2/3})}\lesssim\gamma\lesssim h^{2/3-\epsilon}}G_{h,\gamma}(t,\cdot)\|_{L^{\infty}(\Omega_d)}\\ \lesssim\frac{1}{h^d}\Bigl(\frac ht\Bigr)^{\frac {d-1}2} h^{\frac 13-\frac \epsilon2}.
\]
For  $0< t\lesssim h^{1/3+\epsilon}$ we have
$\|\sum_{\sup{(a,h^{2/3})}\lesssim\gamma\lesssim h^{2/3-\epsilon}}G_{h,\gamma}(t,\cdot)\|_{L^{\infty}(\Omega_d)} \lesssim\frac{1}{h^d}\Bigl(\frac ht\Bigr)^{\frac d2}$.
\end{prop}
\begin{proof}
Let $0<\epsilon<\frac{2}{9(d+1)}$ and set $t(h,\epsilon):=h^{1-3\epsilon-2\epsilon/d}$. The condition $0<\epsilon<\frac{2}{9(d+1)}$ implies $t(h,\epsilon)\ll h^{1/3+\epsilon}$ for all $d\geq 1$.
For $t(h,\epsilon)\lesssim t\lesssim h^{1/3-2\epsilon}$, the same proof as in the previous case applies. Indeed, to use stationary phase with the Airy factors in the symbol we need the condition \eqref{toprove} to be satisfied for all $k\lesssim \lambda_{\gamma_{\sup}}$, which translates into
\begin{equation}\label{toprove2}
h^{-3\epsilon/2}\lesssim \Big(\frac th\Big)^{1/2-\nu}\quad \text {for some } \nu>0.
\end{equation}
Let $\nu=\frac{2}{2+3d}$, then \eqref{toprove2} holds as it rewrites $t\gtrsim h^{1-\frac{3\epsilon}{1-\nu}}=t(h,\epsilon)$. 
Let $t\lesssim t(h,\epsilon)$ and $L:=8h^{-3\epsilon/2}$, then the sum over $k$ in \eqref{sumagam} is taken for $k\leq L$. Using again \eqref{estairy2} yields
\[
  \Big|\sum_{\sup{(a,h^{2/3})}\lesssim\gamma\lesssim h^{2/3-\epsilon}}G_{h,\gamma}(t,\cdot)\Big|
\lesssim \frac{h^{\frac 13}}{h^d} L^{\frac 13}\,.
\]
If $t\lesssim t(h,\epsilon)$, then $\frac{1}{t(h,\epsilon)}\lesssim \frac 1t$ and as $\Big(\frac{h}{t(h,\epsilon)}\Big)^{d/2} h^{-3(d+1)\epsilon/2}=h^{-\epsilon/2}$ we find
\[
h^{1/3}L^{1/3}=2h^{1/3-\epsilon/2}= 2\Big(\frac{h}{t(h,\epsilon)}\Big)^{d/2} h^{1/3-3(d+1)\epsilon/2}\lesssim \Big(\frac ht\Big)^{d/2}h^{1/3-3(d+1)\epsilon/2}\leq \Big(\frac ht\Big)^{d/2},
\]
as the condition $\epsilon<\frac{2}{9(d+1)}$ implies $1/3-3(d+1)\epsilon/2>0$. 
\end{proof}

\section{Proofs of Propositions \ref{dispNgrand}, \ref{dispNpetitloin}, \ref{dispNpetitpres}, }\label{sectproofsprops}
Here we need to analyze in details the structure of higher order derivatives of the phase functions $\phi_{N,a}$. The proof of Proposition \ref{dispNgrand} follows closely the one of \cite[Prop.7]{ILP3} (in the case $x\leq a$) ; the proofs of Propositions \ref{dispNpetitloin} and \ref{dispNpetitpres} become much more delicate in the case of the Schrödinger flow, due to the presence of the critical point $\eta_c$ which is a function depending on $s,\sigma$. As these propositions are crucial in the proof of Theorem \ref{thmdispSchrodinger}, we provide a detailed proof.

Let $V_{N,h,a}$ be defined in \eqref{defVNha} and let $N<\lambda^{1/3}$. 
Using Remark \ref{rmkchi}, we assume (without changing the contribution of $V_{N,h,a}$ modulo $O(h^{\infty})$) that its symbol $\varkappa$ is supported on $|(\sigma,s)|\leq 2\sqrt{\alpha_c}$. 
Fix $T$, $N\in [\frac TM, MT]$ with $M>8$ large enough and let $X=\frac x a\leq 1$, $Y=\frac{y}{\sqrt{a}}$ with $\frac{Y}{2T}\in [\frac 14,2]$.

\subsection*{Proof of Proposition \ref{dispNgrand}} We start with the case where $\lambda^{1/3}\lesssim N$ and we follow closely the proof of \cite[Prop.7]{ILP3}. We will prove the following :
\begin{equation}
  \label{eq:90}
   \left| \int_{\R^{2}} e^{\frac ih \phi_{N,a}}  \varkappa(\sigma,s,t,x,y,h,a,1/N) \, ds d\sigma \right| \lesssim \frac{\lambda^{-2/3}}{1+\lambda^{1/3}|K_a^2(\frac{Y}{4N},\frac{T}{2N})-1|^{1/2}}\,.
\end{equation}
We rescale variables with $\sigma=\lambda^{-1/3}p$ and $s=\lambda^{-1/3}q$ and define
\begin{equation}
  \label{eq:91ff}
  A=\lambda^{2/3}\Big(K_a^2(\frac{Y}{4N},\frac{T}{2N})-X\Big) \quad \text{and}\quad   B=\lambda^{2/3}\Big(K_a^2(\frac{Y}{4N},\frac{T}{2N})-1\Big) \,,
\end{equation}
and we are reduced to proving that the following holds uniformly in $(A,B)$ :
\begin{equation}
  \label{eq:91}
   \left| \int_{\R^{2}} e^{iG_{N,a,\lambda}(p,q,t,x,y)}  \varkappa(\lambda^{-1/3} p,\lambda^{-1/3}q,t,x,y,h,a,1/N) \, dp dq \right| \lesssim \frac{1}{1+|B|^{1/2}}\,,
\end{equation}
where the rescaled phase is
\begin{equation}
  \label{eq:92}
G_{N,a,\lambda}(p,q,t,x,y):=\frac{1}{h} \Big(\phi_{N,a}(\lambda^{-1/3}p,\lambda^{-1/3}q, t,x,y)-\phi_{N,a}(0,0,t,x,y)\Big)\,.  
\end{equation}
Replacing $\gamma$ by $a$ in first order derivatives of $\phi_{N,a,\gamma}$ (\eqref{derphisig1} and \eqref{derphis1}) yields
\begin{gather*}
  \partial_pG_{N,a,\lambda}=\frac 1h\frac{\partial\sigma}{\partial p}\partial_{\sigma}(\phi_{N,a})|_{(\sigma,s)=(\lambda^{-1/3}p,\lambda^{-1/3}q)}%
  =q^{1/2}(\eta_c)(p^2-\lambda^{2/3}(\alpha_c-X))\,,
\\
\partial_qG_{N,a,\lambda}=\frac 1h \frac{\partial s}{\partial q}\partial_{s}(\phi_{N,a})|_{(\sigma,s)=(\lambda^{-1/3}p,\lambda^{-1/3}q)}%
=q^{1/2}(\eta_c)(q^2-\lambda^{2/3}(\alpha_c-1))\,.
\end{gather*}
From \eqref{eq:alpha_c21}, in our new variables, $\alpha_c$ has the following expansion
\begin{equation}\label{eq:alpha_c212}
\alpha_c|_{(\lambda^{-1/3}p,\lambda^{-1/3}q)}=
\Big(K_a(\frac{Y}{4N},\frac{T}{2N})-\lambda^{-1/3}\frac{p}{2N}(1-a\mathcal{E}_1)-\lambda^{-1/3}\frac{q}{2N}(1-a\mathcal{E}_2)\Big)^2,
\end{equation}
where $f_j$ are smooth functions of $(\sigma,s)=\lambda^{-1/3}(p,q)$ and of $\frac{T}{2N}$, $X$, $\frac{Y}{4N}$.
With these notations and with $K_a=K_a(\frac{Y}{4N},\frac{T}{2N})$, we re-write the first order derivatives of $G_{N,a,\lambda}$,
\[
\partial_{p}G_{N,a,\lambda}=q^{1/2}(\eta_c)\Big(p^2-A+\frac{\lambda^{1/3}}{N}K_a(p(1-a\mathcal{E}_1)+q(1-a\mathcal{E}_2))-\frac{1}{4N^2}(p(1-a\mathcal{E}_1)+q(1-a\mathcal{E}_2))^2\Big),
\]
\[
\partial_{q}G_{N,a,\lambda}=q^{1/2}(\eta_c)\Big(q^2-B+\frac{\lambda^{1/3}}{N}K_a(p(1-a\mathcal{E}_1)+q(1-a\mathcal{E}_2))-\frac{1}{4N^2}(p(1-a\mathcal{E}_1)+q(1-a\mathcal{E}_2))^2\Big),.
\]
As $\lambda^{1/3}\leq N$, if $A,B$ are bounded, then \eqref{eq:91} obviously holds for $|(p,q)|$ bounded and by integration by parts if $|(p,q)|$ is large. So we can assume that $|(A,B)|\geq r_{0}$ with $r_{0}\gg 1$. Set $(A,B)=r (\cos(\theta),\sin(\theta))$ and rescale again $(p,q)=r^{1/2} (\tilde p,\tilde q)$: we aim at
\begin{equation}
  \label{eq:93}
    \left| \int_{\R^{2}} e^{i r^{3/2} \tilde G_{N,a,\gamma}}  \varkappa(\lambda^{-1/3}r^{1/2} \tilde p,\lambda^{-1/3}r^{1/2}\tilde q, t,x,y,h,a,1/N) \, d\tilde p d\tilde q \right| \lesssim \frac{1}{r^{5/4}}\,,
\end{equation}
where $r$ is our large parameter, and $\tilde G_{N,a,\lambda}(\tilde p,\tilde q,t,x,y) =r^{-3/2}G_{N,a,\lambda}(r^{1/2}p,r^{1/2}q,t,x,y)$.
Let us compute, using the formulas of the first order derivatives of $G_{N,a,\lambda}$
\begin{gather}\label{derivtildep}
\frac{\partial_{\tilde p} \tilde G_{N,a,\lambda}}{q^{1/2}(\eta_c)}=\tilde p^{2}-\cos\theta+\frac{\lambda^{\frac 13} K_a}{N r^{\frac 12}}(\tilde p(1-a\mathcal{E}_1)+\tilde q(1-a\mathcal{E}_2))-\frac{(\tilde p(1-a\mathcal{E}_1)+\tilde q(1-a\mathcal{E}_2))^{2}}{4 N^{2}}, \\
  \nonumber
 \frac{\partial_{\tilde q} \tilde G_{N,a,\lambda}}{q^{1/2}(\eta_c)}=\tilde q^{2}-\sin\theta+\frac{\lambda^{\frac 13} K_a}{ N r^{\frac 12}}(\tilde p(1-a\mathcal{E}_1)+\tilde q(1-a\mathcal{E}_2))-\frac{(\tilde p(1-a\mathcal{E}_1)+\tilde q(1-a\mathcal{E}_2))^{2}}{4 N^{2}},
\end{gather}
where, abusing notations, $\mathcal{E}_j$ is now $\mathcal{E}_j(r^{1/2}\lambda^{-1/3}(\tilde q,\tilde p),\frac{T}{2N},\frac{Y}{4N})$. On the support of $\varkappa(\cdots)$ %
we have $|(\tilde p,\tilde q)|\lesssim \lambda^{1/3} r^{-1/2}\lesssim \lambda^{1/3} r_{0}^{-1/2}$: then, for $\lambda^{1/3}\lesssim N$, the last term in both derivatives is $O(r_{0}^{-1})$, while the next to last term is $r_{0}^{-1/2}O(\tilde p,\tilde q)$; indeed, using boundedness of $\mathcal{E}_{1,2}$ and $K_a$, we obtain $|\frac{\lambda^{1/3}}{N} K_a\frac{(\tilde p(1-a \mathcal{E}_1)+\tilde q (1-a \mathcal{E}_2))}{r^{1/2}}|\lesssim r_0^{-1/2} |\tilde p+\tilde q|$. Therefore, when $|(\tilde p,\tilde q)|>\tilde C$ with $\tilde C$ sufficiently large, the corresponding part of the integral is $O(r^{-\infty})$ by integration by parts. So we are left with restricting our integral to a compact region in $(\tilde p,\tilde q)$.

We remark that, from $X\leq 1$, we have $A\geq B$ (and $A=B$ if and only if $X=1$), e.g. $\cos \theta \geq \sin \theta$ and therefore $\theta\in (-\frac{3\pi}{4},\frac{\pi}{4})$. We proceed differently upon the size of $B=r\sin\theta$. If $\sin \theta<-C/r^{1/2}$ for some $C>0$ sufficiently large then $\partial_{\tilde q}\tilde G_{N,a,\lambda}>c/(2r^{1/2})$ for some $C>c>0$ and the phase is non stationary.
Indeed, in this case 
\[
\frac{\partial_{\tilde q}\tilde G_{N,a,\lambda}}{q^{1/2}(\eta_c)}\geq \tilde q^{2}+\frac{C}{2 r^{1/2}}+\frac{\lambda^{1/3} K_a}{N r^{1/2}}(\tilde p (1-a \mathcal{E}_1)+\tilde q (1-a \mathcal{E}_2))-\frac{(\tilde p (1-a \mathcal{E}_1)+\tilde q(1-a \mathcal{E}_2))^{2}}{4N^{2}}
\]
and using that $\tilde p, \tilde q$ are bounded, that on the support of $\varkappa$ we have $|r^{1/2}(\tilde p,\tilde q)|\lesssim \lambda^{1/3}$ and that $\frac 1N\lesssim \frac{1}{\lambda^{1/3}}\ll 1$, we then have, for some $C$ large enough
\[
\frac{\lambda^{1/3}}{N} (\tilde p+\tilde q)\Big[\frac{K_a}{r^{1/2}}-\frac{(\tilde p(1-a\mathcal{E}_1)+\tilde q(1-a\mathcal{E}_2))}{4N\lambda^{1/3}}\Big]\lesssim \frac{C}{4r^{1/2}}\,.
\]
We recall that on the support of $\psi_2(\alpha)$ we had $\alpha\in [\frac 12,\frac 32]$ and the critical point $\alpha_{c}$ is such that \eqref{eq:alphac} holds (with $\gamma$ replaced by $a$ in this case) hence 
$K_a=K_a(\frac{Y}{4N},\frac{T}{2N})$ introduced in \eqref{defKa}
stays close to $1$ as the main contribution of $\alpha_c$. It follows that $\partial_{\tilde q}\tilde G_{N,a,\lambda}>C/(2r^{1/2})$ and integrations by parts yield a bound $O(r^{-n})$ for all $n\geq 1$.

Next, let $\sin \theta >-C/r^{1/2}$ and assume $A>0$ (since otherwise the non-stationary phase applies), which in turn implies $A>r_{0}/2$. Indeed, $\cos \theta \geq \sin \theta>-C/r^{1/2}$ implies $\theta\in (-\frac{C}{\sqrt{r_0}},\frac{\pi}{4})$ and therefore in this regime $\cos\theta\geq \frac{\sqrt{2}}{2}$. Consider first the case $|\sin \theta |<C/r^{1/2}$. Non degenerate stationary phase always applies in $\tilde p$, at two (almost) opposite values of $\tilde p$, such that $|\tilde p_{\pm}|\sim |\pm\sqrt{\cos\theta}|\geq 1/4$, and the integral in \eqref{eq:93} rewrites
\begin{multline}
  \label{eq:95}
      r \int_{\R^{2}} e^{i r^{3/2} \tilde G_{N,a,\lambda}}  \varkappa(\lambda^{-1/3}r^{1/2} \tilde p,\lambda^{-1/3}r^{1/2}\tilde q, t,x,y,h,a,1/N) \, d\tilde p d\tilde q \\
      = \frac r {r^{3/4}} \left( \int_{\R} e^{i r^{3/2} \tilde G^+_{N,a,\lambda}}  \varkappa^{+}(\tilde q,t,x,y,h,a,1/N) \,  d\tilde q\right. 
+\left. \int_{\R} e^{i r^{3/2} \tilde G^-_{N,a,\lambda}}  \varkappa^{-}(\tilde q ,h,a,1/N) \, d\tilde q\right)\,.
\end{multline}
Indeed, the phase is stationary in $\tilde p$ when
\[
\tilde p^2=\cos\theta-\frac{\lambda^{1/3}K_a}{Nr^{1/2}}(\tilde p(1-a \mathcal{E}_1)+\tilde q(1-a \mathcal{E}_2))+\frac{(\tilde p(1-a \mathcal{E}_1)+\tilde q(1-a \mathcal{E}_2))^2}{4N^2},
\]
and from $\cos\theta\geq \frac{\sqrt{2}}{2}$ and $\frac{1}{r}\leq \frac{1}{r_0}\ll 1$, there are exactly two disjoint solutions to $\partial_{\tilde p}\tilde G_{N,a,\lambda}=0$, that we denote $\tilde p_{\pm}=\pm\sqrt{\cos\theta}+O(r^{-1/2})$. We compute, at critical points, %
\[
\partial^2_{\tilde p,\tilde p}\tilde G_{N,a,\lambda}|_{
p_{\pm}}=q^{1/2}(\eta_c)\Big(2\tilde p+\frac{\lambda^{1/3}K_a}{Nr^{1/2}}(1+O(a)\Big)+O(N^{-2})|_{\tilde p_{\pm}},
\]
where we used  $\tilde p,\tilde q$ bounded and $\partial_{\tilde p}\mathcal{E}_j=O(\frac{r^{1/2}\lambda^{-1/3}}{N})$ to deduce that all the terms except the first one are small. As $\lambda^{1/3}\lesssim N$, $r^{-1/2}\ll 1$, $K_a$ bounded, close to $1$, for $\tilde p\in\{\tilde p_{\pm}\}$ we get $\partial^2_{\tilde p,\tilde p}\tilde G_{N,a,\lambda}|_{\tilde p_{\pm}}\sim 2\tilde p_{\pm}+O(r^{-1/2})$,
and as $|\tilde p_{\pm}|\geq \frac 14-O(r^{-1/2})$, stationary phase applies. The critical values of the phase at $\tilde p_{\pm}$, denoted $\tilde G^{\pm}_{N,a,\lambda}$, are such that
\begin{multline}\label{derivPhi3q}
\partial_{\tilde q} \tilde G^{\pm}_{N,a,\lambda}(\tilde q,.):=\partial_{\tilde q}\tilde G_{N,a,\lambda}(\tilde q,\tilde p_{\pm},.)
=q^{1/2}(\eta_c) \Big(\tilde q^{2}-\sin\theta\\
+\frac{\lambda^{1/3} K_a(\tilde p(1-a\mathcal{E}_1)+\tilde q(1-a\mathcal{E}_2))}{ N r^{1/2}}-\frac{(\tilde p(1-a\mathcal{E}_1)+\tilde q(1-a\mathcal{E}_2))^{2}}{4 N^{2}}|_{\tilde p=\tilde p_{\pm}}\Big).
\end{multline}
As $|\sin \theta |<C/r^{1/2}$, the phases $\tilde G^{\pm}_{N,a,\lambda}$ may be stationary but degenerate; taking two derivatives in \eqref{derivPhi3q}, one easily checks that
$|\partial^{3}_{\tilde q} \tilde G^{\pm}_{N,a,\lambda}|\geq q^{1/2}(\eta_c)(2-O(r_{0}^{-1/2}))\,$. Hence we get, by Van der Corput Lemma
\begin{equation}
  \label{eq:97}
\left|  \int_{\R} e^{i r^{3/2}\tilde G^{\pm}_{N,a,\lambda}}  \varkappa^{\pm}(\tilde q ,t,x,y,h,a,1/N) \,  d\tilde q\right| \lesssim (r^{3/2})^{-1/3}\,.
\end{equation}
Using \eqref{eq:95} and \eqref{eq:97} eventually yields
\begin{equation}
  \label{eq:98}
  \left|   r \int_{\R^{2}} e^{i r^{3/2}\tilde G_{N,a,\lambda}}  \varkappa(\lambda^{-1/3}r^{1/2}\tilde p, \lambda^{-1/3}r^{1/2}\tilde q, t,x,y,h,a,1/N) \, d\tilde p d\tilde q\right|\lesssim r^{-1/4}.
\end{equation}
Notice moreover that $|B|=|r\sin\theta|\leq C r^{1/2}$, hence from $r^{2}=A^{2}+B^{2}$, we have $A\sim r$ (large) and $r^{-1/4}\lesssim 1/(1+|B|^{1/2})$: \eqref{eq:91} holds true and, replacing $B$ by $\lambda^{2/3}(K_a^2-1)$, it yields \eqref{eq:90}. Substitution with \eqref{eq:91ff} and using $a^2=(h\lambda)^{4/3}$, we obtain from \eqref{eq:90} 
\begin{equation}
|V_{N,h,a}(t,x,y)|\leq \frac{a^2}{h}\frac{1}{\sqrt{\lambda N}}\frac{\lambda^{-\frac 23}}{(1+\lambda^{\frac 13}|K_a^2-1|^{\frac 12})}
=\frac{2h^{\frac 13}}{2\sqrt{N/\lambda^{\frac 13}}+\lambda^{\frac 16}\sqrt{K_a+1}|4NK_a-4N|^{\frac 12}}\,.
\end{equation}
In the last case $\sin \theta>C/r^{1/2}$ ($A\geq B\geq Cr^{1/2}$), stationary phase holds in $(\tilde p,\tilde q)$: the determinant of the Hessian is at least $C\sqrt{\cos \theta}\sqrt{\sin \theta}$ and we get,  %
\[
\Big|\text{(LHS)\eqref{eq:93}}\Big|\lesssim  \frac 1 {(\sqrt{\cos \theta}\sqrt{\sin \theta})^{1/2} r^{3/2}}  \lesssim\frac 1 r \frac 1  {(r \sqrt{\cos \theta}\sqrt{\sin \theta})^{1/2}}\lesssim \frac 1 r \frac{1}{(AB)^{1/4}}
\]
so in this case our estimate is slightly better than \eqref{eq:90}, as we have
\begin{equation}
  \label{eq:100}
     \left| \int_{\R^{2}} e^{\frac ih \phi_{N,a}}  \varkappa(s,\sigma,t,x,y,h,a,1/N) \, ds d\sigma \right| \lesssim \frac 1 {\lambda^{2/3}|AB|^{1/4}}\leq \frac 1 {\lambda^{2/3}|B|^{1/2}}\,.
\end{equation}
This completes the proof of Proposition  \ref{dispNgrand} as it eventually yields
\begin{equation}
|V_{N,h,a}(t,x,y)|\lesssim \frac{(h\lambda)^{4/3}}{h}\frac{\lambda^{-1/2}}{N^{1/2}} \frac 1 {\lambda^{2/3}|B|^{1/2}}\sim h^{1/3}\frac{\lambda^{1/6}}{N^{1/2}}\frac{1}{\lambda^{1/3}|K_a^2-1|^{1/2}}\,.
\end{equation}
%

\subsection*{Proof of Propositions \ref{dispNpetitloin} and \ref{dispNpetitpres}}
The main differences between the proof of Proposition \ref{dispNpetitloin} and that of \cite[Prop.5]{ILP3} occur from the additional critical point $\eta_c$, which is not considered in the case of the wave equation. Similarly, the proof of Proposition \ref{dispNpetitpres} follows the same path as \cite[Prop.6]{ILP3}, but one has to carefully deal with  contributions coming from the higher order derivatives of $\eta_c$. Let $1\leq N< \lambda^{1/3}$: we aim at proving
\begin{equation}\label{eq:90bis}
   \left| \int_{\R^{2}} e^{\frac ih \phi_{N,a}}  \varkappa(\sigma,s,t,x,y,h,a,1/N) \, ds d\sigma \right| \lesssim N^{1/4}\lambda^{-3/4}\,.
\end{equation}
As $N$ is bounded by $\lambda^{1/3}$, ignoring the last two terms in the first order derivatives of $\phi_{N,a}$, as we did in the previous case, is no longer possible. Set $\Lambda=\lambda/N^{3}$ to be the new large parameter. Rescale again variables $\sigma=p'/N$ and $s=q'/N$ and set now
\[
\Lambda G_{N,a}(p',q',t,x,y)=\frac 1h\Big( \phi_{N,a}(\sigma,s, t,x,y)-\phi_{N,a}(0,0,t,x,y)\Big).
\]
We are reduced to proving %
   $\left| \int_{\R^{2}} e^{i \Lambda G_{N,a}}  \varkappa(p'/N,q'/N,\cdots) \, dp' dq' \right| \lesssim \Lambda^{-3/4}$. Compute %
\begin{multline}\label{GNderiv}
\nabla_{(p',q')}G_{N,a}=\frac{N^3}{h}\Big(\frac{\partial \sigma}{\partial p'}\partial_{\sigma} \phi_{N,a},\frac{\partial s}{\partial q'}\partial_{s}\phi_{N,a}\Big)|_{(p'/N,q'/N)}\\
=q^{1/2}(\eta_c)\Big(p'^2+N^2(X-\alpha_c), q'^2+N^2(1-\alpha_c)\Big),
\end{multline}
where, using  \eqref{eq:alpha_c21}, $\alpha_c(\sigma,s,\cdot)|_{(\sigma=p'/N,s=q'/N)}
=\Big(K_a-\frac{p'}{2N^2}(1-af_1)-\frac{q'}{2N^2}(1-af_2)\Big)^2$.
 Recall that $K_a=\sqrt{\alpha^0_c}$ and stays close to $1$ on the support of the symbol.
We define %
$ A'=(K_a^{2}-X)N^{2}$ and $ B'=(K_a^{2}-1)N^{2}$. First order derivatives of $G_{N,a,\lambda}$ read 
\[
\partial_{p'}G_{N,a}=q^{1/2}(\eta_c)\Big(p'^2-A'+K_a(p'(1-a\mathcal{E}_1)+q'(1-a\mathcal{E}_2))-\frac{1}{4N^2}(p'(1-a\mathcal{E}_1)+q'(1-a\mathcal{E}_2))^2\Big),
\]
\[
\partial_{q'}G_{N,a}=q^{1/2}(\eta_c)\Big(q'^2-B'+K_a(p'(1-a\mathcal{E}_1)+q'(1-a\mathcal{E}_2))-\frac{1}{4N^2}(p'(1-a\mathcal{E}_1)+q'(1-a\mathcal{E}_2))^2\Big).
\]
Unlike the previous case, the two last terms are no longer disposable. We start with $|(A',B')|\geq r_{0}$ for some large, fixed $r_{0}$, in which case we can follow the same approach as in the previous case. Set again $A'=r\cos \theta$ and $B'=r\sin \theta$. If $|(p',q')|<r_{0}/2$, then the corresponding integral is non stationary and we get decay by integration by parts. We change variables $(p',q')=r^{1/2}(\tilde p',\tilde q')$ with $r_0\leq r\lesssim N^2$ and aim at proving the following
\begin{equation}
  \label{eq:101}
     \left| r  \int_{\R^{2}} e^{i r^{3/2} \Lambda \tilde G_{N,a}}  \varkappa(r^{1/2}\tilde p'/N,r^{1/2} \tilde q'/N,t,x,y,h,a,1/N) \, d\tilde p' d\tilde q' \right| \lesssim r^{-1/4 }\Lambda^{-5/6}\,,
\end{equation}
The new phase is $\tilde G_{N,a}(\tilde p',\tilde q',t,x,y)=r^{-3/2}G_{N,a}(r^{1/2}\tilde p',r^{1/2}\tilde q',t,x,y)$. 
We compute 
\begin{align*}
  \frac{\partial_{\tilde p'} \tilde G_{N,a}}{q^{1/2}(\eta_c)}&=\tilde p'^{2}-\cos\theta+ \frac{K_a}{r^{1/2}}(\tilde p'(1-a\mathcal{E}_1)+\tilde q'(1-a\mathcal{E}_2))-\frac{(\tilde p'(1-a\mathcal{E}_1)+\tilde q'(1-a\mathcal{E}_2))^{2}}{4 N^{2}}, \\
 \frac{\partial_{\tilde q'} \tilde G_{N,a}}{q^{1/2}(\eta_c)}&=\tilde q'^{2}-\sin\theta+ \frac{ K_a}{r^{1/2}}(\tilde q'(1-a\mathcal{E}_1)+\tilde q'(1-a\mathcal{E}_2))-\frac{(\tilde p'(1-a\mathcal{E}_1)+\tilde q'(1-a\mathcal{E}_2))^{2}}{4 N^{2}}.
 \end{align*}
To the extend it is possible to do so, we follow the previous case $\lambda^{1/3}\lesssim N$. From $X\leq 1$, $ A'\geq  B'$ implying $\cos\theta\geq \sin \theta$. If $|(\tilde p',\tilde q')|\geq \tilde C$ for some large $\tilde C\geq 1$, then $(\tilde p'_c, \tilde q'_c)$ are such that $\tilde p'^2_c\geq \tilde q'^2_c$ and if $\tilde C$ is sufficiently large non-stationary phase applies (pick any $\tilde C>4$.) Therefore we are reduced to bounded $|(\tilde p',\tilde q')|$. We sort out cases, depending upon $B'=r\sin\theta $ : if $\sin \theta<-\frac{C}{\sqrt{r}}$ for some sufficiently large constant $C>0$, then
\[
\frac{\partial_{\tilde q'} \tilde G_{N,a}}{q^{1/2}(\eta_c)} \geq \tilde q'^{2}+\frac{C}{r^{1/2}}+ \frac{K_a}{r^{1/2}}(\tilde p'(1-a\mathcal{E}_1)+\tilde q'(1-a\mathcal{E}_2))-\frac{(\tilde p'(1-a\mathcal{E}_1)+\tilde q'(1-a\mathcal{E}_2))^{2}}{4 N^{2}},
\]
and $\mathcal{E}_{1,2}$ are bounded, $N$ is sufficiently large in this case (indeed, recall that $r_0\leq r\lesssim N^2$ so that $\frac{1}{\sqrt{r}}\geq \frac 1N$);  then, non-stationary phase applies  as the sum of the last three terms in the previous inequality is greater than $C/(2r^{1/2})$ if $C$ is large enough. If $|\sin\theta|\leq \frac{C}{\sqrt{r}}$ then, again, $\theta\in (-\frac{C}{\sqrt{r_0}},\frac{\pi}{4})$ and $\cos\theta\geq \frac{\sqrt{2}}{2}$. We have $|B'| =|r\sin\theta|\leq C \sqrt{r}$; if $|B'|<C$, then $1+|B'|\lesssim r^{1/2}$, while $|A'|\sim r$. Stationary phase applies in $\tilde p'$ with non-degenerate critical points $\tilde p'_{\pm}$ and yields a factor $(r^{3/2}\Lambda)^{-1/2}$; the critical value of the phase function at these critical points, that we denote $\tilde G^{\pm}_{N,a}$, is always such that $|\partial^3_{\tilde q'}\tilde G^{\pm}_{N,a}|\geq q^{1/2}(\eta_c)(2-O(\frac{1}{r_0^{1/2}}))$ and the integral in $\tilde q'$ is bounded by $(r^{3/2}\Lambda)^{-1/3}$ by Van der Corput. We therefore obtain \eqref{eq:101} which yields, using that $|B'|=|N^2(K_a^2-1)|\leq r^{1/2}$,
\begin{multline*}
|V_{N,a,h}(t,x,y)|=\frac{h^{1/3}\lambda^{4/3}}{\sqrt{\lambda N}N^2}\Big |r \int_{\R^{2}} e^{i r^{3/2} \Lambda \tilde G_{N,a}}  \varkappa(r^{1/2}\tilde p'/N,r^{1/2} \tilde q'/N,t,x,y,h,a,1/N) \, d\tilde p' d\tilde q'\Big |\\
\lesssim \frac{h^{1/3}\lambda^{5/6}}{N^{5/2}}r^{-1/4}\Big(\frac{\lambda}{N^3}\Big)^{-5/6}\lesssim \frac{h^{1/3}}{(1+|B'|^{1/2})}\sim \frac{h^{1/3}}{(1+N|K_a(\frac{Y}{4N},\frac{T}{2N})-1|^{1/2})}.
\end{multline*}
If $\sin\theta > \frac{C}{\sqrt{r}}$, then $B'=r\sin\theta>C\sqrt{r}$ and therefore $N^2|K_a^2-1|>Cr^{1/2}$. We do stationary phase in both variables with large parameter $r^{3/2}\Lambda$ as the determinant of the Hessian at critical points is at least $C\sqrt{\cos\theta\sin \theta}$, and obtain, for left hand side term in \eqref{eq:101}, a bound
\[
\frac{c r}{(\sqrt{\sin\theta}\sqrt{\cos\theta})^{1/2}r^{3/2}\Lambda}=\frac{1}{\Lambda}\frac{1}{(A'B')^{1/4}}\leq \frac{1}{\Lambda}\frac{1}{B'^{1/2}}.
\]
We just proved that for $N<\lambda^{1/3}$ and not too small $N^2|K_a(\frac{Y}{4N},\frac{T}{2N})-1|$,
  \begin{equation}
    \label{eq:2ffbis}
       \left| V_{N,h,a}(t,x,y)\right| \lesssim  \frac{h^{1/3}}{\lambda^{1/6}\sqrt{N}|K_a(\frac{Y}{4N},\frac{T}{2N})-1|^{1/2}} \,.
  \end{equation}
  We now move to the most delicate case $|(A',B')|\leq r_{0}$. For $|(p',q')|$ large, the phase is non stationary and integrations by parts provide $O(\Lambda^{-\infty})$ decay. So we may replace $\varkappa$ by a cut-off, that we still call $\varkappa$, compactly supported in $|(p',q')|<R$. We proceed %
   by identifying one variable where usual stationary phase applies and then evaluating the remaining $1D$ oscillatory integral using Van der Corput (with different decay rates depending on the lower bounds on derivatives, of order at most $4$.) Using \eqref{GNderiv}, we compute derivatives of $G_{N,a}$
\begin{equation}\label{eq:H_Nderivp}
\partial_{p'}G_{N,a} =q^{1/2}(\eta_c)(p'^2+N^2(X-\alpha_c)),\quad \partial_{q'}G_{N,a} =q^{1/2}(\eta_c)(q'^2+N^2(1-\alpha_c)).
\end{equation}
The second order derivatives of $G_{N,a}$ 
follow from \eqref{secdersig}, \eqref{secders} and \eqref{secderssig} 
\begin{gather}
 \partial^2_{p'p'}G_{N,a}=q^{1/2}(\eta_c)(2p'-N^2\partial_{p'}\alpha_c)+\frac{\partial_{p'}\eta_c\nabla q(\eta_c)}{2q^{1/2}(\eta_c)}(p'^2+N^2(X-\alpha_c)),\\
  \partial^2_{q'q'}G_{N,a}=q^{1/2}(\eta_c)(2q'-N^2\partial_{q'}\alpha_c)+\frac{\partial_{q'}\eta_c\nabla q(\eta_c)}{2q^{1/2}(\eta_c)}(q'^2+N^2(1-\alpha_c)),\\
   \partial^2_{q'p'}G_{N,a}=q^{1/2}(\eta_c)(-N^2\partial_{q'}\alpha_c)+\frac{\partial_{q'}\eta_c\nabla q(\eta_c)}{2q^{1/2}(\eta_c)}(p'^2+N^2(X-\alpha_c))\\
   \nonumber
  =\partial^2_{p'q'}G_{N,a} =q^{1/2}(\eta_c)(-N^2\partial_{p'}\alpha_c)+\frac{\partial_{p'}\eta_c\nabla q(\eta_c)}{2q^{1/2}(\eta_c)}(q'^2+N^2(1-\alpha_c)).
\end{gather}
At critical points, where $\partial_{p'}G_{N,a}=\partial_{q'}G_{N,a}=0$, the determinant of the Hessian reads
\[
\det\text{Hess}_{(p',q')}G_{N,a}|_{\nabla_{(p',q')}G_{N,a}=0}=q(\eta_c)\Big(4p'q'-N^2(p'+q')\partial_{p'}\alpha_c\Big).
\]
If $|\det\text{Hess}_{(p',q')}G_{N,a}|>c>0$ for some small $c>0$ we can apply usual stationary phase in both variables $p',q'$. We expect the worst contributions to occur in a neighborhood of the critical points where $|\det\text{Hess}_{(p',q')}G_{N,a}|\leq c$ for some $c$ sufficiently small. We turn variables with $\xi_{1}=(p'+q')/2$ and $\xi_{2}=(p'-q')/2$. Then $p'=\xi_{1}+\xi_{2}$ and $q'=\xi_{1}-\xi_{2}$, and we also let $\mu:=A'+B'=N^2(2K_a^2-1-X)$,  $\nu:=A'-B'=N^2(1-X)$. The most degenerate situation will turn out to be $\nu=\mu=0$ and $\xi_{1}=0,\xi_{2}=0$. Let $g_{N,a}(\xi_1,\xi_2)=G_{N,a}(\xi_1+\xi_2,\xi_1-\xi_2)$.
\vskip1mm

\paragraph{\bf Case $c\lesssim |\xi_1|$}
For $\xi_1$ outside a small neighbourhood of $0$, non degenerate stationary phase applies in $\xi_2$ and the critical value $g_{N,a}(\xi_1,\xi_{2,c})$ may have degenerate critical points of order at most $2$.  
The phase $g_{N,a}$ is stationary in $\xi_2$ whenever $\partial_{p'}G_{N,a}=\partial_{q'}G_{N,a}$ and from Remark \ref{rmqcritssigderalphac}, we then have $\partial_{p'}\eta_c=\partial_{q'}\eta_c$ and $\partial_{p'}\alpha_c=\partial_{q'}\alpha_c$. We have 
\[
\partial^2_{\xi_2,\xi_2}g_{N,a}(\xi_1,\xi_2)=\Big(\partial^2_{p'p'}G_{N,a}-2\partial^2_{p'q'}G_{N,a}+\partial^2_{q'q'}G_{N,a}\Big)(p',q')|_{\xi_1,\xi_2}.
\]
Using the explicit form of the second order derivatives of $G_{N,a}$ given above, at $p'=\xi_1+\xi_2$, $q'=\xi_1-\xi_2$ such that $p'^2+N^2(X-\alpha_c)=q'^2+N^2(1-\alpha_c)$ and with $\partial_{p'}\eta_c=\partial_{q'}\eta_c$, we obtain 
\[
\partial^2_{\xi_2,\xi_2}g_{N,a}(\xi_1,\xi_2)|_{\partial_{\xi_2}g_{N,a}=0}=2q^{1/2}(\eta_c)(p'+q')=4q^{1/2}(\eta_c)\xi_1.
\]
As $q(\eta_c)=|\eta_c|q(\eta_c/|\eta_c|)\in [\frac 12m_0^2,\frac 32 M_0^2]$ with $m_0,M_0$ defined in \eqref{limitshalfq}, stationary phase applies in $\xi_2$. We denote $\xi_{2,c}$ the critical point, such that 
\[
\partial_{\xi_2}g_{N,a}(\xi_1,\xi_2)=\Big(\partial_{p'}G_{N,a}-\partial_{q'}G_{N,a}\Big)(p',q')|_{p'=\xi_1+\xi_2,q'=\xi_1-\xi_2}=0\,,
\]
which rewrites $(\xi_1+\xi_{2,c})^2+N^2(X-\alpha_c)=(\xi_1-\xi_{2,c})^2+N^2(1-\alpha_c)$, which, in turn, yields $4\xi_1\xi_{2,c}=N^2(1-X)=\nu$ and therefore $\xi_{2,c}=\frac{\nu}{4\xi_1}$. We will now compute higher order derivatives of the critical value of $g_{N,a}(\xi_1,\xi_{2,c})$ with respect to $\xi_1$.
\begin{lemma}\label{lemxi1grand}
For $|N|\geq 1$, the phase $g_{N,a}(\xi_1,\xi_{2,c})$ may have critical points degenerate of order at most $2$.
\end{lemma}
\begin{proof}
Again, at $\xi_{2,c}$, Remark \eqref{rmqcritssigderalphac} implies $\partial_{p'}\eta_c=\partial_{q'}\eta_c$ and $\partial_{p'}\alpha_c=\partial_{q'}\alpha_c$. In turn, the functions ${\Theta}_{1,2}$ in Lemma \ref{lemalphacetac} coincide as well, hence the functions $\mathcal{E}_{1,2}$ defined in \eqref{f-1},\eqref{f-2} coincide also at $\xi_{2,c}$. We abuse notation  with $\mathcal{E}_{1,2}$ as functions of $(p'/N,q'/N)=(\xi_1+\xi_2)/N,(\xi_1-\xi_2)/N$. Set $\mathcal{E}:=\mathcal{E}_1|_{p'^2+N^2X=q'^2+N^2}=\mathcal{E}_2|_{p'^2+N^2X=q'^2+N^2}$ in \eqref{eq:alpha_c21}, then $\sqrt{\alpha_c}|_{\partial_{\xi_2}g_{N,a}=0}=K_a-\frac{\xi_1}{N^2}(1-a\mathcal{E})$ and therefore
\begin{align}
\nonumber
\partial_{\xi_1}(g_{N,a}(\xi_1,\xi_{2,c}))&=\partial_{\xi_1}g_{N,a}(\xi_1,\xi_{2,c})+\frac{\partial \xi_{2,c}}{\partial\xi_1}\partial_{\xi_2}g_{N,a}(\xi_1,\xi_2)|_{\xi_2=\xi_{2,c}}\\
\nonumber
&=\Big(\partial_{p'}G_{N,a}+\partial_{q'}G_{N,a}\Big)(p',q')|_{\xi_1,\xi_{2,c}}\\
\label{gfirstderiv}
&=q^{1/2}(\eta_c)\Big(2\xi_1^2(1-\frac{1}{N^2}(1-a\mathcal{E}))+2\frac{\nu^2}{16\xi_1^2}-\mu+4K_a\xi_1(1-a\mathcal{E})\Big).
\end{align}
Taking a derivative of \eqref{gfirstderiv} with respect to $\xi_1$ yields
\begin{gather*}
\partial^2_{\xi_1,\xi_1}(g_{N,a}(\xi_1,\xi_{2,c}))=q^{1/2}(\eta_c)\Big[4\xi_1\Big(1-\frac{1}{N^2}(1-a(\mathcal{E}+\frac 12\xi_1\partial_{\xi_1}\mathcal{E}))\Big)-\frac{\nu^2}{8\xi_1^3}\\+4K_a\Big(1-a(\mathcal{E}+\xi_1\partial_{\xi_1}\mathcal{E})\Big)\Big]\\
+\Big(\partial_{p'}(q^{1/2}(\eta_c))+\partial_{q'}(q^{1/2}(\eta_c))+\frac{\partial\xi_{2,c}}{\partial\xi_1}(\partial_{p'}(q^{1/2}(\eta_c))-\partial_{q'}(q^{1/2}(\eta_c))\Big)\frac{\partial_{\xi_1}g_{N,a}(\xi_1,\xi_{2,c})}{q^{1/2}(\eta_c)},\\
\end{gather*}
where the last line vanishes when $\partial_{\xi_1}g_{N,a}(\xi_1,\xi_{2,c})=0$. In the same way we compute
\[
\partial^3_{\xi_1,\xi_1,\xi_1}(g_{N,a}(\xi_1,\xi_{2,c}))|_{\partial_{\xi_1}(g_{N,a}(\xi_1,\xi_{2,c}))=\partial^2_{\xi_1,\xi_1}(g_{N,a}(\xi_1,\xi_{2,c}))=0}=q^{1/2}(\eta_c)\Big(4\Big(1-\frac{1}{N^2}\Big)+\frac{3\nu^2}{8\xi_1^4}+O(a)\Big)\,.
\]
Let first $|N|\geq 2$, then we immediately see that the third order derivative takes positive values and stays bounded from below by a fixed constant, $\partial^3_{\xi_1,\xi_1,\xi_1}(g_{N,a}(\xi_1,\xi_{2,c}))\geq 2$, and therefore the critical points may be degenerate (when $\partial^2_{\xi_1,\xi_1}(g_{N,a}(\xi_1,\xi_{2,c}))=0$) of order at most $2$. Let now $|N|=1$ when the coefficient of $2\xi_1^2$ in \eqref{gfirstderiv} is $O(a)$. Assume that for $c\lesssim |\xi_1|$ the first two derivative vanish, then $\frac{\nu^2}{8\xi_1^3}= 4K_a+O(a)$ and therefore the third derivative cannot vanish since its main contribution is $\frac{3\nu^2}{8\xi_1^4}$.
\end{proof}

\paragraph{\bf Case $|\xi_1|\lesssim c$, for small $0<c<1/2$} First, (usual) stationary phase applies in $\xi_1$:
\[
\partial_{\xi_1}g_{N,a}(\xi_1,\xi_{2})%
=q^{1/2}(\eta_c)\Big((\xi_1+\xi_{2})^2+N^2(X-\alpha_c)+(\xi_1-\xi_{2})^2+N^2(1-\alpha_c)\Big)\,,
\]
and using \eqref{eq:alpha_c21}, we write again, with $K_a=K_a(\frac{Y}{4N},\frac{T}{2N})=\frac{T}{2N}q^{1/2}(\eta^0_c)$,
\[
\sqrt{\alpha_c}=K_a-\frac{(\sigma+s)}{2N}+\frac{T}{2N}(q^{1/2}(\eta_c)-q^{1/2}(\eta^0_c)),
\]
where in the new variables $\sigma+s=2\xi_1/N$. Using \eqref{etacvsseta0c}, we have $(q^{1/2}(\eta_c)-q^{1/2}(\eta^0_c))=\frac{a}{NT}O(\xi_1,\xi_2)$ and with $|\xi_1|\leq c<\frac 12$ small, $a\leq\varepsilon_0$ and $\alpha_c\in [\frac 12,\frac 32]$, from $K_a=\sqrt{\alpha_c}+O(c/N^2)$ we have $K_a\in [1/4,2]$ for all $N\geq 1$. The derivative of $g_{N,a}(\xi_1,\xi_2)$ becomes
\begin{align}\label{part1derivg}
\partial_{\xi_1}g_{N,a}(\xi_1,\xi_{2}) & =q^{1/2}(\eta_c)\Big\{2\xi_1^2+2\xi_{2}^2-\mu-2N^2\Big[\Big(K_a-\frac{\xi_1}{N^2}+\frac{a}{N^2}O(\xi_1,\xi_2)\Big)^2-K^2_a\Big]\Big\}\\
 & =q^{1/2}(\eta_c)\Big(2\xi_1^2(1-\frac{1}{N^2})+2\xi_{2}^2-\mu+4K_a\xi_1+aO(\xi_1,\xi_2)\Big).
\end{align}
At the critical point, the second derivative with respect to $\xi_1$ is 
\begin{equation}
\partial^2_{\xi_1,\xi_1}g_{N,a}(\xi_1,\xi_2)|_{\partial_{\xi_1}g_{N,a}(\xi_1,\xi_2)=0}=q^{1/2}(\eta_c)\Big(4\xi_1(1-\frac{1}{N^2})+4K_a+O(a)\Big),
\end{equation}
and as $K_a\in [\frac 14,2]$, the leading order term is $4q^{1/2}(\eta_c)K_a$. Stationary phase applies for any $|N|\geq 1$ and provides a factor $\Lambda^{-1/2}$. 
We are left with the integral with respect to $\xi_2$.  We first compute the critical point $\xi_{1,c}$, solution to $\partial_{\xi_1}g_{N,a}(\xi_1,\xi_{2})=0$, as a function of $\xi_2$: 
\begin{equation}\label{xi1c}
2\xi_{1,c}^{2}+2\xi_2^2=\mu+2N^2\Big[K_a^2-\Big(K_a^2-\frac{\xi_1}{N^2}+\frac{T}{2N}(q^{1/2}(\eta_c)-q^{1/2}(\eta^0_c))\Big)^2|_{\xi_1,
\xi_2}\Big],
\end{equation}
where, using \eqref{etacvsseta0c}, $\frac{T}{2N}(q^{1/2}(\eta_c)-q^{1/2}(\eta^0_c))=O(\frac{a}{N^2})$. From $|\xi_{1,c}|\leq c$,  $|\mu/2-\xi_2^2|\lesssim c$, as, if $|\mu/2-\xi_2^2|> 4c$, the equation \eqref{xi1c} has no real solution $\xi_{1,c}$ such that $|\xi_{1,c}\leq c$.
\begin{lemma}
For all $|N|\geq 1$ and $|\mu/2-\xi_2^2|\leq 4c$, \eqref{xi1c} has one real valued solution, 
\begin{equation}\label{eq:formxi1cxi2}
\xi_{1,c}=(\mu/2-\xi^2_2)\Xi_0
+a\Big((\mu/2-\xi_2^2)\Xi_1+\xi_2^2\Xi_2+\xi_2 \frac{\nu}{N^2}\Xi_3\Big),
\end{equation}
where $K_{a=0}=\frac{|Y|}{4N}q^{1/2}(-Y/|Y|)$ and $\Xi_0=\Xi_0(\mu/2-\xi_2^2, K_{a=0},1/N^2)$ is defined as
\begin{equation}\label{defXi0}
\Xi_0(\mu/2-\xi_2^2, K_{a=0},1/N^2)=\Bigl(K_{a=0}+\sqrt{K_{a=0}^2+(\mu/2-\xi_2^2)(1-1/N^2)}\Bigr)^{-1}
\end{equation}
and where $\Xi_{1,2,3}$ are a smooth functions of $(\xi_2,\mu/2-\xi_2^2,\nu/N^2,K_a,1/N,a)$ such that $|\partial^k_{\xi_2}\Xi_j|\leq C_k$, for all $k\geq 0$, where $C_k$ are positive constants.
\end{lemma}
\begin{proof}
For $a=0$, \eqref{xi1c} has an unique, explicit solution $\xi_{1,c}|_{a=0}$,
\[
\xi_{1,c}|_{a=0}=(\mu/2-\xi_2^2)\Bigr({K_{a=0}+\sqrt{K_{a=0}^2+(\mu/2-\xi_2^2)(1-1/N^2)}}\Bigl)^{-1}\,,
\]
that we rename $(\mu/2-\xi^2_2)\Xi_0$ with $\Xi_0$ defined in \eqref{defXi0}. Let now $a\neq 0$. Using Lemma \ref{lemalphacetac} with $s+\sigma=(p'+q')/N=2\xi_1/N$, $\sigma-s=(p'-q')/N=2\xi_2/N$, $(a-x)/a=\nu/N^2$, the critical point $\eta_c$ is a function of 
$\xi_1/N$, $\xi_2^2/N^2$ and $\xi_2\nu/N^3$.
Write $\xi_{1,c}$ as $\xi_{1,c}=(\mu/2-\xi^2_2)\Xi_0+a\Xi$ for some unknown function $\Xi$; introducing this in \eqref{xi1c} allows to obtain $\Xi$ as a sum of smooth functions with factors $\mu/2-\xi_2^2$, $\xi_2^2$ and $\xi_2 \nu/N^2$ as follows $ \Xi=(\mu/2-\xi_2^2)\Xi_1+\xi_2^2\Xi_2+\xi_2\frac{\nu}{N^2}\Xi_3$,
where $\Xi_j$ are smooth functions of $\mu/2-\xi_2^2,\xi_2^2/N^2, \xi_2\nu/N^3$.  
\end{proof}
Let $\tilde g_{N,a}(\xi_2):=g_{N,a}(\xi_{1,c},\xi_2)$ : the first derivative of $\tilde g_{N,a}$ with respect to $\xi_2$ vanishes when $(\partial_{p'}G_{N,a}-\partial_{q'}G_{N,a})(p',q')|_{(\xi_{1,c},\xi_2)}=0$ which is equivalent to $4\xi_{1,c}\xi_2=\nu$. We compute, using $\partial_{\xi_2}\tilde g_{N,a}=\nu-4\xi_{1,c}\xi_2$ and $\xi_{1,c}$ given in \eqref{eq:formxi1cxi2}, $\partial^2_{\xi_2\xi_2}\tilde g_{N,a} =-4(\xi_2\partial_{\xi_2}\xi_{1,c}+\xi_{1,c})$. Then, critical points $\xi_{2}$ are degenerate if
\begin{multline}\label{seccritxi2}
(\mu/2-\xi_2^2)\Xi_0+a\Big((\mu/2-\xi_2^2)\Xi_1+\xi_2^2\Xi_2+\xi_2 \frac{\nu}{N^2}\Xi_3\Big)
=2\xi_2^2\Xi_0 \Big(1-(\mu/2-\xi_2^2)\tilde \Xi_0\Big)\\
+a\Big(2\xi_2^2(\Xi_1-\Xi_2-\frac 12 \xi_2\partial_{\xi_2}\Xi_2-\frac{\nu}{N^2}\partial_{\xi_2}\Xi_3)-\xi_2(\mu/2-\xi_2^2)\partial_{\xi_2}\Xi_1-\xi_2\frac{\nu}{N^2}\Xi_3\Big),
\end{multline}
where the term past equality in the first line of \eqref{seccritxi2} is $\xi_2\partial_{\xi_2}\Xi_0$. We have thus set
\[
\tilde \Xi_0(\mu/2-\xi_2^2,K_a,1/N^2):=\frac{(1-1/N^2)\Xi_0(\mu/2-\xi_2^2,K_a,1/N^2)}{2\sqrt{K_{a}^2+(\mu/2-\xi_2^2)(1-1/N^2)}}.
\]
Consider $a=0$ in \eqref{seccritxi2} for a moment, then critical points are degenerate if 
\begin{equation}
  \label{eq:4}
  \mu/2-\xi_2^2=2\xi_2^2\Big(1-(\mu/2-\xi_2^2)\tilde \Xi_0(\mu/2-\xi_2^2,K_0,1/N^2)\Big).
\end{equation}
Recall that $K_a\in [1/4,2]$ and that $|\mu/2-\xi_2^2|\leq 4c$ with $c$ small enough. Rewrite \eqref{eq:4}
\[
(\mu/2-\xi_2^2)\Big(2+\frac{1}{1-(\mu/2-\xi_2^2)\tilde\Xi_0}\Big)=\mu
\]
which may have solutions only if $\mu$ is also small enough, $|\mu|\leq 10c$. Let $z=\mu/2-\xi_2^2$; for $|z|\leq 4c$ and $|\mu|\leq 10c$ with $c$ small enough, we may now seek the solution to \eqref{eq:4}   as $z=\mu Z_0(\mu,K_0,1/N^2)$ and obtain $Z_0(\mu,K_0,1/N^2)$ explicitly, with $Z_0(0,K_0,1/N^2)=\frac 13$. Solutions to \eqref{seccritxi2} for $a=0$ are therefore functions of $\sqrt{\mu}$ which both vanish at $\mu=0$. They write $\xi_{2,\pm}|_{a=0}=\pm\frac{\sqrt{\mu}}{\sqrt{6}}\Big(1+\mu\zeta(\mu,K_0,1/N^2)\Big)$, for some smooth function $\zeta$.

Let now $a\neq 0$: solutions $\xi_2$ to \eqref{seccritxi2} are functions of $\sqrt{\mu}, \nu/N^2, a$ that coincide at $\mu=\nu=0$ (they both vanish.) Actually, as $\Xi_1$ is a function of $\mu/2-\xi_2^2,\xi_2^2, \xi_2 \nu/N^2$, $\xi_2\partial_{\xi_2}\Xi_1$ is also a function of $\mu/2-\xi_2^2,\xi_2^2,\xi_2\nu/N^2$ and we write %
\begin{multline}\label{eqpointdeg}
\mu/2-\xi_2^2=2\xi_2^2(1-(\mu/2-\xi_2^2)\tilde\Xi_0(\mu/2-\xi_2^2,K_a,1/N^2))\\
+a\Big(\xi_2^2F_1(\xi_2^2,\xi_2 \nu/N^2,\mu)+\xi_2 \frac{\nu}{N^2}F_2(\xi_2^2,\xi_2\nu/N^2,\mu)\Big),
\end{multline}
for some smooth functions $F_{1,2}$. %
Notice that, as $|\mu/2-\xi_2^2|\leq 4c$ and $a$ is small, \eqref{eqpointdeg} may have real solutions $\xi_2$ only for $|\xi_2^2|\leq 4c$.
For such small $\xi_2$, equation \eqref{eqpointdeg} has at most two distinct solutions (that coincide at $\mu=\nu=0$) which read
\begin{equation}\label{solxi2pm}
\xi_{2,\pm}=\pm\frac{\sqrt{\mu}}{\sqrt{6}}(1+\mu\zeta(\mu,K_a,1/N^2))+a <(\sqrt{\mu},\frac{\nu}{N^2}),(\zeta_{1,\pm},\zeta_{2,\pm})>(\sqrt{\mu},\frac{\nu}{N^2},K_a,a),
\end{equation}
for some smooth functions $\zeta,\zeta_{j,\pm}$.
We compute the third derivative of $\tilde g_{N,a}$ at $\xi_{2,\pm}$ defined in \eqref{solxi2pm} whenever the second derivative vanishes. Using \eqref{seccritxi2} yields
\begin{multline}\label{thirdderivxi2}
\partial^3_{\xi_2,\xi_2,\xi_2}\tilde g_{N,a}(\xi_{1,c},\xi_2)|_{\xi_2=\xi_{2,\pm}}=-4(2\partial_{\xi_2}\xi_{1,c}+\xi_2\partial^2_{\xi_2,\xi_2}\xi_{1,c})|_{\xi_{2,\pm}}\\
=16\xi_{2}\Xi_0\Big(1-(\mu/2-\xi_2^2)\tilde\Xi_0(\mu/2-\xi_2^2,K_a,1/N^2)\Big)\\
+8a\Big(2\xi_2(\Xi_1-\Xi_2-\frac 12 \xi_2\partial_{\xi_2}\Xi_2-\frac{\nu}{N^2}\partial_{\xi_2}\Xi_3)-(\mu/2-\xi_2^2)\partial_{\xi_2}\Xi_1-\frac{\nu}{N^2}\Xi_3\Big)|_{\xi_{2,\pm}}\\
+8\xi_{2}\Xi_0(1+O(\mu/2-\xi_2^2)+O(a))|_{\xi_{2,\pm}},
\end{multline}
where the last line in \eqref{thirdderivxi2} is $-4\xi_{2,\pm}\partial^2_{\xi_2,\xi_2}\xi_{1,c}$: we do not expand this formula as $\xi_{2,\pm}$ is sufficiently small for what we need. The second and third lines of \eqref{thirdderivxi2} come from the formula for $-8\partial_{\xi_2}\xi_{1,c}$, already obtained in \eqref{seccritxi2} (where $\partial_{\xi_2}\xi_{1,c}$ comes with a factor $\xi_2$.) As the third derivative of $\tilde g_{N,a}$ is evaluated at $\xi_{2,\pm}$ we can replace \eqref{eqpointdeg} in \eqref{thirdderivxi2} and obtain
 \[
 \partial^3_{\xi_2,\xi_2,\xi_2}\tilde g_{N,a}(\xi_{1,c},\xi_2)|_{\xi_2=\xi_{2,\pm}}=\frac{12\xi_{2,\pm}}{K_a}(1+O(\xi_{2,\pm}^2)+O(a))+O(a\nu/N^2).
 \]
It follows that at $\mu=\nu=0$ the order of degeneracy is higher as $\xi_{2,\pm}|_{\mu=\nu=0}=0$ and $\partial^3_{\xi_2,\xi_2,\xi_3}\tilde g_{N,a}|_{\xi_{2,\pm},\mu=\nu=0}=0$. We now write 
\begin{multline*}
\tilde g_{N,a}(\xi_2)=\tilde g_{N,a}(\xi_{2,\pm})+(\xi_2-\xi_{2,\pm})\partial_{\xi_2}\tilde g_{N,a}(\xi_{2,\pm})+\frac{(\xi_2-\xi_{2,\pm})^3}{6}\partial^3_{\xi_2,\xi_2,\xi_2}\tilde g_{N,a}(\xi_{2,\pm})+O((\xi_2-\xi_{2,\pm})^4),
\end{multline*}
where $\partial^{4}_{\xi_{2}^{4}}\tilde g_{N,a}$  does not cancel at $\xi_{2,\pm}$ as it stays close to $12/K_a\in [6,48]$. We are to have $\partial_{\xi_2}\tilde g_{N,a}(\xi_{2,\pm})=0$, from which $\nu=4\xi_{1,c}|_{\xi_{2,\pm}}\xi_{2,\pm}$, which writes
\begin{multline}
\nu=4\Big(\pm \frac{\sqrt{\mu}}{\sqrt{6}}(1+\mu\zeta(\mu))+a( \sqrt{\mu}\zeta_{1,\pm}+\frac{\nu}{N^2}\zeta_{2,\pm})\Big)\\
\times \Big((\mu/2-\xi_{2,\pm}^2)\Xi_0+a((\mu/2-\xi_{2,\pm}^2)\Xi_1+\xi_{2,\pm}^2\Xi_2+\xi_{2,\pm}\frac{\nu}{N^2}\Xi_3)\Big)
\end{multline}
and replacing \eqref{solxi2pm} in \eqref{eq:formxi1cxi2} yields $\nu=\pm \frac{\sqrt{2} \mu^{3/2}}{3\sqrt{3}K_a}(1+O(a))$, which is at leading order the equation of a cusp.
At degenerate critical points $\xi_{2,\pm}$ where $\nu=\pm \frac{\sqrt{2} \mu^{3/2}}{3\sqrt{3}K_a}(1+O(a))$, the phase integral behaves like
$
I=\int_{\xi_{2}} \rho(\xi_{2}) e^{\mp i \Lambda \frac{\sqrt{2}\sqrt{\mu}} {K_a\sqrt 3} (\xi_{2}-\xi_{2,\pm})^{3}}\,d\xi_{2}\,,
$
and we may conclude in a small neighborhood of the set $\{\xi_2^2+|\mu|+|\nu|^{2/3}\lesssim c\}$ (as outside this set, the non-stationary phase applies) by using Van der Corput lemma on the remaining oscillatory integral in $\xi_2$ with phase $\tilde g_{N,a}(\xi_2)$. In fact, on this set, $\partial^4_{\xi_2}\tilde g_{N,a}$ is bounded from below, which yields an upper bound $\Lambda^{-1/4}$, uniformly in all parameters. When $\mu\neq 0$, the third order derivative of the phase is bounded from below by $\frac{|\xi_2|}{K_a}$ : either $|\mu/6-\xi_2^2|\leq |\mu|/12$ and then $|\partial^3_{\xi_2}\tilde g_{N,a}|$ is bounded from below by $|\mu|^{1/2}/(12K_a)$ or
$|\mu/6-\xi_2^2|\geq |\mu|/12$ in which case $|\partial^2_{\xi_2}\tilde g_{N,a}|$ is bounded from below by $|\mu|/(12K_a)$. Hence, using that $K_a\in [1/4,2]$, we find $|\partial^3_{\xi_2}\tilde g_{N,a}|+|\partial^3_{\xi_2}\tilde g_{N,a}|\gtrsim \sqrt{|\mu|}$ (recall that here $\mu$ is small so $\sqrt{|\mu|}\geq |\mu|$) which yields an upper bound $(\sqrt{|\mu|}\Lambda)^{-1/3}$. Eventually we obtain
$ |I|\lesssim \inf\Big\{\frac{1}{\Lambda^{1/4}},\frac{1}{|\mu|^{1/6} \Lambda ^{1/3}}\Big\}$.
From $\mu=A'+B'$ and $\nu=A'-B'\sim \pm |\mu|^{3/2}$ and $|\mu|^{3/2}\ll \mu$ for $\mu<1$, we deduce that $A'\sim B'$ and therefore $\mu \sim 2 B'$, which is our desired bound \eqref{eq:2hh} after unraveling all notations, as the non degenerate stationary phase in $\xi_{1}$ provided a factor $\Lambda^{-1/2}$. 

\section{The defocusing cubic nonlinear Schr\"odinger equation}
\label{sec:defoc-cubic-nonl}
We now turn to nonlinear applications and consider \eqref{schrod} with $\kappa=1$. One could state local in time results for the focusing case $k=-1$ and extend them to global in time provided we require a standard smallness condition on the mass of the data (identifying the threshold is, however, a delicate issue.) Our underlying manifold $\Omega$ is compact and such that, in local charts intersecting its boundary, the Laplace-Beltrami operator is, at first order, our model operator. One may easily construct such 3D manifolds: we simply illustrate what can be done in 2D where one can better visualize the corresponding manifold as embedded in $\mathbb{R}^{3}$. If one periodizes the $y$ variable in $\Omega_{2}$, we may see it as the surface of an upper cylinder $x>0$ of radius one in $\R^{3}$, where $y$ is really an angle in the 2D plane $\{x=0\}$. This surface may be truncated at $x=1$ and we may extend it smoothly with a (compact) cap to get a Riemannian manifold, say with the induced euclidean metric within a subset of the cap, the metric from our model in the bottom part of the upper cylinder and a smooth transition in between. One could alternatively connect two copies of our truncated upper cylinder, or connect one with another one where the operator is chosen to be $\partial_{x}^{2}+(1+x)^{-1}\partial_{y}^{2}$ (so as to have a convex boundary on one side and a concave one on the other). Those last two examples may also be seen as subdomains of an 2D torus (either slice circularly in the middle or slice horizontally). For such a suitable manifold $\Omega$, with $d\geq 3$, we may gather available results and state homogeneous Strichartz estimates. In the next theorem, Sobolev spaces should be understood as defined through the spectral resolution of the Dirichlet Laplacian and they are known to match the ones defined by classical interpolation.
\begin{thm}\label{thmStriSob}
Let $d\geq 2$, $(q,r)$ such that $\frac  1 q\leq \Big(\frac d 2 -\frac 14\Big)\Big(\frac 1 2 -\frac 1 r\Big)$, $s=\frac d2-\frac 2 q-\frac dr$; there exist $C(d)>0$ $T>0$ such that, for $v$ solution to \eqref{schrod} with data $v_0\in H^{s+1/q}(\Omega)$,
  \begin{equation}
    \label{eq:Hs}
    \|v\|_{L^q([0,T],L^r(\Omega))}\leq C(d) T^{1/q} \|v_0\|_{H^{s+1/q}(\Omega)}\,.
  \end{equation}
\end{thm}
Both Lebesgue and Sobolev spaces may be localized in coordinate charts. As such, the proof of Theorem \ref{thmStriSob} follows by standard arguments, localizing the linear solution to each patch, and using either \cite{bgt04} (on patches that do not intersect the boundary), Theorem \ref{thmStrichartz} (on patches that do intersect the convex boundary) or \cite{doi08} (if one considers a concave boundary on one end). Because one works on semiclassical times, source terms that are produced by cut-offs are dealt with through energy estimates, a well-established procedure that goes back to \cite{stta02}. We refer to \cite{bss12} or \cite{IPihp} for detailed implementations of the above strategy. Notice that \eqref{eq:Hs} subsumes the underlying semiclassical estimate: assume $v_{0}$ is spectrally localized at frequency $h^{-1}$, and $T\lesssim h$, then $\|v_{0}\|_{H^{s+1/q}(\Omega)}\sim h^{-1/q} \|v_{0}\|_{H^{s}(\Omega)}$. We also need a suitable semiclassical version of the double endpoint inhomogeneous Strichartz estimate which follows from similar arguments: let
\begin{equation}\label{schrodinhomo}
-i\partial_t u+\Delta_g u=f, \quad u_{|t=0}=0,\quad  u_{|\mathbb{R}\times\partial\Omega}=0\,,
\end{equation}
with $f$ supported in a time interval $I$ such that $|I|\lesssim h$, $I\subset \mathbb{R}_{+}$. Then we have
\begin{prop}\label{inStri}
  Let $d\geq 3$, $r$ such that $\frac  1 2 = \Big(\frac d 2 -\frac 14\Big)\Big(\frac 1 2 -\frac 1 r\Big)$, $\sigma=d/r-d/(2d/(d-2))$, $r'=r/(r-1)$, there exists $C(d)>0$ such that, for $u$ solution to \eqref{schrodinhomo},
  \begin{equation}
    \label{eq:5}
     \|\psi(h^{2}\Delta_{g}) u\|_{L^2(I,L^r(\Omega))}\leq C(d) h^{-2\sigma} \|\psi(h^{2}\Delta_{g}) f\|_{L^{2}(I,L^{r'}(\Omega))}\,.
  \end{equation}
\end{prop}
This should be understood as a weaker version of Lemma 3.4 in \cite{bgt04}: the loss of regularity $\sigma$ corresponds to a Sobolev embedding from $L^{2d/(d-2)}$ (the endpoint Strichartz exponent on a boundaryless manifold) to $L^{r}$ for a spectrally localized function. The value of $r$ will be irrelevant in the forthcoming argument: any $r<+\infty$ would do as long as $\sigma$ is chosen accordingly (preserving scale invariance).
 \begin{rmq}
With $g=dx^2+(1+x)^{-1}dy^2$, the Laplace-Beltrami operator writes
$\triangle_{g}=(1+x)^{1/2}\partial_{x}(1+x)^{-1/2}\partial_{x}+(1+x)\Delta_{y}$. In our model, we use instead $\Delta_{F}=\partial_{x}^{2}+(1+x)\Delta^{2}_{y}$, as $\Delta_{F}$ allows for explicit computations. The difference $\Delta_{g}-\Delta_{F}= -(2(1+x))^{-1}\partial_{x}$ is a first order differential operator: as such, on a semiclassical time scale in a neighborhood of the boundary, it may be treated as a lower order perturbative term; proving semiclassical Strichartz estimates for $\Delta_{F}$ implies the same set of estimates for $\Delta_{g}$.
\end{rmq}
We are interested in $d=3$: for $q=2$ we have $r=10$, $s+1/q=7/10<1$. The crucial point is that $r<+\infty$: in \cite{bss12}, for $q=2$, $r=+\infty$, $s+1/q=1$ and one ends up with two successive logarithmic losses that force to consider only lower order nonlinearities. In \cite{FP}, Strichartz estimates are bypassed and replaced by bilinear estimates: logarithmic losses turn out to be more manageable in that setting and one can implement a suitable version of the Brezis-Gallouet argument, obtaining global existence of solutions to the cubic nonlinear Schr\"odinger equation but only for $H^{s}(\Omega)$ data, with $s>1$.
\subsection{Global well-posedness in the energy space}
\label{sec:glob-wellp-ness}
Recall we are interested in $-i\partial_{t} v + \Delta v = |v|^{2}v$ on $\Omega$, with $v_{0}\in H^{1}_{0}(\Omega)$. Standard compactness methods provide a weak solution $v\in L^{\infty}_{t}(H^{1}_{0}(\Omega))$, and, with $E(v)=\| \nabla v\|^{2}_{2}/2+\|v\|^{4}_{4}/4$, we have $E(v)\leq E(v_{0})$. Moreover, using Duhamel's formula, one may prove that $v\in C_{t}(L^{2}(\Omega))$. Thus we aim at proving that $v$ is unique and that time continuity holds in $H^{1}_{0}(\Omega)$. To this end,  \cite{bgt04} implements in a clever way an argument of Yudovitch, using the inhomogeneous endpoint Strichartz estimate at the semiclassical level. We follow them closely, having only to check that the specific value $r=6$ that they start with is irrelevant to the crux of the matter. Assume we are on $(0,T)$, with $T<1$ and $\Delta_{j}$ is the usual spectral localization operator associated to a Littlewood-Paley decomposition: $h=2^{-j}$, $\Delta_{j}=\psi(h^{2}\Delta_{g})$, $\sum_{j\geq -1}\Delta_{j}=Id$ where $\Delta_{-1}=\phi(\Delta_{g})$ with $\phi$ supported in $B(0,2)$. After time localization on intervals of size $h$, applying their inhomogeneous endpoint Strichartz estimate and summing over $O(h^{-1})$ time intervals lead \cite{bgt04} (on a boundary less manifold) to
$$
\|\Delta_{j} v \|_{L^{2}_{T} L^{6}} \lesssim \|\Delta_{j} v(0)\|_{L^{2}}+\|\Delta_{j} v(T)\|_{L^{2}}+2^{-j/2} \|\Delta_{j} v \|_{L^{2}_{T} H^{1}} + \|\Delta_{j} (|v|^{2}v) \|_{L^{2}_{T} L^{6/5}}\,.
$$
Reproducing their argument but using Proposition \ref{inStri} with the $(2,10)$ Strichartz pair leads to
\begin{equation}
  \label{tameH1}
   2^{-j/5}\|\Delta_{j} v \|_{L^{2}_{T} L^{10}} \lesssim \|\Delta_{j} v(0)\|_{L^{2}}+\|\Delta_{j} v(T)\|_{L^{2}}+2^{-j/2} \|\Delta_{j} v \|_{L^{2}_{T} H^{1}} + 2^{j/5} \|\Delta_{j} (|v|^{2}v ) \|_{L^{2}_{T} L^{10/9}}\,,
\end{equation}
(observe that, as noted earlier, this is essentially the same estimate in terms of scaling.) We now use Sobolev embedding, $H^{1}\hookrightarrow L^{5}$ and product laws to bound the last term,
 $$
2^{j/5} \|\Delta_{j} (|v|^{2} v)\|_{L^{2}_{T} L^{10/9}}\lesssim  \sqrt T 2^{-4j/5} (\sup_{t} \|v\|_{H^{1}})^{3}\,.
$$
We use Sobolev but on the left hand side, with a large $p$ (recall $T\leq 1$)
$$
2^{-3j(1/10-1/p)-j/5}\|\Delta_{j} v \|_{L^{2}_{T} L^{p}} \lesssim 2^{-j}\sup_{t} \|v\|_{H^{1}}+2^{-j/2} \|\Delta_{j} v \|_{L^{2}_{T} H^{1}} + 2^{-4j/5}  (\sup_{t} \|v\|_{H^{1}})^{3}\,,
$$
to get
\begin{equation}
  \label{eq:derdesder}
  \|\Delta_{j} v \|_{L^{2}_{T} L^{p}} \lesssim 2^{-3j/p}\| \Delta_{j} v \|_{L^{2}_{T} H^{1}} + 2^{-3j/p-2j/5}  \sup_{t} (\|v\|_{H^{1}}+\|v\|^{3}_{H^{1}})\,.
\end{equation}
Summing over $j$, applying Cauchy-Schwartz in $j$, we finally get the same bound as in \cite{bgt04}
$$
\| v \|_{L^{2}_{T} L^{p}} \lesssim  \sqrt{p T}  +1\,.
$$
From there we may proceed similarly and conclude to uniqueness of weak solutions by estimating the difference between two solutions in $L^{2}$, the above estimate and an elementary differential inequality. 

Once we have uniqueness, the inequality for $E(v)$ is immediately upgraded to conservation, and continuity in $H^{1}$ follows as the potential part is itself continuous by interpolation between $L^{2}$ and $L^{6}$.

This achieves the proof of Theorem \ref{gwp} in the defocusing case. The focusing case may be handled in a similar way, up to a smallness condition if one wants a global result to preserve coercivity of $E(v)$ and we therefore skip it.

For preservation of higher Sobolev regularity, one may proceed exactly as in \cite{bgt04}, using the Brezis-Gallouet argument. One should notice, however, that the resulting bound on Sobolev norms is a double exponential (to be compared to the triple exponential from \cite{FP}). 
\subsection{Exponential growth for higher Sobolev norms $H^{m}$, $m>1$}
The double exponential growth for higher Sobolev norms was reduced to a single exponential, for solutions on a generic compact, boundaryless manifold, in \cite{PTV}, using modified energy methods. We now check that the elegant treatment of modified energies in \cite{PTV} is not spoiled if we replace the endpoint Strichartz estimate from \cite{bgt04} by our endpoint estimate. Bounds on the $H^{2k}$ norms follow from computing the time derivative of
\begin{equation*}
  {\mathcal E}_{2k}(v)=
\|\partial_t^k v\|_{L^2(M^d)}^2-\frac{1}2 \int_{M^d} |\partial^{k-1}_{t} \nabla_g (|v|^2)|_g^2 - \int_{M^d}|\partial^{k-1}_{t} (|v|^{2} v)|^2 
\end{equation*}
where all norms over $M$ are understood with respect to the metric volume form. One gets 
\begin{multline}\label{energy3}
\frac d{dt} {\mathcal E}_{2k}(v)=2  \int_{M^d} \partial^{k}_{t} (|v|^{2})\partial^{k-1}_{t}
(|\nabla_g v|_g^2)
+\Re \sum_{j=0}^{k-2}
c_j   \int_{M^d} \partial^{k}_{t} (|v|^{2}) \partial_t^j (\Delta_g v) \partial_t^{k-1-j} \bar v 
\\
{}+ \Re \sum_{j=0}^{k-1} c_j  \int_{M^d}  
\partial^{j}_{t} (|v|^{2})  \partial_t^{k-j} v \partial_{t}^{k-1 }(|v|^{2}\bar v) 
 +\Im
\sum_{j=1}^{k-1} 
 c_j \int_{M^d} \partial_t^j (|v|^{2}) \partial_t^{k-j} v \partial^{k}_{t}\bar v \,,
\end{multline}
where $c_{j}$ are harmless numerical constants. We may perform the exact same computation on our manifold $\Omega$: there are no boundary terms due to the Dirichlet boundary condition.

We will address the $H^{2}$ norm: higher order norms can be dealt with similarly (for the top order term, the remaining terms being lower order). The modified energy bound in \cite{PTV} proceeds with
\begin{equation}
  \label{modE2}
\frac{d}{dt} \mathcal{E}_{ 2}(v) \lesssim \int_{M} |\nabla ^{2}v| |\nabla v|^{2} |v |\lesssim \|v\|_{H^{2}}\|\nabla v\|^{2}_{L^{6}}\|v\|_{L^{6}}\,,
\end{equation}
using the equation to eliminate all time derivatives and ignoring lower order terms; one then uses a suitable version of Proposition \ref{inStri}, involving the endpoint Strichartz estimate, for the gradient term, which is therefore bounded by (the square of) an $H^{3/2}$ with an (important) additional power of time; by interpolation one recovers another $H^{2}$ norm and this leads to exponential growth for the $H^{2}$ norm by standard arguments.

 We will modify how we distribute norms on the integral over $M$ in \eqref{modE2}: one should heuristically think we aim at placing each $\nabla v$ in $L^{4}_{t,x}$ (hence controlled by $H^{3/2}$) and $v\in L^{\infty}_{t,x}$. We however need to slightly perturb this choice in order to avoid a log loss: with a large $p$, we have, by our Strichartz estimate \eqref{eq:derdesder} (at the scaling level of $H^{1}$ data) and Sobolev embedding (starting from $L^{6}_{x}$)
$$
\Delta_{j} v \in 2^{3j/p} L^{2}_{t} L^{p}\,,\,\,\,\text{and}\,\,\,\Delta_{j} v \in 2^{-3(1/6-1/p)j} L^{\infty}_{t} L^{p}\,.
$$
hence we get, balancing regularities with $(1-e)\frac 3 p=e (\frac 1 2 - \frac 3 p)$ ($e=6/p$), $v \in L_{t}^{\frac 2 {1-6/p}} L^{p}$ in term of the energy $E(v)$.

Then, we will get $|\nabla v|^{2 }|v|\in  L^{1}_{t} L^{2}$, with $v\in L_{t}^{\frac 2 {1-6/p}} L^{p}$, provided one may estimate both factors $\nabla v \in L^{r}_{t} L^{q}$, where $(r,q)$ is close to $(4,4)$ and such that
$$
1=\frac 1 2 (1-\frac 6 p)+\frac 2 r\,\,\,\text{ and }\,\,\, \frac 1 2 = \frac 1 p +\frac 2 q\,\,\,( \,\Rightarrow \frac 1r+\frac 3 q=1 \,)\,.
$$
We now prove such an estimate on $\nabla v$: start over with \eqref{tameH1} but shift regularity on $v$ from $H^{1}$ to $H^{3/2}$:
\begin{equation}
  \label{tameH32}
   2^{-j/5+j/2}\|\Delta_{j} v \|_{L^{2}_{T} L^{10}} \lesssim 2^{j/2}(\|\Delta_{j} v(0)\|_{L^{2}}+\|\Delta_{j} v(T)\|_{L^{2}})+2^{-j/2} \|\nabla \Delta_{j} v \|_{L^{2}_{T} H^{1/2}} + 2^{j/5+j/2} \|\Delta_{j} (|v|^{2}v ) \|_{L^{2}_{T} L^{10/9}}\,,
\end{equation}
multiply by $2^{j/2}$ and estimate the nonlinear term placing one factor in $H^{3/2}$ and others in $H^{1}\hookrightarrow L^{5}$,
\begin{equation}
  2^{-j/5}\|\nabla \Delta_{j} v \|_{L^{2}_{T} L^{10}}  \lesssim  \|\nabla \Delta_{j} v(0)\|_{L^{2}}+\|\nabla \Delta_{j}v(T)\|_{L^{2}} +\|\nabla \Delta_{j}v \|_{L^{2}_{T} H^{1/2}} + 2^{- 3j/10}\sqrt T\sup_{t} (\|v\|_{H^{3/2}}\|v\|_{H^{1}}^{2})
\end{equation}
from which we easily obtain
\begin{equation}
  \label{eq:6}
   \sup_{j} (2^{-j/5}\|\nabla \Delta_{j} v \|_{L^{2}_{T} L^{10}}  )  \lesssim  E(v)^{1/2}+\sqrt T (1+ E(v)) \sup_{t} \|v\|_{H^{3/2}}\,.
  \end{equation}
Then, pick $Q(>q)$ such that $1/q=(2/r) (1/Q)+(1-2/r) (1/2)$ and use Sobolev embedding to go from $L^{10}$ to $L^{Q}$
\begin{equation}
  \label{eq:6fin}
   \sup_{j} (2^{-j/5-3j(1/10-1/Q)}\|\nabla \Delta_{j} v \|_{L^{2}_{T} L^{Q}}  )  \lesssim  E(v)^{1/2}+\sqrt T (1+ E(v)) \sup_{t} \|v\|_{H^{3/2}}\,.
\end{equation}
On the other hand, from $\nabla v\in L^{\infty}_{T}H^{1/2}$ we have  $\sup_{j} (2^{j/2}\|\nabla \Delta_{j} v\|_{L^{\infty}_{T}L^{2}}) \lesssim \sup_{0\leq t\leq T} \|v\|_{H^{3/2}}$: chosing $\theta=2/r>1/2$, the regularities cancel as scaling dictates: from $3/q+1/r=1$ we check that indeed
$$
(\frac 1 2 -\frac 3{Q})\frac 2 r  =\frac 1 2 (1-\frac 2 r) \,.
$$
We therefore get the desired $L^{r}L^{q}$ estimate for $\nabla v$:
$$
\|\nabla v \|_{L^{r}_{T} L^{q}} \lesssim (E(v)^{1/2}+\sqrt T (1+E(v))\sup_{0\leq t\leq T}\|v \|_{H^{3/2}} )^{\theta} (\sup_{0\leq t\leq T} \|v\|_{H^{3/2}}^{1-\theta})\,.
$$
Then, there exists $C(E)$ (which may change from line to line) such that
$$
\|\nabla v \|_{L^{r}_{T} L^{q}} \lesssim C(E(v)) \big(\sup_{0\leq t\leq T}\|v\|_{H^{3/2}}^{1-\theta}+ T^{1/4} \sup_{0\leq t\leq T}\|v \|_{H^{3/2}}\big)\,.
$$
Gathering all our estimates, we get that, for $0<T<1$
\begin{align}
  \| v(\cdot,T)\|^{2}_{H^{2}}-\|v(\cdot,0)\|^{2}_{H^{2}} & \leq \int_{0}^{T} \int_{M} |\nabla^{2}v||\partial v|^{2}| | v |\\
     & \lesssim \|v\|_{L^{\infty}_{T}H^{2}}\|v |\nabla v|^{2}\|_{L^{1}_{T}L^{2}}\\
   & \lesssim \|v\|_{L^{\infty}_{T}H^{2}}\|\nabla v\|^{2}_{L^{r}_{T}L^{q}}\|v\|_{L^{2/(1-6/p)}_{T}L^{p}}\\
 &    \lesssim C(E) \big(\sup_{0\leq t\leq T}\|v\|_{H^{2}}^{2-\theta}+ T^{1/2} \sup_{0\leq t \leq T}\|v \|^{2}_{H^{2}}\big)
\end{align}
from which exponential growth follows as in \cite{PTV}. This concludes the proof of Theorem \ref{growthHs}.

\end{document}